\PassOptionsToPackage{top=2.5cm, bottom=2cm, left=2cm, right=2cm}{geometry}
\documentclass[final,3p,times, leqno]{elsarticle}

\makeatletter
\def\ps@pprintTitle{%
   \let\@oddhead\@empty
   \let\@evenhead\@empty
   \def\@oddfoot{\reset@font\hfil\thepage\hfil}
   \let\@evenfoot\@oddfoot
}
\makeatother

\journal{}

\usepackage[x11names]{xcolor}
\usepackage[
  colorlinks,
]{hyperref}

\newcommand\rurl[1]{%
  \href{http://#1}{\nolinkurl{#1}}%
}

\AtBeginDocument{\hypersetup{citecolor=magenta, urlcolor=magenta, linkcolor=blue}}

\usepackage{amsthm, amssymb, amsfonts, mathrsfs, amsmath}
\usepackage{lmodern}
\usepackage[T5]{fontenc}
\usepackage{amscd}
\usepackage{graphicx}

\usepackage{pb-diagram}

\usepackage{listings}

\theoremstyle{plain}
\newtheorem{thm}{Theorem}[section]
\newtheorem{conj}[thm]{Conjecture}

\newtheorem{corl}[thm]{Corollary}
\theoremstyle{definition}
\newtheorem{defn}[thm]{Definition}
\newtheorem{rem}[thm]{Remark}

\theoremstyle{plain}
\newtheorem{therm}{Theorem}[subsection]
\newtheorem{propo}[therm]{Proposition}
\newtheorem{lema}[therm]{Lemma}

\theoremstyle{definition}

\newtheorem{rems}[therm]{Remark}

\newtheorem*{acknow}{Acknowledgments}

\def\DD{D\kern-.7em\raise0.4ex\hbox{\char '55}\kern.33em}

\usepackage{sectsty}
\subsectionfont{\fontsize{11.5}{11.5}\selectfont}

\begin{document}
\fontsize{11.5pt}{11.5}\selectfont

\begin{frontmatter}

\title{On the dimension of $H^{*}((\mathbb Z_2)^{\times t}, \mathbb Z_2)$ as a module over Steenrod ring}

\author{\DD\d{\u a}ng V\~o Ph\'uc}
\address{{\fontsize{10pt}{10}\selectfont Faculty of Education Studies, University of Khanh Hoa,\\ 01 Nguyen Chanh, Nha Trang, Khanh Hoa, Viet Nam}}
\ead{dangphuc150488@gmail.com}

\begin{abstract}

We write $\mathbb P$ for the polynomial algebra in one variable over the finite field $\mathbb Z_2$ and $\mathbb P^{\otimes t} = \mathbb Z_2[x_1, \ldots, x_t]$ for its $t$-fold tensor product with itself. We grade $\mathbb P^{\otimes t}$ by assigning degree $1$ to each generator.  We are interested in determining a minimal set of generators for the ring of invariants $(\mathbb P^{\otimes t})^{G_t}$ as a module over Steenrod ring, $\mathscr A_2.$ Here $G_t$ is a subgroup of the general linear group $GL(t, \mathbb Z_2).$ An equivalent problem is to find a monomial basis of the space of "unhit" elements, $\mathbb Z_2\otimes_{\mathscr A_2} (\mathbb P^{\otimes t})^{G_t}$ in each $t$ and degree $n\geq 0.$ The structure of this tensor product is proved surprisingly difficult and has been not yet known for $t\geq 5,$ even for the trivial subgroup $G_t = \{e\}.$ In the present paper, we consider the subgroup $G_t = \{e\}$ for $t \in \{5, 6\},$ and obtain some new results on $\mathscr A_2$-generators of $(\mathbb P^{\otimes t})^{G_t}$ in some degrees. At the same time, some of their applications have been proposed. We also provide an algorithm in MAGMA for verifying the results. This study can be understood as a continuation of our recent works in \cite{D.P2, D.P4}.
\end{abstract}

\begin{keyword}

Steenrod algebra, Peterson hit problem, Adams spectral sequences, Primary cohomology operations, Actions of groups on commutative rings, Algebraic transfer




\MSC[2010] 13A50, 55S10, 55S05, 55T15, 55R12, 55Q45
\end{keyword}
\end{frontmatter}

\tableofcontents

\section{Introduction}\label{s1}

This serves as the motivation and introduction of the article. Throughout the context, we fix the prime field of two elements, $\mathbb Z_2.$  The study of the mod-2 Steenrod algebra $\mathscr A_2$ began with Steenrod's work \cite{N.E} constructing stable cohomology operations acting on cohomology theory with $\mathbb Z_2$-coefficients. This algebra consists of the Steenrod squaring operations $Sq^{i}$ for $i\geq 0,$ which act on the $\mathbb Z_2$-cohomology groups of spaces. These squares are the cohomology operations satisfying the naturality property. Moreover, they commute with the suspension maps and so they are stable. Afterwards, the structure of $\mathscr A_2$ was elucidated by Adem \cite{Adem}, Cartan \cite{Cartan} and Serre \cite{J.S}. It plays a crucial role in the solution of many problems, such as calculating homotopy groups of $n$-sphere, Hopf invariant one problem, and characteristic classes of vector bundles in Algebraic Topology. We know that $\mathbb P^{\otimes t} = \mathbb Z_2[x_1, \ldots, x_t] = H^{*}((\mathbb Z_2)^{\times t}, \mathbb Z_2)$ with $|x_j| = 1$ for every $j$, and this representation of $\mathbb P^{\otimes t}$ defines it as an $\mathscr A_2$-module. The polynomial ring $\mathbb P^{\otimes t}$ is a connected $\mathbb Z$-graded algebra.  One of central problems of Algebraic Topology is to determine a minimal set of generators for the ring of invariants $(\mathbb P^{\otimes t})^{G_t},$ where $G_t\subseteq GL_t:=GL(t, \mathbb Z_2),$ the general linear group of invertible matrices. This $(\mathbb P^{\otimes t})^{G_t}$ is a graded $\mathscr A_2$-submodule of the left $\mathscr A_2$-module $\mathbb P^{\otimes t}$.  Equivalently, we need to find the dimension of the following $\mathbb Z$-graded vector space over $\mathbb Z_2$:
\begin{equation}\label{kgvt}
\{(\mathbb Z_2\otimes_{\mathscr A_2} (\mathbb P^{\otimes t})^{G_t})_n\}_{n\in\mathbb Z,\, n\geq 0}.
\end{equation}
This problem is generally regarded as difficult for arbitrary number $t$ of variables and it has been of interest to topologist (see H\uhorn ng-Peterson \cite{H.P}, Kameko \cite{M.K}, Mothebe-Uys \cite{M.M}, Nam \cite{T.N}, Pengelley-Williams \cite{P.W}, Peterson \cite{F.P1, F.P2}, the present writer \cite{P.S1, D.P1, D.P2, D.P3, D.P4, D.P6}, Singer \cite{W.S1, W.S2}, Sum \cite{N.S, N.S1, N.S2, N.S3}, Walker-Wood \cite{W.W} and others). Alternatively,  its applications to homotopy theory are surveyed by Peterson \cite{F.P2} and Singer \cite{W.S1}. For $G_t = \Sigma_t$ (the symmetric group on $t$ letters), the space \eqref{kgvt} is known by Singer's work \cite{W.S2} for the dual of \eqref{kgvt} with $t\geq 0.$ In the case in which $G_t = GL_t,$ the problem is studied by Singer \cite{W.S1} for $t = 2,$ and by H\uhorn ng-Peterson \cite{H.P} for $t = 3,\, 4.$ Up to present, it remains undetermined in general. In the case in which $G_t = \{e\},$ the trivial subgroup, one has
\begin{equation}\label{kgvt2}
\{Q^{\otimes t}_n:=(\mathbb Z_2\otimes_{\mathscr A_2} \mathbb P^{\otimes t})_n\}_{n\in\mathbb Z,\, n\geq 0},
\end{equation}
as a modular representation of $GL_t.$ Studying this space is related to the $E_2$-page of the Adams spectral sequence for the stable homotopy groups of spheres, which is mentioned below. The structure of the space \eqref{kgvt2} was explicitly determined by Peterson \cite{F.P1} for $t = 1,\, 2,$ by Kameko \cite{M.K} for $t = 3,$ by Sum \cite{N.S1} for $t = 4.$ The case $t = 5$ has been investigated by the author \cite{P.S1, D.P1, D.P2, D.P3, D.P4, D.P6}, and Sum \cite{N.S2} in some certain degrees. But when $t\geq 6,$ the answer is not yet known. The primary interest of this work is the cases $t = 5$ and $6,$ as well as their applications. There exist a few approaches to reduce the computation of the dimension of the space \eqref{kgvt2}. The squaring operation $\widetilde {Sq^0_*}: Q^{\otimes t}_{t+2n}  \longrightarrow Q^{\otimes t}_n$, which was constructed by Kameko \cite{M.K}, could be one of those useful approaches. One should note that this $\widetilde {Sq^0_*}$ is an epimorphism of $\mathbb Z_2GL_t$-modules. The $\mu$-function is also one of the arithmetic functions that have much been used in this context (see also \cite{M.K}, \cite{N.S2}). The range of it is the minimum number of terms of the form $2^{d}-1,$ $d > 0,$ with repetitions allowed, whose sum is $n,$ from which due to Wood \cite{R.W},  $Q^{\otimes t}_n$ is trivial if $\mu(n) > t,$ and due to Kameko \cite{M.K}, the Kameko $\widetilde {Sq^0_*}$ is an isomorphism if $\mu(t + 2n) = t.$ In general, it is not easy to compute or even estimate the dimension of $Q^{\otimes t}_n$ in all degrees $n.$ However, as the results of Wood \cite{R.W} and Kameko \cite{M.K} mentioned, the dimension of the space \eqref{kgvt2} can be completely determined in each $n$ of the so-called "generic degree":
\begin{equation}\label{ct}
n = r(2^d-1) + m.2^d,
\end{equation}
whenever $0\leq \mu(m) < r < t,$ and $d \geq 0.$ As above shown, the problem we are interested in is systematically studied when $t\leq 4,$ however it remains a mystery in the general case. Moreover, there is no general rule for studying the problem in each $t$ and degree $n$ of the form \eqref{ct} above, that means each computation is important on its own. Therefore, our motivation to write up this article is to probe the space \eqref{kgvt2} and its applications for the cases $t = 5$ and $6$. Before presenting the main results, quoted below in detail, we need some background material for completeness of the exposition and for the reader's convenience.  In that aspect,  the next four chief concepts can be found in literatures \cite{M.K}, \cite{D.P2}, \cite{N.S1}, \cite{W.W}.


\begin{defn}[{\bf Weight vector and exponent vector}] A weight vector $\omega$ is a sequence of non-negative integers $\omega = (\omega_1, \omega_2, \ldots, \omega_j,\ldots)$ such that $\omega_j  = 0$ for $j\gg 0.$ Let $X = x_1^{u_1}x_2^{u_2}\ldots x_t^{u_t}\in \mathbb P_n^{\otimes t},$ then one defines the {\it weight vector} and the \textit{exponent vector} of $X$ by
$$ \omega(X) :=(\omega_1(X)=\sum_{1\leq j\leq t}\alpha_{0}(u_j),\ \omega_2(X)=\sum_{1\leq j\leq t}\alpha_{1}(u_j),\ldots,\ \omega_m(X)=\sum_{1\leq j\leq t}\alpha_{m-1}(u_j),\ldots)$$ and $u(X):=(u_1, u_2, \ldots, u_t),$ respectively, where $\alpha_j(n)\in \{0, 1\}$ denotes the $j$-th coefficients in dyadic expansion of a natural number $n.$ One also defines $\deg(\omega(X)) = \sum_{j\geq 1}2^{j-1}\omega_j(X)$ and by convention, the sets of all weight vectors and exponent vectors are given the left lexicographical order.
\end{defn}

From now on, let us write $\mathbb P_n^{\otimes t}$ for the subspace of $\mathbb P^{\otimes t}$ generated by the homogeneous polynomials of degree $n$ in $\mathbb P^{\otimes t}.$ Then $\mathbb P^{\otimes t} = \{\mathbb P_n^{\otimes t}\}_{n\in\mathbb Z,\, n\geq 0},$ in which $\mathbb P_0^{\otimes t}\cong \mathbb Z_2.$  It is known that by restricting to non-singular matrices, $\mathbb P_n^{\otimes t}$ also gives a modular representation of $GL_t,$ and the Steenrod squaring operations $Sq^{i}: \mathbb P^{\otimes t}_{n-i}\to \mathbb P^{\otimes t}_{n}$ are maps of $\mathbb Z_2GL_t$-modules.

\begin{defn}[{\bf Linear order on $\mathbb P^{\otimes t}$}]
Assume that $X = x_1^{u_1}x_2^{u_2}\ldots x_t^{u_t}$ and $Y = x_1^{v_1}x_2^{v_2}\ldots x_t^{v_t}$ are the monomials in $\mathbb P_n^{\otimes t}.$ We say that $X  < Y$ if and only if one of the following holds:
\begin{enumerate}
\item[(i)] $\omega(X) < \omega(Y);$
\item[(ii)] $\omega(X) = \omega(Y)$ and $u(X) < v(Y).$
\end{enumerate}
\end{defn}

\begin{defn}[{\bf Equivalence relations on $\mathbb P^{\otimes t}$}]  For a weight vector $\omega$ of degree $n,$ we denote two subspaces associated with $\omega$ by
$$ \begin{array}{ll}
\medskip
 (\mathbb P_n^{\otimes t})^{\leq \omega} &= \langle\{ X\in \mathbb P^{\otimes t}_n|\, \deg(X) = \deg(\omega) = n,\  \omega(X)\leq \omega\}\rangle,\\
(\mathbb P_n^{\otimes t})^{< \omega} &= \langle \{ X\in \mathbb P^{\otimes t}_n|\, \deg(X) = \deg(\omega) = n,\  \omega(X) < \omega\}\rangle.
\end{array}$$ 
Given homogeneous polynomials $F$ and $G$ in $\mathbb P^{\otimes t}_n,$ one defines the equivalence relations "$\equiv$" and "$\equiv_{\omega}$" on $\mathbb P^{\otimes t}$ by stating that 
\begin{enumerate}
\item [(i)] $F \equiv G $ if and only if $(F - G)\in \overline{\mathscr {A}_2}\mathbb P_n^{\otimes t};$  
\item[(ii)] $F \equiv_{\omega} G$ if and only if $F, \, G\in (\mathbb P_n^{\otimes t})^{\leq \omega}$ and $(F -G)\in ((\overline{\mathscr {A}_2}\mathbb P_n^{\otimes t} \cap (\mathbb P_n^{\otimes t})^{\leq \omega} + (\mathbb P_n^{\otimes t})^{< \omega}),$  $\overline{\mathscr {A}_2}\mathbb P_n^{\otimes t} = \sum_{i\geq 0}{\rm Im}(Sq^{2^{i}})$ denoting the set of all the $\mathscr A_2$-decomposable elements, in which $Sq^{2^{i}}: \mathbb P^{\otimes t}_{n-2^{i}}\to \mathbb P^{\otimes t}_{n}$ and $\overline{\mathscr {A}_2}$ is the kernel of an epimorphism $\mathscr A_2\to \mathbb Z_2$ of graded $\mathbb Z_2$-algebras.  In particular, if $F$ is a linear combination of elements in the images of the Steenrod squaring operations $Sq^{2^{i}}$ for $i\geq 0$, then we say that $F$ is \textit{decomposable} (or \textit{hit}). In other words, $F\equiv 0.$ In addition, $F$ is said to be \textit{$\omega$-decomposable} if $F \equiv_{\omega} 0.$ Then, we have a $\mathbb Z_2$-quotient space of $Q^{\otimes t}$ by the equivalence relation "$\equiv_\omega$": $(Q_n^{\otimes t})^{\omega} = (\mathbb P_n^{\otimes t})^{\leq \omega}/ ((\overline{\mathscr {A}_2}\mathbb P_n^{\otimes t} \cap (\mathbb P_n^{\otimes t})^{\leq \omega})+ (\mathbb P_n^{\otimes t})^{< \omega}).$
\end{enumerate}
\end{defn}
Due to Sum \cite{N.S3}, $(Q_n^{\otimes t})^{\omega}$ is also an $\mathbb Z_2GL_t$-module for every $n.$ The above definitions lead to the following that will be played a central role in the content of this text.

\begin{defn}[{\bf Admissible monomial and inadmissible monomial}]\label{dninadm} We say that a monomial $X\in \mathbb P^{\otimes t}_n$ is {\it inadmissible} if there exist monomials $Y_1, Y_2,\ldots, Y_k$ such that $Y_j < X$ for $1\leq j\leq k$ and $X \equiv  \sum_{1\leq j\leq k}Y_j.$ Then, $X$ is said to be {\it admissible} if it is not inadmissible.
\end{defn}

It is important to remark that the set of all the admissible monomials in $\mathbb P^{\otimes t}_n$ is \textit{a minimal set of $\mathscr {A}_2$-generators for $\mathbb P^{\otimes t}$ in degree $n$.} Therefore, $Q_n^{\otimes t}$ is an $\mathbb Z_2$-vector space with a basis consisting of all the classes represent by the elements in $\mathbb P^{\otimes t}_n.$ Specifically, by our previous discussions \cite{D.P2}, $\dim Q_n^{\otimes t}$ can be represented as $\sum_{\deg(\omega) = n}\dim (Q_n^{\otimes t})^{\omega}.$ We fix some notations to be used through the paper. Writing $(\mathbb P^{\otimes t})^{0}$ and $(\mathbb P^{\otimes t})^{>0}$ for the subspaces of $\mathbb P^{\otimes t}$ spanned all the monomials $x_1^{a_1}x_2^{a_2}\ldots x_t^{a_t}$ such that $\prod_{1\leq j\leq t}a_j = 0,$ and $\prod_{1\leq j\leq t}a_j > 0,$ respectively. If $x_1^{a_1}x_2^{a_2}\ldots x_t^{a_t}\in \mathbb P_n^{\otimes t},$ then we rewrite $(\mathbb P^{\otimes t})^{0}$ and $(\mathbb P^{\otimes t})^{>0}$ as $(\mathbb P_n^{\otimes t})^{0}$ and $(\mathbb P_n^{\otimes t})^{>0},$ respectively. We put \mbox{$(Q_n^{\otimes t})^0:= (\mathbb Z_2\otimes_{\mathscr A_2} (\mathbb P^{\otimes t})^{0})_n$}, and  $(Q_n^{\otimes t})^{>0}:= (\mathbb Z_2\otimes_{\mathscr A_2} (\mathbb P^{\otimes t})^{>0})_n,$ from which one has that $Q_n^{\otimes t} = (Q_n^{\otimes t})^0\,\bigoplus\, (Q_n^{\otimes t})^{>0}.$

We are now able to give the main results of this work. Our first goal is to investigate the space \eqref{kgvt2} for the case $t  =5$ and the degree of the form \eqref{ct} when $r = t = 5, m = 13$ and $d = 1.$ The case $d > 1$ will be discussed at the end of this section. For a start we have seen that due to Kameko's homomorphism, the map $ \widetilde {Sq^0_*}:  Q^{\otimes 5}_{5(2^{1} - 1) + 13.2^{1}} \to Q^{\otimes 5}_{5(2^{0} - 1) + 13.2^{0}}$ is epimorphism of $\mathbb Z_2GL_5$-modules. On the other side, in our previous work, we have indicated that
\begin{thm}[see \cite{D.P2}]\label{dlcb}
Let $n = 3(2^d-1) + 2^d$ be the generic degree of form \eqref{ct} with $d$ a positive integer. The dimension of the $\mathbb Z_2$-vector space $Q^{\otimes 5}_n$ is determined by the following table:

\centerline{\begin{tabular}{c|cccccc}
$n = 3(2^d-1) + 2^d$  &$d=1$ & $d=2$ & $d=3$ & $d=4$ & $d  =5$ & $d\geqslant 6$\cr
\hline
\ $\dim Q^{\otimes 5}_n$ & $46$ & $250$ & $645$ &$945$ &$1115$ & $1116$ \cr
\end{tabular}}
\end{thm}
So, by the theorem, we only need to determine the kernel of $\widetilde {Sq^0_*}.$ For convenience, let us denote $(Q_n^{\otimes t})^{\omega^{0}}:= (Q_n^{\otimes t})^{\omega}\cap (Q_n^{\otimes t})^{0},\ \mbox{and}\ (Q_n^{\otimes t})^{\omega^{>0}} := (Q_n^{\otimes t})^{\omega}\cap (Q_n^{\otimes t})^{>0}.$ One of our main results is then as follows.

\begin{thm}\label{dlc1}
Consider the weight vectors of degree $31 = 5(2^{1} - 1) + 13.2^{1}$: $$\omega_{(1)}:=(1,1,1,1,1),\ \omega_{(2)}:=(3,2,2,2), \ \omega_{(3)}:= (3,4,3,1).$$ Then, we have an isomorphism ${\rm Ker}(\widetilde {Sq^0_*})\cong  (Q_{31}^{\otimes 5})^{0} \bigoplus \big(\bigoplus_{1\leq j\leq 3} (Q_{31}^{\otimes 5})^{\omega_{(j)}^{>0}}\big).$
\end{thm}

One can easily compute $(Q_{31}^{\otimes 5})^{0}$ by using the known results of $Q_{31}^{\otimes 4}.$ Nevertheless, for the convenience of the reader, it is detailed as follows. The reader already knows (see also \cite{D.P4}) that for each non-negative integer $n,$ there is an isomorphism $(Q_{n}^{\otimes 5})^0  \cong \bigoplus_{1\leq s\leq 4}\bigoplus_{\ell(\mathcal J) = s}(Q_{n}^{\otimes\, \mathcal J})^{>0},$ where $Q^{\otimes \mathcal J} = \langle [x_{j_1}^{t_1}x_{j_2}^{t_2}\ldots x_{j_s}^{t_s}]\;|\; t_i \in \mathbb N,\, i = 1,2, \ldots, s\}\rangle \subset Q^{\otimes 5}$ in which $\mathcal J= (j_1, j_2, \ldots, j_s),$  $1 \leq j_1 < \ldots < j_s \leq 5$,\, $1 \leq s \leq 4,$ and $\ell(\mathcal J):=s$ denotes the length of $\mathcal J.$ So, $\dim (Q_{n}^{\otimes 5})^0 = \sum_{1\leq  s\leq 4}\binom{5}{s}\dim (Q_{n}^{\otimes\, s})^{>0}.$ Moreover, by the previous works of Peterson \cite{F.P1}, Kameko \cite{M.K} and Sum \cite{N.S1}, it follows that
$$ \dim (Q_{31}^{\otimes s})^{>0} = \left\{\begin{array}{ll}
1,& \mbox{if } s = 1,\\
1,& \mbox{if } s = 2,\\
8,& \mbox{if } s = 3,\\
47,& \mbox{if } s = 4.
\end{array}\right.$$
We thus have $\dim (Q^{\otimes 5}_{31})^{0}= 1\times\binom{5}{1} + 1\times\binom{5}{2} + 8\times\binom{5}{3} + 47 \times \binom{5}{4}) = 330.$ Next, the dimension results of $(Q_{31}^{\otimes 5})^{\omega_{(j)}^{>0}}$ are determined by the following theorem.

\begin{thm}\label{dlc2}
Let $\omega_{(j)}, j = 1, 2, 3,$ be the weight vectors as in Theorem \ref{dlc1}. Then, with notations as above, we assert that
$$ \dim (Q_{31}^{\otimes 5})^{\omega_{(j)}^{>0}} = \left\{\begin{array}{ll}
1,& \mbox{if } j = 1,\\
215,& \mbox{if } j = 2,\\
70,& \mbox{if } j = 3.
\end{array}\right.$$
\end{thm}

In order to compute the subspaces $(Q_{31}^{\otimes 5})^{\omega_{(j)}^{>0}},$ we adopt some homomorphisms in Sum \cite{N.S1} and a recent result of Mothebe-Uys \cite{M.M}. This approach allows us to reduce many computations for a monomial basis of the space \eqref{kgvt2} in each certain positive degree. This is shown specifically in this paper by proving Propositions \ref{md2}, \ref{md3} in section two and Theorem \ref{dlc3} below. As consequences of the results above, we immediately obtain

\begin{corl}\label{hq31}
There exist exactly $866$ admissible monomials of degree $31$ in $\mathbb P^{\otimes 5}.$ Consequently, the $\mathbb Z_2$-vector space $Q_{31}^{\otimes 5}$ is $866$-dimensional. 
\end{corl}

We continue by considering the case $r = t = 5, m = 32$ and $d = 0.$ So, our second goal is to establish the following.

\begin{thm}\label{dlc3}
We have an isomorphism $ Q^{\otimes 5}_{32} \cong  (Q_{32}^{\otimes 5})^{\omega'_{(1)}}\bigoplus (Q_{32}^{\otimes 5})^{\omega'_{(2)}}\bigoplus (Q_{32}^{\otimes 5})^{\omega'_{(3)}},$ where $\omega'_{(1)}:=(2,1,1,1,1),\ \omega'_{(2)}:=(4,2,2,2),$ and $\omega'_{(3)}:= (4,4,3,1)$ are  the weight vectors of degree $32.$ Consequently, $Q^{\otimes 5}_{32}$ has dimension $1004.$  
\end{thm}

From Theorems \ref{dlc1} and \ref{dlc3}, we remark that the dimensions of $Q_{31}^{\otimes 5}$ and $Q_{32}^{\otimes 5}$ are not only large but also complicated in the calculations. So, we have provided an algorithm in MAGMA for verifying these results (see Appendix). Moreover, this algorithm can be used to verify the results of Peterson \cite{F.P1}, Kameko \cite{M.K} and Sum \cite{N.S1} on the dimension of the space \eqref{kgvt2} for the cases $t\leq 4.$ Before presenting the next main result, we briefly review Sum's conjecture \cite{N.S1}, which  provides a connection between admissible monomials in $\mathbb P_n^{\otimes (t-1)}$ and $\mathbb P_n^{\otimes t}.$ For each $1\leq l\leq t,$ define the $\mathbb Z_2$-homomorphism $\mathsf{q}_{(l,\,t)}: \mathbb P^{\otimes (t-1)}\to  \mathbb P^{\otimes t}$  by setting that $\mathsf{q}_{(l,\,t)}(x_j)$ is $x_j$ if $1\leq j \leq l-1,$ and is $x_{j+1}$ if $l\leq j \leq t-1.$ We put $ \mathcal N_t := \{(l, \mathscr L)\;|\; \mathscr L = (l_1,l_2,\ldots, l_r), 1\leq l < l_1< l_2 < \ldots < l_r\leq t, \ 0\leq r \leq t-1\},$ where by convention, $\mathscr L = \emptyset$ if $r = 0.$ Writing $r = \ell(\mathscr L)$ for the length of $\mathscr L.$ Let us consider the pair $(l, \mathscr L)\in\mathcal{N}_t,\ 1 \leq r \leq t-1$ and for each $1\leq u < r,$ we set $ X_{(\mathscr L,\,u)} = x_{l_u}^{2^{r-1} + 2^{r-2} +\, \cdots\, + 2^{r-u}}\prod_{u < d\leq r}x_{l_d}^{2^{r-d}}$ and $X_{(\emptyset, 1)} = 1.$ Then, one has the following important linear transformation, which is due to Sum  \cite{N.S1}:
$$ \begin{array}{ll}
\psi_{(l, \mathscr L)}: \mathbb P^{\otimes (t-1)}&\longrightarrow \mathbb P^{\otimes t}\\
 \prod_{1\leq s\leq t-1}x_s^{a_s} &\longmapsto  \left\{ \begin{array}{ll}
(x_l^{2^{r} - 1}\mathsf{q}_{(l,\,t)}(\prod_{1\leq s\leq t-1}x_s^{a_s}))/X_{(\mathscr L,\,u)}&\text{if there exists $u$ such that:} \\
&a_{l_1 - 1} +1= \ldots = a_{l_{(u-1)} - 1} +1 = 2^{r},\\
& a_{l_{u} - 1} + 1 > 2^{r},\\
&\alpha_{r-d}(a_{l_{u} - 1}) -1 = 0,\, 1\leq d\leq u, \\
&\alpha_{r-d}(a_{l_{d}-1}) -1 = 0,\, \ u+1 \leq d \leq r,\\
0&\text{otherwise}.
\end{array} \right.
\end{array}$$
A few comments below are necessary and useful.

\begin{rem}
If  $\mathscr L = \emptyset,$ then $\psi_{(l, \mathscr L)} = \mathsf{q}_{(l,\,t)}$ for all $1\leq l\leq t.$ Also, one can easily check that if $\psi_{(l, \mathscr L)}( \prod_{1\leq s\leq t-1}x_s^{a_s})\neq  0,$ then $\omega(\psi_{(l, \mathscr L)}( \prod_{1\leq s\leq t-1}x_s^{a_s})) = \omega( \prod_{1\leq s\leq t-1}x_s^{a_s}).$ Note that, in general, $\psi_{(l, \mathscr L)}$ is not a homomorphism of $\mathscr A_2$-modules. An illustrative instance is necessary for this: Taking $t = 4,$  $\mathscr L = (2,3,4)\neq \emptyset,$ and the monomial $x_1^{12}x_2^{6}x_3^{9}\in \mathbb P^{\otimes 3}_{27},$ then, it is straightforward to see that $$\psi_{(1, \mathscr L)}(x_1^{12}x_2^{6}x_3^{9}) = \dfrac{x_1^{2^{3}-1}\mathsf{q}_{(1, 4)}(x_1^{12}x_2^{6}x_3^{9})}{X_{(\mathscr L,\,1)}} = \dfrac{x_1^{7}x_2^{12}x_3^{6}x_4^{9}}{x_2^{4}x_3^{2}x_4} = x_1^{7}x_2^{8}x_3^{4}x_4^{8}\in \mathbb P^{\otimes 4}_{27},$$ and $\omega(x_1^{12}x_2^{6}x_3^{9}) =  (1,1,2,2) =\omega(\psi_{(1, \mathscr L)}(x_1^{12}x_2^{6}x_3^{9})).$ Hence, one gets $$Sq^{2}(\psi_{(1, \mathscr L)}(x_1^{12}x_2^{6}x_3^{9})) = x_1^{9}x_2^{8}x_3^{4}x_4^{8}\neq x_1^{7}x_2^{8}x_3^{6}x_4^{8} = \psi_{(1, \mathscr L)}(Sq^{2}(x_1^{12}x_2^{6}x_3^{9})).$$
\end{rem}

Now, for a subset $\mathscr V\subset \mathbb P_n^{\otimes (t-1)},$ let us write
$$ \begin{array}{ll}
\widetilde {\Phi^0}(\mathscr  V) &= \bigcup_{1\leq l \leq t}\psi_{(l, \emptyset)}(\mathscr  V) =  \bigcup_{1\leq l \leq t}\mathsf{q}_{(l, t)}(\mathscr V),\\
 \widetilde{\Phi^{>0}}(\mathscr  V) &= \bigcup_{(l; \mathscr L)\in\mathcal{N}_t,\;1 \leq \ell(\mathscr L) \leq t-1}(\psi_{(l, \mathscr L)}(\mathscr V)\setminus (\mathbb P_n^{\otimes t})^{0}),
\end{array}$$
and put $\widetilde{\Phi_*}(\mathscr  V) = \widetilde {\Phi^0}(\mathscr  V) \bigcup \widetilde {\Phi^{>0}}(\mathscr V).$ It is straightforward to check that $\mathsf{q}_{(l,\,t)}$ is also an $\mathscr A_2$-homomorphism, which implies that if $\mathscr V$ is the set of all the admissible monomials in $\mathbb P_n^{\otimes (t-1)}$, then $\widetilde{\Phi}^0(\mathscr V) $ is the set of all the admissible monomials in $(\mathbb P_n^{\otimes t})^{0}.$ Let us now denote by $\mathscr {C}^{\otimes t}_n$ the set of all admissible monomials in $\mathbb P_n^{\otimes t}.$ If $\omega$ is a weight vector of degree $n,$ then we set$(\mathscr{C}^{\otimes t}_{n})^{\omega} := \mathscr {C}^{\otimes t}_n\cap (\mathbb P_n^{\otimes t})^{\leq \omega}.$ In \cite{N.S2}, Sum sets up the following yet left-open. 

\begin{conj}\label{gtSum}
Under the above notations, for each $(l, \mathscr L)\in \mathcal N_t,$ and $1\leq r = \ell(\mathscr L)\leq t-1,$ if $X = \prod_{1\leq j\leq t-1}x_j^{a_j}\in \mathscr C^{\otimes (t-1)}_n,$ and there exist $u,\ 1\leq u\leq r$ such that
\begin{equation}\label{dk}
 \begin{array}{ll}
a_{l_1 - 1} +1= \ldots = a_{l_{(u-1)} - 1} +1 = 2^{r},\ \ a_{l_{u} - 1} + 1 > 2^{r},\\
\alpha_{r-d}(a_{l_{u} - 1}) -1 = 0,\ \mbox{for}\ 1\leq d\leq u, \  \mbox{and}\ \alpha_{r-d}(a_{l_{d}-1}) -1 = 0,\ \mbox{for}\ u+1 \leq d \leq r,
\end{array}
\end{equation}
 then $ \psi_{(l, \mathscr L)}(X) \in \mathscr C^{\otimes t}_n.$ Moreover, if  $\omega$ is a weight vector of degree $n$, then $\widetilde{\Phi_*}((\mathscr{C}^{\otimes (t-1)}_{n})^{\omega})\subseteq  (\mathscr{C}^{\otimes t}_{n})^{\omega}.$
\end{conj}

Basing the proofs of Theorems \ref{dlc2} and \ref{dlc3}, we have seen that Conjecture \ref{gtSum} also holds when $t = 5$ and the degrees $31,\, 32$ (see Remarks \ref{nx} and \ref{nx2} in Sect.\ref{s2}).

We now describe the next main result by studying the tensor product $Q^{\otimes 6}$ in generic degree $5(2^{d+4}-1) + 47.2^{d+4}.$ To do this, we use our previous results in \cite{D.P4, D.P6} and get the following.

\begin{thm}\label{dlc4}
Consider degree $n_d$ of the form \eqref{ct} with $r = t = 5$ and $m = 21.$ Then, we have that $\dim Q^{\otimes 6}_{5(2^{d+4}-1) + n_1.2^{d+4}} = (2^{6}-1)\dim Q^{\otimes 5}_{n_d}.$ Consequently, $Q^{\otimes 6}_{5(2^{d+4}-1) + n_1.2^{d+4}} $ is $119322$-dimensional, for all positive integers $d.$
\end{thm}

{\bf An application to Singer's algebraic transfer.} As is well known, computing explicitly mod-2 cohomology groups of the Steenrod algebra, ${\rm Ext}_{\mathscr {A}_2}^{t,t+*}(\mathbb Z_2,\mathbb Z_2)$ is a major open problem of Algebraic Topology, because it is the $E_{2}$-page of the Adams spectral sequence converging to the stable homotopy groups of the spheres. One approach to better understand the structure of these Ext groups was proposed by Singer \cite{W.S1} where he introduced a transfer homomorphism of algebras
$$ Tr^{\mathscr A_2}:= \bigoplus_{t\geq 0} Tr^{\mathscr A_2}_t: \bigoplus_{t\geq 0} [P_{\mathscr A_2}((\mathbb P^{\otimes t})^{*})]_{GL_t}\to \bigoplus_{t\geq 0} {\rm Ext}_{\mathscr {A}_2}^{t,t+*}(\mathbb Z_2,\mathbb Z_2) = \bigoplus_{t\geq 0} H^{t, t+*}(\mathscr A_2, \mathbb Z_2)$$
from the subspace of the divided power algebra $(\mathbb P^{\otimes t})^{*} = H_*((\mathbb Z/2)^{\times t}, \mathbb Z/2)$  consisting of all elements that are $\overline{\mathscr A_2}$-annihilated by all positive degree Steenrod squaring operations to the cohomology of the Steenrod algebra. It should be noted that the action of the algebra $\mathscr A_2$ and the action of the group $GL_t$ on $(\mathbb P^{\otimes t})^{*}$ commute, and so there exists an induced action of $GL_t$ on $P_{\mathscr A_2}((\mathbb P^{\otimes t})^{*}).$ As it was demonstrated in  the works by Singer \cite{W.S1} for $t \leq 2$ and Boardman \cite{J.B} for $t = 3$ that $Tr_t^{\mathscr A_2}$ is an isomorphism. So, the algebraic transfer is highly nontrivial. Additionally, by the work of Minami \cite{Minami},  Singer's transfer is related to the problem of finding permanent cycles in the classical Adams spectral sequence. It may need to be recalled that the Hopf elements are represented by the elements $h_i\in {\rm Ext}_{\mathscr {A}_2}^{1,2^{i}}(\mathbb Z_2,\mathbb Z_2)$ on the $E_2$-page of the Adams spectral sequence. By Adam's solution of the Hopf invariant one problem \cite{Adams}, only $h_i,$ for $0\leq i\leq 3,$ survive to the $E_{\infty}$-page. By Browder's work \cite{Browder}, the Kervaire classes $\theta_i\in \pi_{2^{i+1}-2}^{\mathbb S}$ , if they exist, are represented by the elements $h_i^{2}\in {\rm Ext}_{\mathscr {A}_2}^{2,2^{i+1}}(\mathbb Z_2,\mathbb Z_2)$ on $E_2$-page. People have been known that $h_i^{2}$ survives for $i\leq 5,$ but when $i\geq 7,$ the surprising work of Hill, Hopkins, and Ravenel \cite{Hill} proved that those elements cannot survive to the $E_{\infty}$-page. The case $i = 6$  is unknown.  Thus, understanding hit problems and Singer's algebraic transfer is quite important for solving the open problems mentioned. Now coming back to the Singer algebraic transfer on hand, the behavior of $Tr_t^{\mathscr A_2}$ in higher ranks was partly surveyed by Singer himself \cite{W.S1}, where he applies the invariant theory to show that $Tr_4^{\mathscr A_2}$ is an isomorphism in the certain internal degrees. Thence, he predicts that \textit{$Tr_t^{\mathscr A_2}$ is a monomorphism}, but this is currently no answer for $t\geq 5.$ This hypothesis has been considered by many researchers, including as well the present writer. For instance, due to Singer himself \cite{W.S1} and Boardman \cite{J.B}, the conjecture was confirmed for ranks $\leq 3.$ The recent works of the author \cite{D.P1, D.P2, D.P4, D.P6, D.P10} and Sum \cite{N.S2, N.S3} have studied it in the ranks 4 and 5 and some certain generic degrees of the form \eqref{ct}. Most notably, our calculations \cite{D.P10} have shown that Singer's conjecture holds also for rank four. Continuing these works, in the present paper, we would like to probe this conjecture for the rank $5$ case and the internal degrees $14,\, 31$ and $32.$ To accomplish this, we first explicitly determine the domain of the fifth transfer in those degrees. Therefrom, together with the calculations of Lin \cite{Lin} and Chen \cite{Chen} on the fifth cohomology groups of the Steenrod algebra, we obtain the information on $Tr_5^{\mathscr A_2}$ in degrees given. More precisely, by direct computations using a monomial basis of $Q^{\otimes 5}_{14}$ (see Remark \ref{nx14} in Sect.\ref{s2}) and the proofs of Theorems \ref{dlc2} and \ref{dlc3}, one gets the technical theorem below.

\begin{thm}\label{dlc5}
With the notations as above, 
$$\dim [P_{\mathscr A_2}((\mathbb P_n^{\otimes 5})^{*})]_{GL_5} = \left\{\begin{array}{ll}
1 &\mbox{if $n = 14$},\\
2 &\mbox{if $n = 31$},\\
0 &\mbox{if $n = 32$}.
\end{array}\right.$$
\end{thm}

On the other side, an important result due to Lin \cite{Lin} and Chen \cite{Chen} shows that
$$  {\rm Ext}_{\mathscr {A}_2}^{5,5+n}(\mathbb Z_2,\mathbb Z_2) = \left\{\begin{array}{ll}
\langle h_0d_0 \rangle &\mbox{if $n = 14$},\\
\langle h_0^{4}h_5, n_0 \rangle &\mbox{if $n = 31$},\\
\langle h_4e_0 \rangle = 0 &\mbox{if $n = 32$}.
\end{array}\right.$$
Moreover, following Singer \cite{W.S1}, H\`a \cite{Ha}, Ch\ohorn n-H\`a \cite{C.H}, the non-zero elements $h_i\in {\rm Ext}_{\mathscr {A}_2}^{1,2^{i}}(\mathbb Z_2,\mathbb Z_2)$ for $i\geq 0$, $d_0\in {\rm Ext}_{\mathscr {A}_2}^{4, 4+14}(\mathbb Z_2,\mathbb Z_2)$ and $n_0\in {\rm Ext}_{\mathscr {A}_2}^{5,5+31}(\mathbb Z_2,\mathbb Z_2),$ are detected by the algebraic transfer. Also note that although the authors \cite{C.H} can easily deduce that the fifth transfer is an epimorphism in the internal degree $31,$ their techniques were not directly applicable in confirming or refuting Singer's hypothesis in this case. So, the following assertion, which is much more important and meaningful than their work, is a direct consequence of the above data, Theorem \ref{dlc5} and the fact that $Tr^{\mathscr A_2}$ is a homomorphism of algebras.

\begin{corl}
The transfer homomorphism
$$ Tr_5^{\mathscr A_2}: [P_{\mathscr A_2}((\mathbb P_n^{\otimes 5})^{*})]_{GL_5}\to {\rm Ext}_{\mathscr {A}_2}^{5, 5+n}(\mathbb Z_2,\mathbb Z_2)$$
is an isomorphism for $n\in \{14, 31, 32\}.$ Consequently, Singer's conjecture holds for the fifth transfer in those internal degrees.
\end{corl}

{\bf Discussions and future research.} Using the method outlined in the article, we will continue to study the space \eqref{kgvt2} for $t\in \{5, 6\},$ and extend Corollary \ref{hq31} for degree \mbox{$n_d := 5(2^{d}-1) + 13.2^{d}$} whenever $d > 1.$ In actual fact, we only need to investigate  the case $d = 2.$ Indeed, it can be easily seen that $\mu(n_d) = 5$ for any $d > 2,$ and we therefore deduce an isomorphism $(\widetilde {Sq^0_*})^{d-2}:  Q^{\otimes 5}_{n_d} \xrightarrow{\cong} Q^{\otimes 5}_{n_2},$ for all $d \geq 2.$ This together with Corollary \ref{hq31} allow us to assert the above request. Now, when $d = 2,$ notice that the Kameko map $ [\widetilde {Sq^0_*}]_{n_2}:= \widetilde {Sq^0_*}: Q^{\otimes 5}_{n_2} \to Q^{\otimes 5}_{31}$ is an epimorphism, which implies that $Q^{\otimes 5}_{n_2}\cong {\rm Ker}[\widetilde {Sq^0_*}]_{n_2} \bigoplus Q^{\otimes 5}_{31}.$ So, invoking Corollary \ref{hq31}, we need to compute the kernel of $[\widetilde {Sq^0_*}]_{n_2}.$ Using Theorem \ref{dlc3} together with some of the same arguments as in the proof of Theorem \ref{dlc1}, it may be concluded that $${\rm Ker}[\widetilde {Sq^0_*}]_{n_2}\cong (Q_{n_2}^{\otimes 5})^{0} \bigoplus \big(\bigoplus_{1\leq i\leq 3} (Q_{n_2}^{\otimes 5})^{\overline{\omega}_{(i)}^{>0}}\big)\cong (Q_{n_2}^{\otimes 5})^{0} \bigoplus ({\rm Ker}[\widetilde {Sq^0_*}]_{n_2}\cap (Q_{n_2}^{\otimes 5})^{> 0}),$$ where $\overline{\omega}_{(1)} = (3,2,1,1,1,1),$ $\overline{\omega}_{(2)} = (3,4,2,2,2)$ and $\overline{\omega}_{(3)} = (3,4,4,3,1)$ are the weight vectors of degree $n_2.$ Applying the formula $\dim (Q_{n_2}^{\otimes 5})^0 = \sum_{1\leq  s\leq 4}\binom{5}{s}\dim (Q_{n_2}^{\otimes\, s})^{>0}$ and the previous results by Peterson \cite{F.P1}, Kameko \cite{M.K}, Sum \cite{N.S1}, we easily obtain the dimension of $(Q_{n_2}^{\otimes 5})^{0}.$ Thus, together with this and the above data, the dimension of $Q^{\otimes 5}_{n_d}, \ d\geq 2,$ can be determined by calculating explicitly the monomial basis of $(Q_{n_2}^{\otimes 5})^{\overline{\omega}_{(i)}^{>0}}$ for $i = 1,2, 3.$ Using techniques mentioned and some preliminary calculations, we notice that determining $(Q_{n_2}^{\otimes 5})^{\overline{\omega}_{(2)}^{>0}}$ and $(Q_{n_2}^{\otimes 5})^{\overline{\omega}_{(3)}^{>0}}$ is simple, but it is not easy for $(Q_{n_2}^{\otimes 5})^{\overline{\omega}_{(1)}^{>0}}.$ Of course, the results can be obtained by using an algorithm in MAGMA. Additionally, a useful observation that $\mu(n_2) = 3 = \alpha(n_2 + \mu(n_2)),$ and so according to Theorem \ref{dlP} (see Sect.\ref{s2}), it follows that $$\dim Q^{\otimes 6}_{5(2^{d+3}-1) + n_2.2^{d+3}} = (2^{6}-1)\dim Q^{\otimes 5}_{n_d},\ \mbox{for arbitrary $d > 1.$}$$ Regarding to Singer's conjecture for the rank 5 case and the internal degree $n_d,$ we have the following conjecture, which is motivated by these discussions and Theorem \ref{dlc5}.
\begin{conj}\label{gtP}
For each non-negative integer $d,$ $\dim [P_{\mathscr A_2}((\mathbb P_{n_d}^{\otimes 5})^{*})]_{GL_5} = 0$ if $d = 0,$ and $\dim [P_{\mathscr A_2}((\mathbb P_{n_d}^{\otimes 5})^{*})]_{GL_5} = 2$ otherwise.
\end{conj}
Based on our previous work \cite{D.P2} and Theorem \ref{dlc5}, we see that Conjecture \ref{gtP} holds for the cases $d = 0$ and $1.$ On the other side, using the results by Lin \cite{Lin} and Chen \cite{Chen}, one has that
$$  {\rm Ext}_{\mathscr {A}_2}^{5,5+n_d}(\mathbb Z_2,\mathbb Z_2) = \left\{\begin{array}{ll}
\langle h_1h^{4}_2, h_0^{2}h_2^{2}h_3\rangle = 0 &\mbox{if $d = 0$},\\
\langle h_0^{4}h_5, n_0 \rangle &\mbox{if $d = 1$},\\
\langle n_{d-1}, Q_3(d-2) \rangle &\mbox{if $d > 1$}.
\end{array}\right.$$
Let us recall that the $\mathbb Z_2$-cohomology of Steenrod ring ${\rm Ext}_{\mathscr {A}_2}^{*,*}(\mathbb Z_2,\mathbb Z_2)$ is a bigraded algebra. Following Chen \cite{Chen}, let ${\rm Ext}_{\mathscr {A}_2}^{t,*}(\mathbb Z_2,\mathbb Z_2)\otimes {\rm Ext}_{\mathscr {A}_2}^{t',*}(\mathbb Z_2,\mathbb Z_2) \xrightarrow{\mu}{\rm Ext}_{\mathscr {A}_2}^{t+t',*}(\mathbb Z_2,\mathbb Z_2)$ be the multiplication map of ${\rm Ext}_{\mathscr {A}_2}^{*,*}(\mathbb Z_2,\mathbb Z_2).$ An element $\zeta\in {\rm Ext}_{\mathscr {A}_2}^{\overline{t},*}(\mathbb Z_2,\mathbb Z_2)$ is said to be a \textit{decomposable element} if $\overline{t}\geq 2$ and there exist $t > 0$,  $t' > 0$ such that $t + t' = \overline{t}$ and $\zeta = \sum_{j}\mu(\alpha_j\otimes \beta_j)$ for some $\alpha_j\in {\rm Ext}_{\mathscr {A}_2}^{t,*}(\mathbb Z_2,\mathbb Z_2)$ and $\beta_j\in {\rm Ext}_{\mathscr {A}_2}^{t',*}(\mathbb Z_2,\mathbb Z_2).$ For instance, $h_0^{4}h_5$ is a decomposable element in ${\rm Ext}_{\mathscr {A}_2}^{5,36}(\mathbb Z_2,\mathbb Z_2).$ A non-zero element in ${\rm Ext}_{\mathscr {A}_2}^{t,*}(\mathbb Z_2,\mathbb Z_2)$ is an \textit{indecomposable element} if either $t=1$ or if $t\geq 2$  then its image in ${\rm Ext}_{\mathscr {A}_2}^{t,*}(\mathbb Z_2,\mathbb Z_2)/D^{t, *}$ is non-zero. Here the $\mathbb Z_2$-submodule $D^{t, *}$ is the set of all the decomposable elements in ${\rm Ext}_{\mathscr {A}_2}^{t,*}(\mathbb Z_2,\mathbb Z_2),$ where $D^{1, *} = 0.$ When $t = 5,$ Chen showed in \cite{Chen} that ${\rm Ext}_{\mathscr {A}_2}^{5,*}(\mathbb Z_2,\mathbb Z_2)$ contains eleven $Sq^{0}$-families of indecomposable elements, namely $n_d,\, \chi_d,\, D_1(d),\, H_1(d),\, Q_3(d),\, K_d,\, J_d,\, T_d,\, V_d,\, V'_d,$ and $U_d,$ for $d\geq 0.$ According to Ch\ohorn n, and H\`a \cite{C.H}, the image of $Tr_5^{\mathscr A_2}$ contains every element in the family $\{n_{d-1} = (Sq^{0})^{d-1}(n_0)\}_{d\geq 1}.$  In \cite{Hung}, H\uhorn ng conjectured that \textit{$Q_3(0)$ is detected by $Tr_5^{\mathscr A_2}.$} If this prediction is true, then the image of the fifth transfer contains all the indecomposable elements of the family $\{Q_3(d-2) = (Sq^{0})^{d-2}(Q_3(0))\}_{d\geq 2}$ since the following diagram is commutative:
$$ \begin{diagram}
\node{[P_{\mathscr A_2}((\mathbb P_{n_{d-2}}^{\otimes 5})^{*})]_{GL_5}} \arrow{e,t}{Tr^{\mathscr A_2}_5}\arrow{s,r}{\widetilde{Sq^0}} 
\node{{\rm Ext}_{\mathscr A_2}^{5, 5+n_{d-2}}(\mathbb Z_2, \mathbb Z_2)} \arrow{s,r}{Sq^0}\\ 
\node{[P_{\mathscr A_2}((\mathbb P_{5+2n_{d-2}}^{\otimes 5})^{*})]_{GL_5}} \arrow{e,t}{Tr^{\mathscr A_2}_5} \node{{\rm Ext}_{\mathscr A_2}^{5, 2(5+n_{d-2})}(\mathbb Z_2, \mathbb Z_2).}
\end{diagram}$$
Here $\widetilde{Sq^0}$ denotes Kameko's squaring operation, while $Sq^0$ is the classical squaring operation. At the same time, combining with the calculations by Singer \cite{W.S1}, Boardman \cite{J.B}, Bruner \cite{Bruner}, H\`a \cite{Ha}, and the fact that $Tr^{\mathscr A_2}$ is an algebraic homomorphism, it may be concluded that the following non-zero elements are detected by the algebraic transfer:
$$ \begin{array}{lll}
\medskip
 &h_0Q_3(0)\in {\rm Ext}_{\mathscr A_2}^{6,6+n_2}(\mathbb Z_2, \mathbb Z_2), & h_0^{2}Q_3(0)\in {\rm Ext}_{\mathscr A_2}^{7,7+n_2}(\mathbb Z_2, \mathbb Z_2), \\
\medskip
& h_1Q_3(0)\in {\rm Ext}_{\mathscr A_2}^{6,6+n_2+1}(\mathbb Z_2, \mathbb Z_2), & h_1^{2}Q_3(0)\in {\rm Ext}_{\mathscr A_2}^{7,7+n_2+2}(\mathbb Z_2, \mathbb Z_2),\\
\medskip
& h_2Q_3(0)\in {\rm Ext}_{\mathscr A_2}^{6,6+n_2+3}(\mathbb Z_2, \mathbb Z_2),& h_0h_2Q_3(0)\in {\rm Ext}_{\mathscr A_2}^{7,7+n_2+3}(\mathbb Z_2, \mathbb Z_2),\\
\medskip
& h_2^{2}Q_3(0)\in {\rm Ext}_{\mathscr A_2}^{7,7+n_2+6}(\mathbb Z_2, \mathbb Z_2),& c_0Q_3(0)\in {\rm Ext}_{\mathscr A_2}^{8,8+n_2+8}(\mathbb Z_2, \mathbb Z_2),\\
\medskip
& h_1c_0Q_3(0)\in {\rm Ext}_{\mathscr A_2}^{9,9+n_2+9}(\mathbb Z_2, \mathbb Z_2), & d_0Q_3(0)\in {\rm Ext}_{\mathscr A_2}^{9,9+n_2+14}(\mathbb Z_2, \mathbb Z_2),\\
\medskip
& h_0d_0Q_3(0)\in {\rm Ext}_{\mathscr A_2}^{10,10+n_2+14}(\mathbb Z_2, \mathbb Z_2),&  h_4Q_3(0)\in {\rm Ext}_{\mathscr A_2}^{6,6+n_2+15}(\mathbb Z_2, \mathbb Z_2),\\
\medskip
&  h_0h_4Q_3(0)\in {\rm Ext}_{\mathscr A_2}^{7,7+n_2+15}(\mathbb Z_2, \mathbb Z_2),&  h_0^{2}h_4Q_3(0)\in {\rm Ext}_{\mathscr A_2}^{8,8+n_2+15}(\mathbb Z_2, \mathbb Z_2),\\
\medskip
& h_1h_4Q_3(0)\in {\rm Ext}_{\mathscr A_2}^{7,7+n_2+16}(\mathbb Z_2, \mathbb Z_2), &  h_1^{2}h_4Q_3(0)\in {\rm Ext}_{\mathscr A_2}^{8,8+n_2+17}(\mathbb Z_2, \mathbb Z_2), \\
\medskip
&  e_0Q_3(0)\in {\rm Ext}_{\mathscr A_2}^{9,9+n_2+17}(\mathbb Z_2, \mathbb Z_2),& h_2h_4Q_3(0)\in {\rm Ext}_{\mathscr A_2}^{7,7+n_2+18}(\mathbb Z_2, \mathbb Z_2),\\
\medskip
&  h_0h_2h_4Q_3(0)\in {\rm Ext}_{\mathscr A_2}^{8,8+n_2+18}(\mathbb Z_2, \mathbb Z_2), & h_5Q_3(0)\in {\rm Ext}_{\mathscr A_2}^{6,6+n_2+31}(\mathbb Z_2, \mathbb Z_2), \\
\medskip
& h_2h_5Q_3(0)\in {\rm Ext}_{\mathscr A_2}^{7,7+n_2+34}(\mathbb Z_2, \mathbb Z_2).
\end{array}$$

The above data together with Conjecture \ref{gtP} tempt us to propose: 
\begin{conj}\label{gtP2}
The algebraic transfer is an isomorphism in bidegree $(5, 5+n_d)$ for any $d\geq 0.$
\end{conj}
Obviously, the conjecture is true for $d \in \{0, 1\}.$ Computationally the above issues seem rather non-trivial and we hope to come back to them in another work. Summarizing, although our work does not apparently lead to a general rule for determining the space \eqref{kgvt2}, we feel that it represents an interesting note about an efficient approach to understand the structure of those spaces and their applications.  Perhaps a continuation of our methods may provefruitful.  It is our belief that further research canspring from these ideas.

\begin{acknow}
This work is supported by National Foundation for Science and Technology Development (NAFOSTED) of Vietnam under grant number 101.04-2017.05. The author would like to gratefully and sincerely thank Prof. N. Sum, for many enlightening e-mail exchanges.

I am very grateful to the anonymous referees for carefully reading the manuscript and for offering comments with suggestions that allowed me to substantially improve the article. I would also like to express my thankfulness to Prof. R. Bruner, for his help about the algorithm in MAGMA. 
\end{acknow}

\section{Proofs of main results}\label{s2}

To facilitate this, we will start in this section by recalling a few useful preliminaries on strictly inadmissible and spike monomials, and reminding the reader of something important.

\begin{defn}[{\bf Strictly inadmissible monomial}]\label{dnkcndc} A monomial $X\in\mathbb P_n^{\otimes t}$ is said to be {\it strictly inadmissible} if and only if there exist monomials $Y_1, Y_2,\ldots, Y_k$ in $\mathbb P_n^{\otimes t}$  such that $Y_j < X$ for $1\leq j \leq k$ and 
$X = \sum_{1\leq j\leq k}Y_j + \sum_{1\leq  l\leq  2^s - 1}Sq^{l}(Z_l),$ where $s = {\rm max}\{i\in\mathbb Z: \omega_i(X) > 0\}$ and suitable polynomials $Z_l\in \mathbb P_{n-l}^{\otimes t}.$
\end{defn}

It would be helpful for readers to note that due to Definitions \ref{dninadm} and \ref{dnkcndc}, each strictly inadmissible monomial is inadmissible. Conversely, it is, however in general, not true;  for instance, $X = x_1x_2^2x_3^2x_4^2x_5^2x_6\in \mathbb P^{\otimes 6}_{10}$ is an inadmissible monomial, but it is not strictly inadmissible.

\begin{defn}[{\bf Spike monomial}]
A monomial $Z = \prod_{1\leq j\leq t}x_j^{u_j}$ in $\mathbb P_n^{\otimes t}$ is called a {\it spike} if every exponent $u_j$ is of the form $2^{\lambda_j} - 1.$ Specifically, if the exponents $\lambda_j$ can be arranged to satisfy $\lambda_1 > \lambda_2 > \ldots > \lambda_{s-1}\geq \lambda_s \geq 1,$ where only the last two smallest exponents can be equal, and $\lambda_j = 0$ for $ s < j  \leq t,$ then $Z$ is called a {\it minimal spike}.
\end{defn}

The following technical theorems are very important throughout this paper.

\begin{thm}[see Kameko \cite{M.K}]\label{dlKS}
Let $X, Y$ be monomials in $\mathbb P_n^{\otimes t}$ and let $s$ be a positive integer. Assume that there is an index $j > s$ such that $\omega_j(X) = 0.$ Then, if $Y$ is an inadmissible monomial, so is $XY^{2^s}.$
\end{thm} 

\begin{thm}[see Ph\'uc-Sum \cite{P.S1}]\label{dlPS}
All the spikes in $\mathbb P_n^{\otimes t}$ are admissible and their weight vectors are weakly decreasing. Furthermore, if a weight vector $\omega = (\omega_1, \omega_2, \ldots)$ is weakly decreasing and $\omega_1\leq t,$ then there is a spike $Z$ in $\mathbb P_n^{\otimes t}$ such that $\omega(Z) = \omega.$
\end{thm}

\begin{thm}[see Singer \cite{W.S2}]\label{dlSin}
Suppose that $X\in \mathbb P_n^{\otimes t}$ is a monomial of degree $n,$ where $\mu(n)\leq t.$ Let $Z$ be the minimal spike of degree $n$ in $\mathbb P_n^{\otimes t}.$ If $\omega(X) < \omega(Z),$ then $X\equiv 0.$ 
\end{thm}

\begin{thm}[see Mothebe-Uys \cite{M.M}]\label{dlMU}
Let $l, d$ be positive integers such that $1\leq l\leq t.$ If $X$ belongs to $\mathscr {C}^{\otimes (t-1)}_{n},$ then $x_l^{2^{d}-1}\mathsf{q}_{(l,\,t)}(X)$ belongs to $\mathscr {C}^{\otimes t}_{n + 2^{d}-1}.$
\end{thm}

\begin{thm}[see Sum \cite{N.S1}, Ph\'uc \cite{D.P6}]\label{dlP}
Consider generic degree of the form \eqref{ct} with $r = t-1$ and $d,\, m$ positive integers. Suppose there is a integer $\zeta$ such that $0\leq \zeta < d$ and $1\leq t-3\leq \mu(n_{\zeta}) = \alpha(n_{\zeta}+ \mu(n_{\zeta}))\leq t-2.$ Then, for each $d \geq \zeta,$ we have 
$$ \dim Q^{\otimes t}_{(t-1)(2^{d-\zeta + t-1}-1) + n_{\zeta}2^{d-\zeta + t-1}} =(2^{t}-1)\dim Q^{\otimes (t-1)}_{n_d}.$$
\end{thm}

Noting that this theorem is an equivalent statement of Theorem 1.3 in Sum \cite{N.S1} by using Kameko's theorem \cite{M.K} and the proof of Theorem 1.3.

In what follows, for a monomial $F\in \mathbb P^{\otimes t}_n,$ we denote by $[F]$ the equivalence class of $F$ in $Q_n^{\otimes t}.$ If $\omega$ is a weight vector of degree $n$ and $F\in (\mathbb P_n^{\otimes t})^{\leq \omega},$ then denote by $[F]_\omega$ the equivalence class of $F$ in $(Q_n^{\otimes t})^{\omega}.$  For a subset $\mathscr{C}\subset \mathbb P_n^{\otimes t},$ we will often write $|\mathscr C|$ for the cardinality of $\mathscr C$ and put $[\mathscr C] = \{[F]\, :\, F\in \mathscr C\}.$ If $\mathscr C\subset  (\mathbb P_n^{\otimes t})^{\leq \omega},$ then denote $[\mathscr C]_{\omega} = \{[F]_{\omega}\, :\, F\in \mathscr C\}.$ Let us consider the sets $(\mathscr{C}^{\otimes t}_{n})^{\omega^{0}} := (\mathscr{C}^{\otimes t}_{n})^{\omega}\cap  (\mathbb P_n^{\otimes t})^{0},$ and $(\mathscr{C}^{\otimes t}_{n})^{\omega^{>0}} := (\mathscr{C}^{\otimes t}_{n})^{\omega}\cap  (\mathbb P_n^{\otimes t})^{>0},$ where $(\mathscr{C}^{\otimes t}_{n})^{\omega} := \mathscr {C}^{\otimes t}_n\cap (\mathbb P_n^{\otimes t})^{\leq \omega}.$ Then, one has that the sets $[(\mathscr{C}^{\otimes t}_{n})^{\omega}]_\omega,\, [(\mathscr{C}^{\otimes t}_{n})^{\omega^{0}}]_\omega$ and $[(\mathscr{C}^{\otimes t}_{n})^{\omega^{>0}}]_\omega$ are respectively the bases of the $\mathbb Z_2$-vector spaces $(Q_n^{\otimes t})^{\omega},\ (Q_n^{\otimes t})^{\omega^{0}}$ and $(Q_n^{\otimes t})^{\omega^{>0}}.$ 

\subsection{Proof of Theorem \ref{dlc1}}

It is known that for $t = 5,$ the Kameko $\widetilde {Sq^0_*}$, which is an epimorphism, determined as follows:
$$ \begin{array}{ll}
\widetilde {Sq^0_*}: Q^{\otimes 5}_{31}  &\longrightarrow Q^{\otimes 5}_{13}\\
\ \ \ \mbox{[}\prod_{1\leq j\leq 5}x_j^{a_j}\mbox{]}&\longmapsto \left\{\begin{array}{ll}
\mbox{[}\prod_{1\leq j\leq 5}x_j^{\frac{a_j-1}{2}}\mbox{]}& \text{if $a_j$ odd, $j = 1, 2, \ldots, 5$},\\
0& \text{otherwise}.
\end{array}\right.
\end{array}$$
So, in order to prove theorem, we first show that if $X\in \mathscr {C}^{\otimes 5}_{31}$ such that $[X]\in {\rm Ker}(\widetilde {Sq_*^0}),$ then $\omega(X)$ is one of the sequences $\omega_{(j)},$ for $1\leq j\leq 7,$ where $\omega_{(4)}:= (1,1,1,3),$ $\omega_{(5)}:= (1,3,2,2),$ $\omega_{(6)}:= (1,3,4,1)$ and $\omega_{(7)}:=(3,2,4,1).$ Indeed,  we observe that $Z = x_1^{31}$ is the minimal spike in $\mathbb P_{31}^{\otimes 5}$ and $\omega(Z) = \omega_{(1)}.$ Since $X\in \mathscr C_{31}^{\otimes 5}$ and $\deg(X) = 31$ odd, by Theorem \ref{dlSin}, it may be concluded that $\omega_1(X)\in \{1, 3, 5\}.$ 

 If $\omega_1(X) = 5,$ then $X = \prod_{1\leq j\leq 5}x_jY^2$ with $Y$  a monomial of degree $13$ in $\mathbb P^{\otimes 5}.$ Since $X\in \mathscr {C}^{\otimes 5}_{31}$ and $\omega_i(\prod_{1\leq j\leq 5}x_j) = 0,$ for all $i > 1,$ according to Theorem \ref{dlKS}, $Y\in \mathscr {C}^{\otimes 5}_{13}.$  This implies that $\widetilde {Sq_*^0}([X]) = [Y]\neq 0,$ which contradicts the fact that $[X]\in \mbox{Ker}(\widetilde {Sq_*^0}).$ Hence, we get either $\omega_1(X) = 1\ \mbox{or}\ \omega_1(X) = 3.$ 

If $\omega_1(X) = 1,$ then following Theorem \ref{dlKS}, one has that $X = x_jZ^{2}$ with $Z\in \mathscr {C}^{\otimes 5}_{15}$ and $1\leq j\leq 5.$ By Sum \cite{N.S}, we get $\omega(X)\in \{\omega_{(1)}, \omega_{(4)}, \omega_{(5)}, \omega_{(6)}\}.$

If $\omega_1(X) = 3,$ then by Theorem \ref{dlKS}, $X = x_ax_bx_cT^{2}$ where $T\in \mathscr {C}^{\otimes 5}_{14}$ and $1\leq a<b<c\leq 5.$ Due to Ly-Tin \cite{L.T}, we deduce that $\omega(X)$ is one of the sequences $\omega_{(j)},\ j = 2, 3, 7.$

Next, we will show that if $\omega(X)\in \{\omega_{(i)}\}_{4\leq i\leq 7},$ then $(Q^{\otimes 5}_{31})^{\omega(X)} = 0.$ As shown in Sect.\ref{s1}, $\dim(Q^{\otimes 5}_{31})^{0} = 330,$ further by a simple computation, we obtain $$(\mathscr C_{31}^{\otimes 5})^{0} = \bigcup_{1\leq l\leq 5}\mathsf{q}_{(l,\,5)}(\mathscr C_{31}^{\otimes 4}) = (\mathscr C_{31}^{\otimes 5})^{\omega_{(1)}^{0}}\cup (\mathscr C_{31}^{\otimes 5})^{\omega_{(2)}^{0}},$$
where $|(\mathscr C_{31}^{\otimes 5})^{\omega_{(1)}^{0}}| = 30,$ and $|(\mathscr C_{31}^{\otimes 5})^{\omega_{(2)}^{0}}| = 300.$ So, $\dim (Q_{31}^{\otimes 5})^{\omega_{(1)}^{0}} = 30,$ and  $\dim (Q_{31}^{\otimes 5})^{\omega_{(2)}^{0}} = 300.$ Thus, we need only to prove that the spaces $(Q^{\otimes 5}_{31})^{\omega_{(i)}^{>0}}$ are trivial for $4\leq i\leq 7.$ 

\medskip 

\underline{{\it Case  $\omega(X) = \omega_{(4)}$}}. It is easy to see that the monomial $X = x_jZ^{2}\in (\mathbb P_{31}^{\otimes 5})^{>0}, 1\leq j\leq 5,$ is one of the following monomials: $x_1^{2}x_2^{4}x_3^{8}x_4^{8}x_5^{9},\ x_1^{2}x_2^{4}x_3^{8}x_4^{9}x_5^{8},\ x_1^{2}x_2^{4}x_3^{9}x_4^{8}x_5^{8},\  x_1^{2}x_2^{5}x_3^{8}x_4^{8}x_5^{8},\ x_1^{3}x_2^{4}x_3^{8}x_4^{8}x_5^{8}.$ They are inadmissible. Indeed, by a simple composition, we have
$$ X = x_jZ^{2} = x_1^{3}x_2^{4}x_3^{8}x_4^{8}x_5^{8} = Sq^{1}(x_1^{3}x_2^{3}x_3^{8}x_4^{8}x_5^{8}) + Sq^{2}(x_1^{2}x_2^{3}x_3^{8}x_4^{8}x_5^{8}) + x_1^{2}x_2^{5}x_3^{8}x_4^{8}x_5^{8} \mod (\mathbb P_{31}^{\otimes 5})^{< \omega_{(4)}}.$$ 
Since $x_1^{2}x_2^{5}x_3^{8}x_4^{8}x_5^{8} < x_1^{3}x_2^{4}x_3^{8}x_4^{8}x_5^{8},$ $X$ is inadmissible. 
\medskip

\underline{{\it Case  $\omega(X) = \omega_{(5)}$}}. Since $X = x_jZ^{2}$ with $Z\in \mathscr {C}^{\otimes 5}_{15},$ and $1\leq j\leq 5,$ $X$ is a permutation of one of the following monomials:

\begin{center}
\begin{tabular}{lrrr}
$x_ix_j^{2}x_k^{2}x_l^{12}x_m^{14}$, & \multicolumn{1}{l}{$x_ix_j^{2}x_k^{4}x_l^{10}x_m^{14}$,} & \multicolumn{1}{l}{$x_ix_j^{2}x_k^{6}x_l^{8}x_m^{14}$,} & \multicolumn{1}{l}{$x_ix_j^{2}x_k^{6}x_l^{10}x_m^{12}$,} \\
$x_ix_j^{4}x_k^{6}x_l^{10}x_m^{10}$, & \multicolumn{1}{l}{$x_ix_j^{6}x_k^{6}x_l^{8}x_m^{10}$,} & \multicolumn{1}{l}{$x_i^{3}x_j^{2}x_k^{2}x_l^{12}x_m^{12}$,} & \multicolumn{1}{l}{$x_i^{3}x_j^{2}x_k^{4}x_l^{8}x_m^{14}$,} \\
$x_i^{3}x_j^{2}x_k^{4}x_l^{10}x_m^{12}$, & \multicolumn{1}{l}{$x_i^{3}x_j^{2}x_k^{6}x_l^{8}x_m^{12}$,} & \multicolumn{1}{l}{$x_i^{3}x_j^{4}x_k^{4}x_l^{10}x_m^{10}$,} & \multicolumn{1}{l}{$x_i^{3}x_j^{4}x_k^{6}x_l^{8}x_m^{10}$,} \\
$x_i^{3}x_j^{6}x_k^{6}x_l^{8}x_m^{8}$, & \multicolumn{1}{l}{$x_i^{7}x_j^{2}x_k^{2}x_l^{8}x_m^{12}$,} & \multicolumn{1}{l}{$x_i^{7}x_j^{2}x_k^{4}x_l^{8}x_m^{10}$,} & \multicolumn{1}{l}{$x_i^{7}x_j^{2}x_k^{6}x_l^{8}x_m^{8}$,} \\
$x_i^{15}x_j^{2}x_k^{2}x_l^{4}x_m^{8}$, &       &       &  
\end{tabular}
\end{center}
where $(i, j, k, l, m)$  is a permutation of $(1, 2,3,4,5).$ A direct computation shows that 
$$ X = x_ix_j^{2}x_k^{2}x_l^{12}x_m^{14} = Sq^{1}(x_i^{2}x_jx_kx_l^{12}x_m^{14}) + Sq^{2}(x_ix_jx_kx_l^{12}x_m^{14} + x_ix_jx_kx_l^{10}x_m^{16}) \mod (\mathbb P_{31}^{\otimes 5})^{< \omega_{(5)}}.$$
This implies that $X$ is inadmissible. 

\medskip

\underline{{\it Case  $\omega(X) = \omega_{(6)}$}}. We have $X = x_jZ^{2}$ with $\omega(Z) = (3,4,1)$ and $1\leq j\leq 5.$ By a simple computation, we find that $X$ of the form $FG^{2^{s}},$ where $s$ is a suitable integer and $F$ is one of the following inadmissible monomials:

\begin{center}
\begin{tabular}{lllll}
$F_{1}=x_1^{2}x_4^{2}x_5^{3}$, & $F_{2}=x_1^{2}x_4^{3}x_5^{2}$, & $F_{3}=x_1^{2}x_3^{2}x_5^{3}$, & $F_{4}=x_1^{2}x_3^{2}x_4^{3}$, & $F_{5}=x_1^{2}x_3^{3}x_5^{2}$, \\
$F_{6}=x_1^{2}x_3^{3}x_4^{2}$, & $F_{7}=x_1^{2}x_2^{2}x_5^{3}$, & $F_{8}=x_1^{2}x_2^{2}x_4^{3}$, & $F_{9}=x_1^{2}x_2^{2}x_3^{3}$, & $F_{10}=x_1^{2}x_2^{3}x_5^{2}$, \\
$F_{11}=x_1^{2}x_2^{3}x_4^{2}$, & $F_{12}=x_1^{2}x_2^{3}x_3^{2}$, & $F_{13}=x_1^{3}x_4^{2}x_5^{2}$, & $F_{14}=x_1^{3}x_3^{2}x_5^{2}$, & $F_{15}=x_1^{3}x_3^{2}x_4^{2}$, \\
$F_{16}=x_1^{3}x_2^{2}x_5^{2}$, & $F_{17}=x_1^{3}x_2^{2}x_4^{2}$, & $F_{18}=x_1^{3}x_2^{2}x_3^{2}$, & $F_{19}=x_1^{2}x_3x_4^{2}x_5^{2}$, & $F_{20}=x_1^{2}x_3^{2}x_4x_5^{2}$, \\
$F_{21}=x_1^{2}x_3^{2}x_4^{2}x_5$, & $F_{22}=x_1^{2}x_2x_4^{2}x_5^{2}$, & $F_{23}=x_1^{2}x_2x_3^{2}x_5^{2}$, & $F_{24}=x_1^{2}x_2x_3^{2}x_4^{2}$, & $F_{25}=x_1^{2}x_2^{2}x_4x_5^{2}$, \\
$F_{26}=x_1^{2}x_2^{2}x_4^{2}x_5$, & $F_{27}=x_1^{2}x_2^{2}x_3x_5^{2}$, & $F_{28}=x_1^{2}x_2^{2}x_3x_4^{2}$, & $F_{29}=x_1^{2}x_2^{2}x_3^{2}x_5$, & $F_{30}=x_1^{2}x_2^{2}x_3^{2}x_4$.
\end{tabular}%
\end{center}

Then, according to Theorem \ref{dlKS}, $X$ is inadmissible.

\medskip

\underline{{\it Case $\omega(X) = \omega_{(7)}$}}. Note that $X = x_ax_bx_cT^{2}$ with $T\in \mathscr {C}^{\otimes 5}_{14},$ and $1\leq a<b<c\leq 5.$ Then, $X$ is a permutation of one of the following monomials:

\begin{center}
\begin{tabular}{lrrr}
$x_i^{3}x_j^{4}x_k^{4}x_l^{5}x_m^{15}$, & $x_i^{3}x_j^{4}x_k^{4}x_l^{7}x_m^{13}$, & $x_i^{3}x_j^{4}x_k^{5}x_l^{5}x_m^{14}$, & \multicolumn{1}{l}{$x_i^{3}x_j^{4}x_k^{5}x_l^{6}x_m^{13}$,} \\
$x_i^{3}x_j^{4}x_k^{5}x_l^{7}x_m^{12}$, & $x_i^{3}x_j^{5}x_k^{5}x_l^{6}x_m^{12}$, & $x_i^{7}x_j^{2}x_k^{4}x_l^{5}x_m^{13}$, & \multicolumn{1}{l}{$x_i^{7}x_j^{2}x_k^{5}x_l^{5}x_m^{12}$,} \\
$x_i^{7}x_j^{4}x_k^{4}x_l^{5}x_m^{11}$, & $x_i^{7}x_j^{4}x_k^{5}x_l^{5}x_m^{10}$, & $x_i^{15}x_j^{2}x_k^{4}x_l^{5}x_m^{5}$, &  \\
\end{tabular}%
\end{center}
where $(i, j, k, l, m)$  is a permutation of $(1, 2,3,4,5).$ Using the Cartan formula, one has
$$ \begin{array}{ll}
\medskip
X = x_i^{3}x_j^{4}x_k^{4}x_l^{5}x_m^{15}& = Sq^{1}(x_i^{3}x_jx_k^{2}x_l^{9}x_m^{15} + x_i^{3}x_jx_k^{2}x_l^{5}x_m^{19})\\
\medskip
&\quad + Sq^{2}(x_i^{5}x_j^{2}x_k^{2}x_l^{5}x_m^{15} + x_i^{5}x_jx_k^{2}x_l^{6}x_m^{15})\\
\medskip
&\quad + Sq^{4}(x_i^{3}x_j^{2}x_k^{2}x_l^{5}x_m^{15}) \mod (\mathbb P_{31}^{\otimes 5})^{< \omega_{(7)}},
\end{array}$$
and we therefore deduce that $X$ is inadmissible.

Thus the above computations show that $(Q^{\otimes 5}_{31})^{\omega_{(j)}^{>0}} = 0$ for $j \in \{4, 5, 6, 7\}.$ Now, from the above cases, we have a direct summand decomposition of the $\mathbb Z_2$-vector spaces:
$$\mbox{\rm Ker}(\widetilde {Sq_*^0})\cap (Q_{31}^{\otimes 5})^{>0} = (Q_{31}^{\otimes 5})^{\omega_{(1)}^{>0}}\bigoplus (Q_{31}^{\otimes 5})^{\omega_{(2)}^{>0}}\bigoplus (Q_{31}^{\otimes 5})^{\omega_{(3)}^{>0}},$$ from which we deduce that
$$ \begin{array}{ll}
\mbox{\rm Ker}(\widetilde {Sq_*^0}) &\cong (Q_{31}^{\otimes 5})^0 \bigoplus (\mbox{\rm Ker}(\widetilde {Sq_*^0})\cap (Q_{31}^{\otimes 5})^{>0})\\
&=  (Q_{31}^{\otimes 5})^0  \bigoplus \big(\bigoplus_{1\leq j\leq 3} (Q_{31}^{\otimes 5})^{\omega_{(j)}^{>0}}\big).
\end{array}$$ 
This finishes the proof of Theorem \ref{dlc1}.

\subsection{Proof of Theorem \ref{dlc2}}

In this section, we use the following homomorphisms: For any $(l, \mathscr L)\in\mathcal{N}_5,$ the $\mathbb Z_2$-linear transformation $\mathsf{p}_{(l, \mathscr L)}: \mathbb P_n^{\otimes 5}\to \mathbb P_n^{\otimes 4}$ defined by 
$$ \mathsf{p}_{(l, \mathscr L)}(x_j) = \left\{ \begin{array}{ll}
{x_j}&\text{if }\;1\leq j \leq l-1, \\
\sum_{p\in \mathscr L}x_{p-1}& \text{if}\; j = l,\\
x_{j-1}&\text{if}\; l+1 \leq j \leq 5.
\end{array} \right.$$
It is remarkable that  $\mathsf{p}_{(l, \emptyset)}(x_l) = 0$ for all $l,\ 1\leq l\leq 5.$ Furthermore, this linear map induces a homomorphism of $\mathscr A_2$-algebras which is also denoted by $\mathsf{p}_{(l, \mathscr L)}: \mathbb P^{\otimes 5}\to \mathbb P^{\otimes 4}.$ In particular, if $X$ is a monomial in $\mathbb P_n^{\otimes 5},$ then $\mathsf{p}_{(l, \mathscr L)}(X)\in (\mathbb P_n^{\otimes 4})^{\leq \omega(X)}$ (see \cite{P.S1}).

\newpage
Now, the proof of this theorem follows from the statements below.

\begin{propo}\label{md1}
We have $(Q^{\otimes 5}_{31})^{\omega_{(1)}^{>0}} = \langle [x_1x_2^{2}x_3^{4}x_4^{8}x_5^{16}] \rangle,$ and $[x_1x_2^{2}x_3^{4}x_4^{8}x_5^{16}]\neq 0.$
\end{propo}

\begin{proof}
Let $X\in (\mathscr C^{\otimes 5}_{31})^{\omega_{(1)}^{>0}}.$ Since $X = x_jZ^{2}$ with $Z\in (\mathscr C^{\otimes 5}_{15})^{> 0},$ $1\leq j\leq 5,$ by a simple computation, we see that if $X\neq x_1x_2^{2}x_3^{4}x_4^{8}x_5^{16},$ then $X$ is a permutation of the following monomials:

\begin{center}
\begin{tabular}{lrrr}
$x_i^{2}x_jx_k^{4}x_l^{8}x_m^{16}$, & \multicolumn{1}{l}{$x_l^{2}x_m^{29}$,} & \multicolumn{1}{l}{$x_l^{3}x_m^{28}$,} & \multicolumn{1}{l}{$x_lx_m^{30}$,} \\
$x_kx_l^{2}x_m^{28}$, & \multicolumn{1}{l}{$x_k^{2}x_l^{4}x_m^{25}$,} & \multicolumn{1}{l}{$x_k^{2}x_l^{5}x_m^{24}$,} & \multicolumn{1}{l}{$x_k^{3}x_l^{4}x_m^{24}$,} \\
$x_jx_k^{2}x_l^{4}x_m^{24}$, & \multicolumn{1}{l}{$x_j^{2}x_k^{4}x_l^{8}x_m^{17}$,} & \multicolumn{1}{l}{$x_j^{2}x_k^{4}x_l^{9}x_m^{16}$,} & \multicolumn{1}{l}{$x_j^{2}x_k^{5}x_l^{8}x_m^{16}$,} \\
$x_j^{3}x_k^{4}x_l^{8}x_m^{16}$, &       &       &  
\end{tabular}%
\end{center}
where $i, j, k, l,  m$ are distinct integers and $1\leq i, j, k, l, m \leq 5.$ It is easy to see that $$X = x_i^{2}x_jx_k^{4}x_l^{8}x_m^{16} = Sq^{1}(x_ix_jx_k^{4}x_l^{8}x_m^{16}) + x_ix_j^{2}x_k^{4}x_l^{8}x_m^{16}$$ and $x_ix_j^{2}x_k^{4}x_l^{8}x_m^{16} < x_i^{2}x_jx_k^{4}x_l^{8}x_m^{16}.$ Hence, $X$ is inadmissible. Now suppose there exist a linear relation $\gamma x_1x_2^{2}x_3^{4}x_4^{8}x_5^{16}\equiv 0$ with $\gamma\in\mathbb Z_2.$ Then the homomorphism $\mathsf{p}_{(4, (5))}: \mathbb P_{31}^{\otimes 5}\to \mathbb P_{31}^{\otimes 4}$ sends this equality to $\mathsf{p}_{(4, (5))}(\gamma x_1x_2^{2}x_3^{4}x_4^{8}x_5^{16})\equiv \gamma x_1x_2^{2}x_3^{4}x_4^{24}\equiv 0$ and so  $\gamma  = 0$ since according to Sum \cite{N.S1}, $x_1x_2^{2}x_3^{4}x_4^{24}$ is an admissible monomial in $\mathbb P_{31}^{\otimes 4}.$  The proposition follows.
\end{proof}

The following (Lemmata \ref{bd1} and \ref{bd2}) are crucial observations in the proof of the theorem.
\begin{lema}\label{bd1}
The following monomials are strictly inadmissible:

\begin{center}
\begin{tabular}{llll}
$x_1x_2^{14}x_3^{3}x_4^{12}x_5$, & $x_1^{3}x_2^{12}x_3x_4x_5^{14}$, & $x_1^{3}x_2^{12}x_3x_4^{14}x_5$, & $x_1x_2^{6}x_3^{11}x_4^{12}x_5$, \\
$x_1^{3}x_2^{12}x_3x_4^{2}x_5^{13}$, & $x_1^{3}x_2^{7}x_3^{8}x_4^{12}x_5$, & $x_1^{3}x_2^{12}x_3x_4^{3}x_5^{12}$, & $x_1^{3}x_2^{12}x_3^{3}x_4x_5^{12}$, \\
$x_1^{3}x_2^{12}x_3^{3}x_4^{12}x_5$, & $x_1^{3}x_2^{15}x_3^{4}x_4x_5^{8}$, & $x_1^{3}x_2^{15}x_3^{4}x_4^{8}x_5$, & $x_1^{3}x_2^{4}x_3x_4^{15}x_5^{8}$, \\
$x_1^{3}x_2^{4}x_3^{15}x_4x_5^{8}$, & $x_1^{3}x_2^{4}x_3x_4^{8}x_5^{15}$, & $x_1^{3}x_2^{4}x_3^{15}x_4^{8}x_5$, & $x_1^{3}x_2^{4}x_3^{8}x_4x_5^{15}$, \\
$x_1^{3}x_2^{4}x_3^{8}x_4^{15}x_5$, & $x_1^{15}x_2^{3}x_3^{4}x_4x_5^{8}$, & $x_1^{15}x_2^{3}x_3^{4}x_4^{8}x_5$, & $x_1x_2^{14}x_3^{3}x_4^{4}x_5^{9}$, \\
$x_1^{3}x_2^{4}x_3x_4^{9}x_5^{14}$, & $x_1^{3}x_2^{4}x_3^{9}x_4x_5^{14}$, & $x_1^{3}x_2^{4}x_3x_4^{14}x_5^{9}$, & $x_1^{3}x_2^{4}x_3^{9}x_4^{14}x_5$, \\
$x_1^{3}x_2^{4}x_3^{11}x_4^{12}x_5$, & $x_1x_2^{14}x_3^{3}x_4^{5}x_5^{8}$, & $x_1^{3}x_2^{5}x_3x_4^{8}x_5^{14}$, & $x_1^{3}x_2^{5}x_3x_4^{14}x_5^{8}$, \\
$x_1^{3}x_2^{5}x_3^{14}x_4x_5^{8}$, & $x_1^{3}x_2^{5}x_3^{8}x_4x_5^{14}$, & $x_1^{3}x_2^{5}x_3^{14}x_4^{8}x_5$, & $x_1^{3}x_2^{5}x_3^{8}x_4^{14}x_5$, \\
$x_1x_2^{6}x_3^{3}x_4^{13}x_5^{8}$, & $x_1x_2^{6}x_3^{3}x_4^{12}x_5^{9}$, & $x_1^{3}x_2^{7}x_3^{12}x_4x_5^{8}$,  & $x_1^{3}x_2^{3}x_3^{12}x_4^{12}x_5$,\\
$x_1x_2^{7}x_3^{10}x_4^{12}x_5$, & $x_1^{7}x_2x_3^{10}x_4^{12}x_5$. & & 
\end{tabular}%
\end{center}
\end{lema}

\begin{proof}
Let us consider the monomials $x_1x_2^{6}x_3^{3}x_4^{12}x_5^{9},$ and $x_1^{3}x_2^{7}x_3^{12}x_4x_5^{8}.$ Because the statements are similar,  we will only prove the lemma for these monomials. By a direct computation using the Cartan formula, we get the equality below.
$$ \begin{array}{ll}
x_1x_2^{6}x_3^{3}x_4^{12}x_5^{9}&= x_1x_2^{3}x_3^{5}x_4^{8}x_5^{14}+
x_1x_2^{3}x_3^{6}x_4^{8}x_5^{13}+
x_1x_2^{3}x_3^{6}x_4^{12}x_5^{9}+
\medskip
x_1x_2^{3}x_3^{8}x_4^{5}x_5^{14}\\
&\quad +
x_1x_2^{3}x_3^{8}x_4^{6}x_5^{13}+
x_1x_2^{3}x_3^{8}x_4^{12}x_5^{7}+
x_1x_2^{4}x_3^{3}x_4^{9}x_5^{14}+
\medskip
x_1x_2^{4}x_3^{3}x_4^{10}x_5^{13}\\
&\quad +
x_1x_2^{4}x_3^{3}x_4^{12}x_5^{11}+
x_1x_2^{4}x_3^{5}x_4^{10}x_5^{11}+
x_1x_2^{4}x_3^{9}x_4^{10}x_5^{7}+
\medskip
x_1x_2^{4}x_3^{10}x_4^{5}x_5^{11}\\
&\quad +
x_1x_2^{4}x_3^{10}x_4^{9}x_5^{7}+
x_1x_2^{5}x_3^{5}x_4^{10}x_5^{10}+
x_1x_2^{5}x_3^{6}x_4^{9}x_5^{10}+
\medskip
x_1x_2^{5}x_3^{8}x_4^{10}x_5^{7}\\
&\quad +
x_1x_2^{6}x_3^{2}x_4^{9}x_5^{13}+
x_1x_2^{6}x_3^{3}x_4^{8}x_5^{13}+
\medskip
x_1x_2^{6}x_3^{3}x_4^{9}x_5^{12}\\
&\quad + Sq^1\big(x_1x_2^{3}x_3^{5}x_4^{8}x_5^{13}+
x_1x_2^{3}x_3^{8}x_4^{5}x_5^{13}+
x_1x_2^{5}x_3^{5}x_4^{5}x_5^{14}+
\medskip
x_1x_2^{5}x_3^{5}x_4^{6}x_5^{13}\\
&\quad +
x_1x_2^{5}x_3^{5}x_4^{10}x_5^{9}+
x_1x_2^{5}x_3^{5}x_4^{12}x_5^{7}+
x_1x_2^{5}x_3^{6}x_4^{5}x_5^{13}+
\medskip
x_1x_2^{5}x_3^{6}x_4^{9}x_5^{9}\\
&\quad +
x_1x_2^{5}x_3^{8}x_4^{9}x_5^{7}+
x_1x_2^{6}x_3^{5}x_4^{9}x_5^{9}+
x_1x_2^{8}x_3^{5}x_4^{9}x_5^{7}+
\medskip
x_1^{4}x_2^{3}x_3^{5}x_4^{5}x_5^{13}\\
&\quad +
\medskip
x_1^{4}x_2^{5}x_3^{5}x_4^{9}x_5^{7}\big)\\
\end{array}$$

\newpage
$$ \begin{array}{ll}
&\quad + Sq^2\big(x_1x_2^{3}x_3^{6}x_4^{12}x_5^{7}+
x_1x_2^{6}x_3^{2}x_4^{9}x_5^{11}+
x_1x_2^{6}x_3^{3}x_4^{5}x_5^{14}+
\medskip
x_1x_2^{6}x_3^{3}x_4^{6}x_5^{13}\\
&\quad +
x_1x_2^{6}x_3^{3}x_4^{9}x_5^{10}+
x_1x_2^{6}x_3^{3}x_4^{12}x_5^{7}+
x_1x_2^{6}x_3^{6}x_4^{5}x_5^{11}+
\medskip
x_1^{2}x_2^{3}x_3^{5}x_4^{5}x_5^{14}\\
&\quad +
x_1^{2}x_2^{3}x_3^{5}x_4^{6}x_5^{13}+
x_1^{2}x_2^{3}x_3^{6}x_4^{5}x_5^{13}+
x_1^{2}x_2^{5}x_3^{5}x_4^{10}x_5^{7}+
\medskip
x_1^{2}x_2^{5}x_3^{6}x_4^{9}x_5^{7}\\
&\quad +
\medskip
x_1^{2}x_2^{6}x_3^{5}x_4^{9}x_5^{7}
\big)\\
&\quad + Sq^4\big(x_1x_2^{4}x_3^{3}x_4^{5}x_5^{14}+
x_1x_2^{4}x_3^{3}x_4^{6}x_5^{13}+
x_1x_2^{4}x_3^{3}x_4^{12}x_5^{7}+
x_1x_2^{4}x_3^{5}x_4^{10}x_5^{7}\\
&\quad +
x_1x_2^{4}x_3^{6}x_4^{5}x_5^{11}+
x_1x_2^{4}x_3^{6}x_4^{9}x_5^{7}+
x_1x_2^{10}x_3^{4}x_4^{5}x_5^{7}
\big) + Sq^{8}(x_1x_2^{6}x_3^{4}x_4^{5}x_5^{7}) \mod (\mathbb P_{31}^{\otimes 5})^{< \omega_{(2)}}.
\end{array}$$
This shows that $x_1x_2^{6}x_3^{3}x_4^{12}x_5^{9}$ is strictly inadmissible. Next, the following computation implies that $x_1^{3}x_2^{7}x_3^{12}x_4x_5^{8}$ is also strictly inadmissible:
$$ \begin{array}{ll}
x_1^{3}x_2^{7}x_3^{12}x_4x_5^{8} &= x_1^{2}x_2^{11}x_3^{5}x_4x_5^{12}+
x_1^{2}x_2^{11}x_3^{5}x_4^{4}x_5^{9}+
x_1^{2}x_2^{11}x_3^{5}x_4^{8}x_5^{5}+
\medskip
x_1^{2}x_2^{13}x_3^{3}x_4^{4}x_5^{9}\\
&\quad +
x_1^{2}x_2^{13}x_3^{3}x_4^{8}x_5^{5}+
x_1^{2}x_2^{13}x_3^{5}x_4x_5^{10}+
x_1^{2}x_2^{13}x_3^{9}x_4x_5^{6}+
\medskip
x_1^{3}x_2^{3}x_3^{5}x_4^{8}x_5^{12}\\
&\quad +
x_1^{3}x_2^{3}x_3^{9}x_4^{4}x_5^{12}+
x_1^{3}x_2^{3}x_3^{12}x_4x_5^{12}+
x_1^{3}x_2^{5}x_3^{5}x_4^{8}x_5^{10}+
\medskip
x_1^{3}x_2^{5}x_3^{6}x_4^{8}x_5^{9}\\
&\quad +
x_1^{3}x_2^{5}x_3^{9}x_4^{4}x_5^{10}+
x_1^{3}x_2^{5}x_3^{9}x_4^{8}x_5^{6}+
x_1^{3}x_2^{5}x_3^{10}x_4^{4}x_5^{9}+
\medskip
x_1^{3}x_2^{5}x_3^{10}x_4^{8}x_5^{5}\\
&\quad +
x_1^{3}x_2^{7}x_3^{5}x_4^{8}x_5^{8}+
x_1^{3}x_2^{7}x_3^{8}x_4x_5^{12}+
x_1^{3}x_2^{7}x_3^{8}x_4^{4}x_5^{9}+
\medskip
x_1^{3}x_2^{7}x_3^{8}x_4^{8}x_5^{5}\\
&\quad +
\medskip
x_1^{3}x_2^{7}x_3^{9}x_4^{4}x_5^{8} + Sq^{1}\big(x_1^{5}x_2^{3}x_3^{5}x_4^{8}x_5^{9}+
x_1^{5}x_2^{3}x_3^{9}x_4x_5^{12}+
x_1^{5}x_2^{7}x_3^{5}x_4x_5^{12}\\
&\quad +x_1^{5}x_2^{7}x_3^{5}x_4^{4}x_5^{9}+
x_1^{5}x_2^{7}x_3^{5}x_4^{8}x_5^{5}+
\medskip
x_1^{5}x_2^{7}x_3^{9}x_4x_5^{8}\big)\\
&\quad + Sq^{2}\big(x_1^{2}x_2^{11}x_3^{3}x_4^{4}x_5^{9}+
x_1^{2}x_2^{11}x_3^{3}x_4^{8}x_5^{5}+
x_1^{2}x_2^{11}x_3^{5}x_4x_5^{10}+
\medskip
x_1^{2}x_2^{11}x_3^{9}x_4x_5^{6}\\
&\quad +
x_1^{3}x_2^{3}x_3^{5}x_4^{8}x_5^{10}+
x_1^{3}x_2^{3}x_3^{6}x_4^{8}x_5^{9}+
x_1^{3}x_2^{3}x_3^{9}x_4^{2}x_5^{12}+
\medskip
x_1^{3}x_2^{3}x_3^{10}x_4x_5^{12}\\
&\quad +
x_1^{3}x_2^{7}x_3^{5}x_4^{2}x_5^{12}+
x_1^{3}x_2^{7}x_3^{5}x_4^{4}x_5^{10}+
x_1^{3}x_2^{7}x_3^{5}x_4^{8}x_5^{6}+
\medskip
x_1^{3}x_2^{7}x_3^{6}x_4x_5^{12}\\
&\quad +
x_1^{3}x_2^{7}x_3^{6}x_4^{4}x_5^{9}+
x_1^{3}x_2^{7}x_3^{6}x_4^{8}x_5^{5}+
x_1^{3}x_2^{7}x_3^{9}x_4^{2}x_5^{8}+
\medskip
x_1^{3}x_2^{7}x_3^{10}x_4x_5^{8}\\
&\quad +
x_1^{6}x_2^{3}x_3^{3}x_4^{8}x_5^{9}+
x_1^{6}x_2^{3}x_3^{9}x_4x_5^{10}+
x_1^{6}x_2^{7}x_3^{3}x_4^{4}x_5^{9}+
\medskip
x_1^{6}x_2^{7}x_3^{3}x_4^{8}x_5^{5}\\
&\quad +
x_1^{6}x_2^{7}x_3^{5}x_4x_5^{10}+
\medskip
x_1^{6}x_2^{7}x_3^{9}x_4x_5^{6}\big)\\
&\quad + Sq^{4}\big(x_1^{3}x_2^{5}x_3^{5}x_4^{2}x_5^{12}+
x_1^{3}x_2^{5}x_3^{5}x_4^{4}x_5^{10}+
x_1^{3}x_2^{5}x_3^{5}x_4^{8}x_5^{6}+
\medskip
x_1^{3}x_2^{5}x_3^{6}x_4x_5^{12}\\
&\quad +
x_1^{3}x_2^{5}x_3^{6}x_4^{4}x_5^{9}+
x_1^{3}x_2^{5}x_3^{6}x_4^{8}x_5^{5}+
x_1^{4}x_2^{7}x_3^{3}x_4^{4}x_5^{9}+
\medskip
x_1^{4}x_2^{7}x_3^{3}x_4^{8}x_5^{5}\\
&\quad +
x_1^{4}x_2^{7}x_3^{5}x_4x_5^{10}+
x_1^{4}x_2^{7}x_3^{9}x_4x_5^{6}+
x_1^{10}x_2^{5}x_3^{3}x_4^{4}x_5^{5}+
\medskip
x_1^{10}x_2^{5}x_3^{5}x_4x_5^{6}
\big)\\
&\quad+ Sq^{8}\big(x_1^{6}x_2^{5}x_3^{3}x_4^{4}x_5^{5}+
x_1^{6}x_2^{5}x_3^{5}x_4x_5^{6}\big) \mod (\mathbb P_{31}^{\otimes 5})^{< \omega_{(2)}}.
\end{array}$$
The lemma follows.
\end{proof}

\begin{lema}\label{bd2}
The following monomials are strictly inadmissible:

\begin{center}
\begin{tabular}{llll}
$x_1^{3}x_2^{7}x_3^{12}x_4^{8}x_5$, & $x_1^{7}x_2^{3}x_3^{12}x_4x_5^{8}$, & $x_1^{7}x_2^{3}x_3^{8}x_4^{12}x_5$, & $x_1^{7}x_2^{3}x_3^{12}x_4^{8}x_5$, \\
$x_1x_2^{6}x_3^{11}x_4^{4}x_5^{9}$, & $x_1^{7}x_2^{11}x_3^{4}x_4x_5^{8}$, & $x_1^{7}x_2^{11}x_3^{4}x_4^{8}x_5$, & $x_1x_2^{6}x_3^{11}x_4^{5}x_5^{8}$, \\
$x_1^{7}x_2^{7}x_3^{8}x_4^{8}x_5$, & $x_1^{3}x_2^{4}x_3^{9}x_4^{2}x_5^{13}$, & $x_1^{3}x_2^{5}x_3^{8}x_4^{2}x_5^{13}$, & $x_1^{3}x_2^{5}x_3^{9}x_4^{2}x_5^{12}$, \\
$x_1^{3}x_2^{4}x_3^{3}x_4^{13}x_5^{8}$, & $x_1^{3}x_2^{4}x_3^{3}x_4^{8}x_5^{13}$, & $x_1^{3}x_2^{4}x_3^{8}x_4^{3}x_5^{13}$, & $x_1^{3}x_2^{4}x_3^{9}x_4^{3}x_5^{12}$, \\
$x_1^{3}x_2^{4}x_3^{3}x_4^{12}x_5^{9}$, & $x_1^{3}x_2^{12}x_3^{3}x_4^{4}x_5^{9}$, & $x_1^{3}x_2^{5}x_3^{8}x_4^{3}x_5^{12}$, & $x_1^{3}x_2^{12}x_3^{3}x_4^{5}x_5^{8}$, \\
$x_1^{3}x_2^{4}x_3^{11}x_4^{4}x_5^{9}$, & $x_1^{3}x_2^{4}x_3^{11}x_4^{5}x_5^{8}$, & $x_1^{3}x_2^{4}x_3^{9}x_4^{6}x_5^{9}$, & $x_1^{3}x_2^{7}x_3^{8}x_4^{4}x_5^{9}$, \\
$x_1^{3}x_2^{4}x_3^{7}x_4^{9}x_5^{8}$, & $x_1^{3}x_2^{4}x_3^{9}x_4^{7}x_5^{8}$, & $x_1^{3}x_2^{4}x_3^{7}x_4^{8}x_5^{9}$, & $x_1^{3}x_2^{4}x_3^{8}x_4^{7}x_5^{9}$, \\
$x_1^{7}x_2^{3}x_3^{8}x_4^{4}x_5^{9}$, & $x_1^{7}x_2^{8}x_3^{3}x_4^{4}x_5^{9}$, & $x_1^{3}x_2^{5}x_3^{9}x_4^{6}x_5^{8}$, & $x_1^{3}x_2^{5}x_3^{8}x_4^{6}x_5^{9}$, \\
$x_1^{3}x_2^{7}x_3^{8}x_4^{5}x_5^{8}$, & $x_1^{3}x_2^{5}x_3^{8}x_4^{7}x_5^{8}$, & $x_1^{7}x_2^{3}x_3^{8}x_4^{5}x_5^{8}$, & $x_1^{7}x_2^{8}x_3^{3}x_4^{5}x_5^{8}$.
\end{tabular}%
\end{center}
\end{lema}

\begin{proof}
We prove the lemma for the monomials $x_1^{3}x_2^{4}x_3^{3}x_4^{8}x_5^{13},$ and $x_1^{7}x_2^{7}x_3^{8}x_4^{8}x_5.$ Computing these monomials is long and technical. The others can be proved by a similar technique. A direct computation shows that
$$ \begin{array}{ll}
x_1^{3}x_2^{4}x_3^{3}x_4^{8}x_5^{13} &=Sq^{1}\big(x_1^{3}x_2x_3^{3}x_4^{6}x_5^{17}+
x_1^{3}x_2x_3^{3}x_4^{9}x_5^{14}+
x_1^{3}x_2x_3^{3}x_4^{10}x_5^{13}+
\medskip
x_1^{3}x_2x_3^{6}x_4^{9}x_5^{11}\\
&\quad +
x_1^{3}x_2x_3^{10}x_4^{5}x_5^{11}+
x_1^{3}x_2^{3}x_3x_4^{6}x_5^{17}+
x_1^{3}x_2^{3}x_3x_4^{10}x_5^{13}+
\medskip
x_1^{3}x_2^{3}x_3x_4^{12}x_5^{11}\\
&\quad +
x_1^{3}x_2^{3}x_3x_4^{16}x_5^{7}+
x_1^{3}x_2^{3}x_3^{4}x_4^{3}x_5^{17}+
x_1^{3}x_2^{3}x_3^{8}x_4^{3}x_5^{13}+
\medskip
x_1^{3}x_2^{3}x_3^{8}x_4^{5}x_5^{11}\\
&\quad +
x_1^{3}x_2^{3}x_3^{8}x_4^{9}x_5^{7}+
x_1^{3}x_2^{4}x_3^{5}x_4^{5}x_5^{13}+
\medskip
x_1^{5}x_2^{5}x_3x_4^{5}x_5^{14}\big)\\
&\quad + Sq^{2}\big(x_1^{2}x_2x_3^{3}x_4^{9}x_5^{14}+
x_1^{2}x_2x_3^{3}x_4^{10}x_5^{13}+
x_1^{2}x_2x_3^{6}x_4^{9}x_5^{11}+
\medskip
x_1^{2}x_2x_3^{10}x_4^{5}x_5^{11}\\
&\quad +
x_1^{2}x_2^{3}x_3x_4^{9}x_5^{14}+
x_1^{2}x_2^{3}x_3x_4^{10}x_5^{13}+
x_1^{2}x_2^{3}x_3x_4^{12}x_5^{11}+
\medskip
x_1^{2}x_2^{3}x_3^{8}x_4^{3}x_5^{13}\\
&\quad +
x_1^{2}x_2^{3}x_3^{8}x_4^{5}x_5^{11}+
x_1^{2}x_2^{3}x_3^{8}x_4^{9}x_5^{7}+
x_1^{3}x_2^{5}x_3x_4^{6}x_5^{14}+
\medskip
x_1^{3}x_2^{5}x_3^{2}x_4^{5}x_5^{14}\\
&\quad +
x_1^{3}x_2^{6}x_3x_4^{5}x_5^{14}+
x_1^{5}x_2x_3^{6}x_4^{6}x_5^{11}+
x_1^{5}x_2^{2}x_3^{2}x_4^{9}x_5^{11}+
\medskip
x_1^{5}x_2^{2}x_3^{3}x_4^{5}x_5^{14}\\
&\quad +
x_1^{5}x_2^{2}x_3^{3}x_4^{6}x_5^{13}+
x_1^{5}x_2^{2}x_3^{6}x_4^{5}x_5^{11}+
x_1^{5}x_2^{2}x_3^{8}x_4^{3}x_5^{11}+
\medskip
x_1^{5}x_2^{3}x_3x_4^{6}x_5^{14}\\
&\quad +
x_1^{5}x_2^{3}x_3^{2}x_4^{6}x_5^{13}+
x_1^{5}x_2^{3}x_3^{2}x_4^{8}x_5^{11}+
x_1^{5}x_2^{3}x_3^{2}x_4^{10}x_5^{9}+
\medskip
x_1^{5}x_2^{3}x_3^{2}x_4^{12}x_5^{7}\\
&\quad +
x_1^{5}x_2^{3}x_3^{4}x_4^{3}x_5^{14}+
x_1^{5}x_2^{3}x_3^{4}x_4^{6}x_5^{11}+
x_1^{5}x_2^{3}x_3^{4}x_4^{10}x_5^{7}+
\medskip
x_1^{6}x_2^{3}x_3x_4^{5}x_5^{14}
\big)\\
&\quad + Sq^{4}\big(x_1^{3}x_2x_3^{6}x_4^{5}x_5^{12}+
x_1^{3}x_2x_3^{6}x_4^{6}x_5^{11}+
x_1^{3}x_2^{2}x_3^{2}x_4^{9}x_5^{11}+
\medskip
x_1^{3}x_2^{2}x_3^{3}x_4^{5}x_5^{14}\\
&\quad +
x_1^{3}x_2^{2}x_3^{3}x_4^{6}x_5^{13}+
x_1^{3}x_2^{2}x_3^{6}x_4^{5}x_5^{11}+
x_1^{3}x_2^{2}x_3^{8}x_4^{3}x_5^{11}+
\medskip
x_1^{3}x_2^{3}x_3x_4^{6}x_5^{14}\\
&\quad +
x_1^{3}x_2^{3}x_3^{2}x_4^{6}x_5^{13}+
x_1^{3}x_2^{3}x_3^{2}x_4^{8}x_5^{11}+
x_1^{3}x_2^{3}x_3^{2}x_4^{10}x_5^{9}+
\medskip
x_1^{3}x_2^{3}x_3^{2}x_4^{12}x_5^{7}\\
&\quad +
x_1^{3}x_2^{3}x_3^{4}x_4^{3}x_5^{14}+
x_1^{3}x_2^{3}x_3^{4}x_4^{6}x_5^{11}+
x_1^{3}x_2^{3}x_3^{4}x_4^{10}x_5^{7}+
\medskip
x_1^{3}x_2^{4}x_3x_4^{5}x_5^{14}\\
&\quad +
x_1^{3}x_2^{8}x_3^{4}x_4^{5}x_5^{7}+
\medskip
x_1^{4}x_2^{3}x_3x_4^{5}x_5^{14}\big) +Sq^{8}\big(x_1^{3}x_2^{4}x_3^{4}x_4^{5}x_5^{7}\big)\\
&\quad +x_1^{2}x_2x_3^{3}x_4^{12}x_5^{13}+
x_1^{2}x_2x_3^{5}x_4^{9}x_5^{14}+
x_1^{2}x_2x_3^{5}x_4^{10}x_5^{13}+
\medskip
x_1^{2}x_2x_3^{6}x_4^{9}x_5^{13}\\
&\quad +
x_1^{2}x_2x_3^{10}x_4^{5}x_5^{13}+
x_1^{2}x_2x_3^{12}x_4^{5}x_5^{11}+
x_1^{2}x_2^{5}x_3x_4^{9}x_5^{14}+
\medskip
x_1^{2}x_2^{5}x_3x_4^{10}x_5^{13}\\
&\quad +
x_1^{2}x_2^{5}x_3x_4^{12}x_5^{11}+
x_1^{2}x_2^{5}x_3^{8}x_4^{3}x_5^{13}+
x_1^{2}x_2^{5}x_3^{8}x_4^{5}x_5^{11}+
\medskip
x_1^{2}x_2^{5}x_3^{8}x_4^{9}x_5^{7}\\
&\quad +
x_1^{3}x_2x_3^{4}x_4^{9}x_5^{14}+
x_1^{3}x_2x_3^{4}x_4^{10}x_5^{13}+
x_1^{3}x_2x_3^{6}x_4^{8}x_5^{13}+
\medskip
x_1^{3}x_2x_3^{8}x_4^{6}x_5^{13}\\
&\quad +
x_1^{3}x_2^{2}x_3^{4}x_4^{9}x_5^{13}+
x_1^{3}x_2^{2}x_3^{5}x_4^{8}x_5^{13}+
x_1^{3}x_2^{3}x_3x_4^{12}x_5^{12}+
\medskip
x_1^{3}x_2^{3}x_3^{4}x_4^{12}x_5^{9}\\
&\quad +
x_1^{3}x_2^{3}x_3^{8}x_4^{4}x_5^{13}+
x_1^{3}x_2^{3}x_3^{8}x_4^{5}x_5^{12}+
x_1^{3}x_2^{4}x_3x_4^{9}x_5^{14}+
\medskip
x_1^{3}x_2^{4}x_3x_4^{10}x_5^{13}\\
&\quad +
x_1^{3}x_2^{4}x_3x_4^{12}x_5^{11}+
x_1^{3}x_2^{4}x_3^{2}x_4^{9}x_5^{13} \mod (\mathbb P_{31}^{\otimes 5})^{< \omega_{(2)}},
\end{array}$$
from which we deduce that $x_1^{3}x_2^{4}x_3^{3}x_4^{8}x_5^{13}$ is strictly inadmissible. Next, we have
$$ \begin{array}{ll}
x_1^{7}x_2^{7}x_3^{8}x_4^{8}x_5 & = Sq^{1}\big(x_1^{7}x_2^{13}x_3x_4^{4}x_5^{5}+
x_1^{7}x_2^{13}x_3x_4^{5}x_5^{4}+
x_1^{7}x_2^{13}x_3^{4}x_4x_5^{5}+
\medskip
x_1^{7}x_2^{13}x_3^{4}x_4^{5}x_5\\
&\quad +
x_1^{7}x_2^{13}x_3^{5}x_4x_5^{4}+
\medskip
x_1^{7}x_2^{13}x_3^{5}x_4^{4}x_5\big)\\
&\quad + Sq^{2}\big(x_1^{7}x_2^{11}x_3x_4^{4}x_5^{6}+
x_1^{7}x_2^{11}x_3x_4^{6}x_5^{4}+
x_1^{7}x_2^{11}x_3^{4}x_4x_5^{6}+
\medskip
x_1^{7}x_2^{11}x_3^{4}x_4^{6}x_5\\
&\quad + 
x_1^{7}x_2^{11}x_3^{6}x_4x_5^{4}+
x_1^{7}x_2^{11}x_3^{6}x_4^{4}x_5+
x_1^{7}x_2^{13}x_3^{2}x_4^{2}x_5^{5}+
\medskip
x_1^{7}x_2^{13}x_3^{2}x_4^{5}x_5^{2}\\
&\quad + 
x_1^{7}x_2^{13}x_3^{5}x_4^{2}x_5^{2}+
x_1^{7}x_2^{14}x_3x_4^{2}x_5^{5}+
x_1^{7}x_2^{14}x_3x_4^{3}x_5^{4}+
\medskip
x_1^{7}x_2^{14}x_3^{2}x_4x_5^{5}\\
&\quad + 
x_1^{7}x_2^{14}x_3^{3}x_4x_5^{4}+
x_1^{7}x_2^{14}x_3^{3}x_4^{4}x_5+
\medskip
x_1^{7}x_2^{14}x_3^{4}x_4^{3}x_5\big)\\
\end{array}$$

\newpage
$$ \begin{array}{ll}
&\quad + Sq^{4}\big(x_1^{5}x_2^{11}x_3x_4^{4}x_5^{6}+
x_1^{5}x_2^{11}x_3x_4^{6}x_5^{4}+
x_1^{5}x_2^{11}x_3^{4}x_4x_5^{6}+
\medskip
x_1^{5}x_2^{11}x_3^{4}x_4^{6}x_5\\
&\quad + 
x_1^{5}x_2^{11}x_3^{6}x_4x_5^{4}+
x_1^{5}x_2^{11}x_3^{6}x_4^{4}x_5+
x_1^{5}x_2^{13}x_3^{2}x_4^{2}x_5^{5}+
\medskip
x_1^{5}x_2^{13}x_3^{2}x_4^{5}x_5^{2}\\
&\quad + 
x_1^{5}x_2^{13}x_3^{5}x_4^{2}x_5^{2}+
x_1^{5}x_2^{14}x_3x_4^{2}x_5^{5}+
x_1^{5}x_2^{14}x_3x_4^{3}x_5^{4}+
\medskip
x_1^{5}x_2^{14}x_3^{2}x_4x_5^{5}\\
&\quad + 
x_1^{5}x_2^{14}x_3^{3}x_4x_5^{4}+
x_1^{5}x_2^{14}x_3^{3}x_4^{4}x_5+
x_1^{5}x_2^{14}x_3^{4}x_4^{3}x_5+
\medskip
x_1^{11}x_2^{7}x_3x_4^{4}x_5^{4}\\
&\quad + 
x_1^{11}x_2^{7}x_3^{4}x_4x_5^{4}+
\medskip
x_1^{11}x_2^{7}x_3^{4}x_4^{4}x_5\big)\\
&\quad  + Sq^{8}\big(x_1^{7}x_2^{7}x_3x_4^{4}x_5^{4}+
x_1^{7}x_2^{7}x_3^{4}x_4x_5^{4}+
\medskip
x_1^{7}x_2^{7}x_3^{4}x_4^{4}x_5\big)\\
&\quad +x_1^{5}x_2^{11}x_3x_4^{4}x_5^{10}+
x_1^{5}x_2^{11}x_3x_4^{6}x_5^{8}+
x_1^{5}x_2^{11}x_3x_4^{8}x_5^{6}+
\medskip
x_1^{5}x_2^{11}x_3x_4^{10}x_5^{4}\\
&\quad + 
x_1^{5}x_2^{11}x_3^{4}x_4x_5^{10}+
x_1^{5}x_2^{11}x_3^{4}x_4^{10}x_5+
x_1^{5}x_2^{11}x_3^{6}x_4x_5^{8}+
\medskip
x_1^{5}x_2^{11}x_3^{6}x_4^{8}x_5\\
&\quad + 
x_1^{5}x_2^{11}x_3^{8}x_4x_5^{6}+
x_1^{5}x_2^{11}x_3^{8}x_4^{6}x_5+
x_1^{5}x_2^{11}x_3^{10}x_4x_5^{4}+
\medskip
x_1^{5}x_2^{11}x_3^{10}x_4^{4}x_5\\
&\quad + 
x_1^{5}x_2^{13}x_3^{2}x_4^{2}x_5^{9}+
x_1^{5}x_2^{13}x_3^{2}x_4^{9}x_5^{2}+
x_1^{5}x_2^{13}x_3^{9}x_4^{2}x_5^{2}+
\medskip
x_1^{5}x_2^{14}x_3x_4^{2}x_5^{9}\\
&\quad + 
x_1^{5}x_2^{14}x_3x_4^{3}x_5^{8}+
x_1^{5}x_2^{14}x_3^{2}x_4x_5^{9}+
x_1^{5}x_2^{14}x_3^{3}x_4x_5^{8}+
\medskip
x_1^{5}x_2^{14}x_3^{3}x_4^{8}x_5\\
&\quad + 
x_1^{5}x_2^{14}x_3^{8}x_4^{3}x_5+
x_1^{7}x_2^{7}x_3x_4^{8}x_5^{8}+
x_1^{7}x_2^{7}x_3^{8}x_4x_5^{8} \mod (\mathbb P_{31}^{\otimes 5})^{< \omega_{(2)}}.
\end{array}$$
These equalities show that $x_1^{7}x_2^{7}x_3^{8}x_4^{8}x_5$ is also strictly inadmissible. The lemma is proven. 
\end{proof}

Now, for each $1\leq l\leq 5,$ we put
$$ \begin{array}{ll}
\medskip
A(30) &= \{x_l\mathsf{q}_{(l,\,5)}(X):\ X\in (\mathscr {C}^{\otimes 4}_{30})^{>0},\} ,\\
\medskip
A(28) &= \{x_l^{3}\mathsf{q}_{(l,\,5)}(X):\ X\in (\mathscr {C}^{\otimes 4}_{28})^{>0}\},\\
\medskip
A(24) &= \{x_l^{7}\mathsf{q}_{(l,\,5)}(X):\ X\in (\mathscr {C}^{\otimes 4}_{24})^{>0}\},\\
\medskip
A(16) &= \{x_l^{15}\mathsf{q}_{(l,\,5)}(X):\ X\in (\mathscr {C}^{\otimes 4}_{16})^{>0}\},
\end{array}$$
and $A(\omega_{(i)}) = (\mathbb P^{\otimes 5})^{\omega_{(i)}^{>0}}\cap (A(30)\cup A(28)\cup A(24)\cup A(16)),$ where $$(\mathbb P^{\otimes 5})^{\omega_{(i)}^{>0}} := \langle \{X\in (\mathbb P^{\otimes 5})^{>0}:\ \omega(X)=\omega_{(i)}\} \rangle,\ i = 2, 3.$$
It is not hard to show that $A(\omega_{(2)}) = \{X_j|\ 1\leq j\leq 68\}$ and $\widetilde{\Phi^{>0}}((\mathscr C^{\otimes 4}_{31})^{\omega_{(2)}})\cap A(\omega_{(2)}) \neq \emptyset.$ We put $B(\omega_{(2)}):= (\widetilde{\Phi^{>0}}((\mathscr C^{\otimes 4}_{31})^{\omega_{(2)}}) \setminus A(\omega_{(2)})) \cup \{X_j|\ 154\leq j\leq 215\},$ where a direct computation shows that $\widetilde{\Phi^{>0}}((\mathscr C^{\otimes 4}_{31})^{\omega_{(2)}}) \setminus A(\omega_{(2)}) = \{X_j|\ 69\leq j\leq 153\}.$ Here, the monomials $X_j,$ for all $j,\ 1\leq j\leq 215,$ are determined as in the Appendix (see Subsect.\ref{s31}). 
 
\begin{propo}\label{md2}
$(Q^{\otimes 5}_{31})^{\omega_{(2)}^{>0}}$ is an $\mathbb Z_2$-vector space of dimension $215$ with the basis $$S = [A(\omega_{(2)}) \cup B(\omega_{(2)})]_{\omega_{(2)}}.$$
\end{propo}

\begin{proof}
Let $\omega:= (2,2,2)$ be the weight vector of degree $14$ and let $X$ be an admissible monomial in $(\mathscr C_{31}^{\otimes 5})^{\omega_{(2)}^{>0}}.$ Then $X = x_ix_jx_kY^2$ with $1\leq i<j<k\leq 5,$ and $Y\in \mathbb P_{14}^{\otimes 5}.$ Since $X\in  (\mathscr C_{31}^{\otimes 5})^{\omega_{(2)}}$, due to Theorem \ref{dlKS}, $Y\in (\mathscr C_{14}^{\otimes 5})^{\omega}.$ By a direct computation, we see that if $Z\in (\mathscr C^{\otimes 5}_{14})^{\omega}$ and the monomial $T:=x_ix_jx_kZ^2\neq X_m,$ for all $m,\ 1\leq m\leq 215,$ then either $T$ is one of the monomials as given in Lemmas \ref{bd1}, \ref{bd2}, or $T$ is of the form $Z_1Y_1^{4},$ where $Z_1$ is one of the following inadmissible monomials: $x_1x_2^{2}x_3^{2}x_4x_5,\ x_1^{3}x_2^{2}x_3x_5,\ x_1^{3}x_2^{2}x_3x_4,\ x_1^{3}x_2^{2}x_4x_5, \ x_1^{3}x_3^{2}x_4x_5,\ x_2^{3}x_3^{2}x_4x_5,\ x_2^{2}x_3x_4x_5^{3},\ x_2^{2}x_3x_4^{3}x_5,\ x_2^{2}x_3^{3}x_4x_5,\\ x_1^{2}x_3x_4x_5^{3},\ x_1^{2}x_3x_4^{3}x_5,\ x_1^{2}x_3^{3}x_4x_5,\ x_1^{2}x_2x_4x_5^{3},\ x_1^{2}x_2x_4^{3}x_5,\ x_1^{2}x_2x_3x_5^{3},\ x_1^{2}x_2x_3x_4^{3},\ x_1^{2}x_2x_3^{3}x_5,\ x_1^{2}x_2x_3^{3}x_4,\\ x_1^{2}x_2^{3}x_4x_5,\ x_1^{2}x_2^{3}x_3x_5,\ x_1^{2}x_2^{3}x_3x_4,\ x_1^{2}x_2x_3x_4x_5^{2},\ x_1^{2}x_2x_3x_4^{2}x_5,\ x_1^{2}x_2x_3^{2}x_4x_5.$ Then, by Theorem \ref{dlKS}, $T$ is inadmissible. Since $X = x_ix_jx_kY^2$ is admissible, we have $X = X_m$ for some $m,\ 1\leq m\leq 215.$ Thus, the statement will be confirmed if the set $S$ is linearly independent in $(Q^{\otimes 5}_{31})^{\omega_{(2)}}.$ Indeed, assume that there is a linear relation $\mathcal M:= \sum_{1\leq m\leq 215}\gamma_m X_m \equiv_{\omega_{(2)}} 0,$ wherenever $\gamma_m\in \mathbb Z_2.$ For each $(l, \mathscr L)\in \mathcal N_5,$ we explicitly determine $\mathsf{p}_{(l, \mathscr L)}(\mathcal M)$ in terms of  admissible monomials modulo ($\overline{\mathscr A_2}(\mathbb P_{31}^{\otimes 4})^{>0}$). Then, from the relations $\mathsf{p}_{(l, \mathscr L)}(\mathcal M)\equiv_{\omega_{(2)}} 0,$ for all $(l, \mathscr L)\in \mathcal N_5,\ \ell(\mathscr L) > 0,$ and some other simple computations, we find that $\gamma_m = 0$, $\forall m,\,  1\leq m\leq 215.$ The proposition is proved.
\end{proof}

\begin{rems}\label{nx}
According to Sum \cite{N.S1}, we have a direct summand decomposition of the $\mathbb Z_2$-vector spaces: $ Q^{\otimes 4}_{31} = (Q^{\otimes 4}_{31})^{0}\bigoplus (Q^{\otimes 4}_{31})^{\omega_{(1)}^{>0}}\bigoplus (Q^{\otimes 4}_{31})^{\omega_{(2)}^{>0}}.$ Since $(Q^{\otimes 4}_{31})^{\omega_{(1)}^{>0}} = \langle [x_1x_2^2x_3^4x_4^{24}] \rangle$ and the powers of $x_1x_2^2x_3^4x_4^{24}$ do not satisfy \eqref{dk} in Conjecture \ref{gtSum}, $\widetilde{\Phi^{>0}}((\mathscr{C}^{\otimes 4}_{31})^{\omega_{(1)}}) = \emptyset.$ By a simple computation, one gets $\widetilde{\Phi^{0}}((\mathscr{C}^{\otimes 4}_{31})^{\omega_{(2)}})\subseteq (\mathscr C^{\otimes 5}_{31})^{\omega_{(2)}^{0}}$, and $$(\mathscr C^{\otimes 5}_{31})^{\omega_{(2)}^{>0}} = \widetilde{\Phi^{>0}}((\mathscr{C}^{\otimes 4}_{n})^{\omega_{(2)}})\cup \big(A(\omega_{(2)})\setminus \{X_j|\ 1\leq j\leq 39\}\big)\cup \{X_j|\ 154\leq j\leq 215\},$$
which implies that $\widetilde{\Phi^{>0}}((\mathscr{C}^{\otimes 4}_{n})^{\omega_{(2)}})\subseteq (\mathscr C^{\otimes 5}_{31})^{\omega_{(2)}^{>0}}=A(\omega_{(2)}) \cup B(\omega_{(2)}).$ Further, we see that if the weight vector $\omega\neq \omega_{(j)},\ j = 1, 2,$ then $(\mathscr C^{\otimes 4}_{31})^{\omega} = \emptyset.$ So, by the proof of Proposition \ref{md3} below, it follows that $\widetilde{\Phi_*}((\mathscr{C}^{\otimes 4}_{31})^{\omega_{(3)}}) = \emptyset.$ Combining these data, we may conclude that Conjecture \ref{gtSum} holds for $t = 5$ and degree $31.$
\end{rems}

\begin{propo}\label{md3}
$(Q^{\otimes 5}_{31})^{\omega_{(3)}^{>0}}$ is an $\mathbb Z_2$-vector space of dimension $70$ with the basis $$ \mathcal S = [\{Y_j|\ 1\leq j\leq 50 \} \cup A(\omega_{(3)})]_{\omega_{(3)}},$$
where $A(\omega_{(3)}) = \{Y_j|\ 51\leq j\leq 70\}$ and the monomials $Y_j,\ 1\leq j\leq 70,$ are given as in the Subsect.\ref{s32} of the Appendix.
\end{propo}

\begin{proof}
Firstly, let $i, j, k, l,  m$ be five distinct integers such that $1\leq i, j, k, l, m\leq 5.$ Then, it is easy to see that the following monomials are inadmissible:
\begin{itemize}
\item[(i)] $T_1:=x_j^{2}x_k^{7}x_l^{7}x_m^{7},\ T_2:=x_j^{3}x_k^{6}x_l^{7}x_m^{7},\ T_3:=x_ix_j^{2}x_k^{6}x_l^{7}x_m^{7},$
\item[(ii)] $T_4:=x_i^{2}x_j^{2}x_k^{5}x_l^{7}x_m^{7},\ T_5:=x_i^{2}x_j^{3}x_k^{4}x_l^{7}x_m^{7},\ T_6:=x_i^{3}x_j^{3}x_k^{4}x_l^{6}x_m^{7},$
\item[(iii)] $T_7:=x_i^{7}x_jx_k^{6}x_l^{3}x_m^{6}\ T_8:=x_i^{7}x_j^{3}x_k^{6}x_lx_m^{6},\ T_9:=x_i^{7}x_j^{6}x_kx_l^{3}x_m^{6},\ T_{10}:=x_i^{7}x_j^{6}x_k^{3}x_lx_m^{6},$ with $j < k < l,$
\item[(iv)] $T_{11}:=x_i^{3}x_j^{2}x_k^{5}x_l^{6}x_m^{7},\ T_{12}:=x_i^{3}x_j^{6}x_k^{5}x_l^{3}x_m^{6}$ with $j < k,$
\item[(v)] $T_{13}:=x_1^{3}x_2^{6}x_3^{5}x_4^{2}x_5^{7},\ T_{14}:=x_1^{3}x_2^{6}x_3^{5}x_4^{7}x_5^{2},\ T_{15}:=x_1^{3}x_2^{6}x_3^{7}x_4^{5}x_5^{2},\\ T_{16}:=x_1^{3}x_2^{7}x_3^{6}x_4^{5}x_5^{2},\ T_{17}:=x_1^{7}x_2^{3}x_3^{6}x_4^{5}x_5^{2},\ T_{18}:=x_1^{7}x_2^{6}x_3^{3}x_4^{5}x_5^{2}.$

\end{itemize}
Consider the weight vector $\overline{\omega}:= (4,3,1)$ of degree $14.$ Suppose that $Y$ is an admissible monomial in $(\mathscr C_{31}^{\otimes 5})^{\omega_{(3)}^{>0}}.$ Then $Y = x_ax_bx_cZ^2$ with $1\leq a<b<c\leq 5,$ and $Z\in \mathbb P_{14}^{\otimes 5}.$ Since $Y$ is admissible, according to Theorem \ref{dlKS}, $Z$ is admissible and $\omega(Z) = \overline{\omega}.$ A direct computation shows that if $Z$ is an admissible monomial in $\mathbb P_{14}^{\otimes 5}$ such that the monomial $x_ax_bx_cZ_1^2\in (\mathbb P_{31}^{\otimes 5})^{\leq \omega_{(3)}}$ and $x_ax_bx_cZ_1^2\neq Y_j,$ for all $j,\ 1\leq j\leq 70,$ then either $x_ax_bx_cZ_1^2$ of the form $T_dW^{2^{s}},$ where $s$ is a suitable integer and $T_d$ is one of the above inadmissible monomials, $1\leq d\leq 18,$ or $x_ax_bx_cZ_1^2$ is one of the following monomials:

\begin{center}
\begin{tabular}{lllr}
$x_1x_2^{3}x_3^{6}x_4^{14}x_5^{7}$, & $x_1x_2^{3}x_3^{14}x_4^{6}x_5^{7}$, & $x_1x_2^{3}x_3^{14}x_4^{7}x_5^{6}$, & \multicolumn{1}{l}{$x_1^{3}x_2x_3^{6}x_4^{14}x_5^{7}$,} \\
$x_1^{3}x_2x_3^{14}x_4^{6}x_5^{7}$, & $x_1^{3}x_2x_3^{14}x_4^{7}x_5^{6}$, & $x_1^{3}x_2^{5}x_3^{2}x_4^{14}x_5^{7}$, & \multicolumn{1}{l}{$x_1^{3}x_2^{5}x_3^{14}x_4^{2}x_5^{7}$,} \\
$x_1^{3}x_2^{5}x_3^{14}x_4^{7}x_5^{2}$, & $x_1^{3}x_2^{13}x_3^{2}x_4^{6}x_5^{7}$, & $x_1^{3}x_2^{13}x_3^{2}x_4^{7}x_5^{6}$, & \multicolumn{1}{l}{$x_1^{3}x_2^{13}x_3^{6}x_4^{2}x_5^{7}$,} \\
$x_1^{3}x_2^{13}x_3^{6}x_4^{7}x_5^{2}$, & $x_1^{3}x_2^{13}x_3^{7}x_4^{2}x_5^{6}$, & $x_1^{3}x_2^{13}x_3^{7}x_4^{6}x_5^{2}$, & \multicolumn{1}{l}{$x_1^{3}x_2^{5}x_3^{14}x_4^{3}x_5^{6}$,} \\
$x_1^{3}x_2^{5}x_3^{14}x_4^{6}x_5^{3}$, & $x_1^{3}x_2^{13}x_3^{6}x_4^{3}x_5^{6}$, & $x_1^{3}x_2^{13}x_3^{6}x_4^{6}x_5^{3}$, & \multicolumn{1}{l}{$x_1^{3}x_2^{5}x_3^{6}x_4^{11}x_5^{6}$,} \\
$x_1^{3}x_2^{5}x_3^{6}x_4^{10}x_5^{7}$, & $x_1^{3}x_2^{5}x_3^{10}x_4^{6}x_5^{7}$, & $x_1^{3}x_2^{5}x_3^{10}x_4^{7}x_5^{6}.$ &  
\end{tabular}%
\end{center}
These monomials are strictly inadmissible. Indeed, we will prove this statement for the monomials $x_1x_2^{3}x_3^{6}x_4^{14}x_5^{7},$ and $x_1^{3}x_2^{13}x_3^{6}x_4^{3}x_5^{6}.$ The others can be computed by a similar technique. It is easily seen that $\omega(x_1x_2^{3}x_3^{6}x_4^{14}x_5^{7}) = \omega(x_1^{3}x_2^{13}x_3^{6}x_4^{3}x_5^{6}) = \omega_{(3)}.$ Basing Cartan's formula, it is not too difficult to indicate that
$$ \begin{array}{ll}
x_1x_2^{3}x_3^{6}x_4^{14}x_5^{7}&= Sq^{1}\big(x_1^{2}x_2^{3}x_3^{5}x_4^{7}x_5^{13}+
x_1^{2}x_2^{3}x_3^{5}x_4^{13}x_5^{7}+
x_1^{2}x_2^{5}x_3^{5}x_4^{7}x_5^{11}+
\medskip
x_1^{2}x_2^{5}x_3^{5}x_4^{11}x_5^{7}\\
&\quad +
x_1^{2}x_2^{5}x_3^{9}x_4^{7}x_5^{7}+
x_1^{2}x_2^{9}x_3^{5}x_4^{7}x_5^{7}+
\medskip
x_1^{5}x_2x_3^{7}x_4^{8}x_5^{9}\big)\\
&\quad +Sq^{2}\big(x_1x_2^{3}x_3^{3}x_4^{9}x_5^{13}+
x_1x_2^{3}x_3^{3}x_4^{13}x_5^{9}+
x_1x_2^{3}x_3^{5}x_4^{7}x_5^{13}+
\medskip
x_1x_2^{3}x_3^{5}x_4^{9}x_5^{11}\\
&\quad +
x_1x_2^{3}x_3^{5}x_4^{11}x_5^{9}+
x_1x_2^{3}x_3^{5}x_4^{13}x_5^{7}+
x_1x_2^{5}x_3^{5}x_4^{7}x_5^{11}+
\medskip
x_1x_2^{5}x_3^{5}x_4^{11}x_5^{7}\\
&\quad +
x_1x_2^{5}x_3^{9}x_4^{7}x_5^{7}+
\medskip
x_1x_2^{9}x_3^{5}x_4^{7}x_5^{7}\big)\\
&\quad + Sq^{4}\big(x_1x_2^{3}x_3^{3}x_4^{9}x_5^{11}+
x_1x_2^{3}x_3^{3}x_4^{11}x_5^{9}+
x_1x_2^{5}x_3^{5}x_4^{7}x_5^{9}+
x_1x_2^{5}x_3^{5}x_4^{9}x_5^{7}+
\medskip
x_1x_2^{6}x_3^{6}x_4^{7}x_5^{7}\big)\\
\medskip
&\quad  + x_1x_2^{3}x_3^{6}x_4^{7}x_5^{14} \mod (\mathbb P^{\otimes _5}_{31})^{<\omega_{(3)}},\\
x_1^{3}x_2^{13}x_3^{6}x_4^{3}x_5^{6} &= Sq^{1}\big(x_1^{3}x_2^{7}x_3^{5}x_4^{5}x_5^{10}+
x_1^{3}x_2^{7}x_3^{5}x_4^{6}x_5^{9}+
x_1^{3}x_2^{7}x_3^{6}x_4^{5}x_5^{9}+
\medskip
x_1^{3}x_2^{7}x_3^{6}x_4^{9}x_5^{5}\\
&\quad +
x_1^{3}x_2^{7}x_3^{9}x_4^{6}x_5^{5}+
x_1^{3}x_2^{11}x_3^{5}x_4^{6}x_5^{5}+
x_1^{3}x_2^{14}x_3^{3}x_4^{5}x_5^{5}+
\medskip
x_1^{6}x_2^{11}x_3^{5}x_4^{3}x_5^{5}\big)\\
&\quad  + Sq^{2}\big(x_1^{3}x_2^{11}x_3^{5}x_4^{5}x_5^{5}+
x_1^{3}x_2^{11}x_3^{6}x_4^{3}x_5^{6}+
x_1^{3}x_2^{13}x_3^{5}x_4^{3}x_5^{5}+
\medskip
x_1^{5}x_2^{7}x_3^{3}x_4^{5}x_5^{9}\\
&\quad +
x_1^{5}x_2^{7}x_3^{3}x_4^{9}x_5^{5}+
x_1^{5}x_2^{7}x_3^{5}x_4^{3}x_5^{9}+
x_1^{5}x_2^{7}x_3^{9}x_4^{3}x_5^{5}+
x_1^{5}x_2^{9}x_3^{3}x_4^{3}x_5^{9}+
\medskip
x_1^{5}x_2^{11}x_3^{3}x_4^{5}x_5^{5}\big)\\
&\quad  + Sq^{4}\big(x_1^{3}x_2^{7}x_3^{3}x_4^{5}x_5^{9}+
x_1^{3}x_2^{7}x_3^{3}x_4^{9}x_5^{5}+
x_1^{3}x_2^{7}x_3^{5}x_4^{3}x_5^{9}+
x_1^{3}x_2^{7}x_3^{9}x_4^{3}x_5^{5}+
\medskip
x_1^{3}x_2^{9}x_3^{3}x_4^{3}x_5^{9}\\
&\quad +
x_1^{3}x_2^{9}x_3^{5}x_4^{5}x_5^{5}+
x_1^{3}x_2^{11}x_3^{3}x_4^{5}x_5^{5}+
\medskip
x_1^{3}x_2^{11}x_3^{5}x_4^{3}x_5^{5}\big) +  x_1^{3}x_2^{13}x_3^{3}x_4^{6}x_5^{6} \mod (\mathbb P^{\otimes _5}_{31})^{<\omega_{(3)}}.
\end{array}$$
The above computations, together with Theorem \ref{dlKS}, allow us to claim that $x_ax_bx_cZ_1^2$ is inadmissible. Since $Y = x_ax_bx_cZ^2$ is admissible, we have $Y = Y_j$ for some $j,\ 1\leq j\leq 70.$ We now prove the set $\mathcal S$ is linearly independent in $(Q^{\otimes 5}_{31})^{\omega_{(3)}}.$ Suppose  there is a linear relation $\mathcal M:= \sum_{1\leq j\leq 70}\gamma_j Y_j \equiv 0,$ in which $\gamma_j\in \mathbb Z_2$ for $1\leq j\leq 70.$ By direct calculations using the relations $\mathsf{p}_{(l, \mathscr L)}(\mathcal M)\equiv 0,$ with $\ell(\mathscr L)\leq 2,$ it may be concluded that  $\gamma_j = 0,$ for all $j.$ The proof is completed.
\end{proof}

\subsection{Proof of Theorem \ref{dlc3}}

The proof of the theorem is proceeded through the following three steps.

\medskip

{\it Step 1.} Firstly, a direct computation shows that if $X\in \mathscr {C}^{\otimes 5}_{32},$  then $\omega(X)$ is one of the sequences $\omega'_{(j)},$ for $1\leq j\leq 7,$ where $ \omega'_{(4)}:= (2,1,1,3), \ \omega'_{(5)}:= (2,3,2,2), \ \omega'_{(6)}:= (2,3,4,1),$ and $\omega'_{(7)}:= (4,2,4,1).$ Indeed,  it is easy to see that $x_1^{31}x_2$ is the minimal spike in $\mathbb P_{32}^{\otimes 5}$ and $\omega(x_1^{31}x_2) = \omega'_{(1)}.$ Since $X\in \mathscr C_{32}^{\otimes 5}$ and $\deg(X) = 32$ even, by Theorem \ref{dlSin}, either $\omega_1(X) = 2$ or $\omega_1(X) = 4.$ If $\omega_1(X) = 2,$ then $X = x_ix_jY^2$ with $1\leq i < j\leq 5$ and $Y\in \mathbb P_{15}^{\otimes 5}.$ Since $X\in \mathscr {C}^{\otimes 5}_{32},$ by Theorem \ref{dlKS}, $Y\in \mathscr {C}^{\otimes 5}_{15}.$ According to Sum \cite{N.S}, we have $\omega(Y)\in \{(1,1,1,1), (1,1,3), (3,2,2), (3,4,1)\}$ and so $\omega(X)\in \{\omega'_{(j)}\}_{j = 1, 4, 5, 6}.$ In the case in which $\omega_1(X) = 4,$ we have seen that by Theorem \ref{dlKS}, $X$ has the form $x_ix_jx_kx_{\ell}T^{2}$ with $T\in \mathscr {C}^{\otimes 5}_{14}$ and $1\leq i<j<k<\ell\leq 5.$ From the calculations of Ly-Tin \cite{L.T}, we get $\omega(T)\in \{(2,2,2), (2,4,1), (4,3,1)\}$ and therefore $\omega(X)\in \{\omega'_{(j)}\}_{j = 2, 3, 7}.$  

\medskip

{\it Step 2.} Next, we will show that if $\omega(X)\in \{\omega'_{(j)}\}_{4\leq j\leq 7},$ then the space $(Q^{\otimes 5}_{32})^{\omega(X)}$ is trivial. Indeed, from the works of Peterson \cite{F.P1}, Kameko \cite{M.K} and Sum \cite{N.S1}, we deduce that
$$ \dim (Q_{32}^{\otimes s})^{>0} = \left\{\begin{array}{ll}
0,& \mbox{if } s = 1,\\
3,& \mbox{if } s = 2,\\
5,& \mbox{if } s = 3,\\
57,& \mbox{if } s = 4.
\end{array}\right.$$
The facts above show that
$$ \dim (Q^{\otimes 5}_{32})^{0}= \sum_{1\leq  s\leq 4}\binom{5}{s}\dim (Q_{32}^{\otimes\, s})^{>0} =  3\times\binom{5}{2} + 5\times\binom{5}{3} + 57 \times \binom{5}{4}=365.$$
Furthermore, by a simple computation, we get $(\mathscr C_{32}^{\otimes 5})^{0} = \bigcup_{1\leq l\leq 5}\mathsf{q}_{(l,\,5)}(\mathscr C_{32}^{\otimes 4}) = \bigcup_{1\leq j\leq 3}(\mathscr C_{32}^{\otimes 5})^{(\omega'_{(j)})^{0}},$ where $|(\mathscr C_{32}^{\otimes 5})^{(\omega'_{(1)})^{0}}| = 115,$  $|(\mathscr C_{32}^{\otimes 5})^{(\omega'_{(2)})^{0}}| = 175$ and $|(\mathscr C_{32}^{\otimes 5})^{(\omega'_{(3)})^{0}}| = 75.$ Thus, it may be concluded that
$$ \dim (Q_{32}^{\otimes 5})^{(\omega'_{(j)})^{0}} =\left\{ \begin{array}{ll}
115&\mbox{if $j = 1$},\\
175&\mbox{if $j = 2$},\\
75&\mbox{if $j = 3$},\\
0&\mbox{if $j = 4, 5, 6$}.
\end{array}\right.$$
This means that we need only to prove $(Q^{\otimes 5}_{32})^{(\omega'_{(j)})^{>0}} = 0$ for $j\in \{4, 5, 6\}.$ Recall that if $\omega(X)$ is one of the sequences $\omega'_{(j)}$ for $j = 4, 5, 6,$ then $X = x_ix_jY^2$ with $1\leq i < j\leq 5$ and $Y\in \mathscr C_{15}^{\otimes 5}.$ When $\omega(X) = \omega'_{(4)},$ we see that $X$ is one of the following monomials:
$$ x_1^{3}x_2^{4}x_3^{8}x_4^{8}x_5^{9}, \ \ x_1^{3}x_2^{4}x_3^{8}x_4^{9}x_5^{8}, \ \ x_1^{3}x_2^{4}x_3^{9}x_4^{8}x_5^{8}, \ \ x_1^{3}x_2^{5}x_3^{8}x_4^{8}x_5^{8}.$$
Due to Cartan's formula, one can easily see that
$$ \begin{array}{ll}
\medskip
x_1^{3}x_2^{4}x_3^{9}x_4^{8}x_5^{8} &= Sq^{1}(x_1^{3}x_2x_3^{3}x_4^{8}x_5^{16}) + Sq^{2}(x_1^{5}x_2^{2}x_3^{3}x_4^{4}x_5^{16}) + Sq^{4}(x_1^{3}x_2^{2}x_3^{3}x_4^{4}x_5^{16} + x_1^{3}x_2^{8}x_3^{5}x_4^{4}x_5^{8}) \\
&\quad + Sq^{8}(x_1^{3}x_2^{4}x_3^{5}x_4^{4}x_5^{8}) \mod (\mathbb P_{32}^{\otimes 5})^{< \omega'_{(4)}}
\end{array}$$
which implies that $x_1^{3}x_2^{4}x_3^{9}x_4^{8}x_5^{8}$ is $\omega'_{(4)}$-decomposable. Similarly, the remaining monomials are also $\omega'_{(4)}$-decomposable.\\[1mm]
Now if $\omega(X) = \omega'_{(5)},$ then $X$ is  a permutation of the following monomials:

\begin{center}
\begin{tabular}{llll}
$x_1x_2x_3^{2}x_4^{14}x_5^{14}$, & $x_1x_2x_3^{6}x_4^{10}x_5^{14}$, & $x_1x_2^{2}x_3^{2}x_4^{13}x_5^{14}$, & $x_1x_2^{2}x_3^{3}x_4^{12}x_5^{14}$, \\
$x_1x_2^{2}x_3^{5}x_4^{10}x_5^{14}$, & $x_1x_2^{2}x_3^{6}x_4^{10}x_5^{13}$, & $x_1x_2^{3}x_3^{4}x_4^{10}x_5^{14}$, & $x_1x_2^{3}x_3^{6}x_4^{8}x_5^{14}$, \\
$x_1x_2^{3}x_3^{6}x_4^{10}x_5^{12}$, & $x_1^{2}x_2^{2}x_3^{3}x_4^{12}x_5^{13}$, & $x_1^{2}x_2^{3}x_3^{4}x_4^{10}x_5^{13}$, & $x_1^{2}x_2^{3}x_3^{5}x_4^{8}x_5^{14}$, \\
$x_1^{2}x_2^{3}x_3^{5}x_4^{10}x_5^{12}$, & $x_1^{2}x_2^{3}x_3^{6}x_4^{8}x_5^{13}$, & $x_1^{3}x_2^{4}x_3^{5}x_4^{10}x_5^{10}$, & $x_1^{3}x_2^{5}x_3^{6}x_4^{8}x_5^{10}$.
\end{tabular}%
\end{center}

These monomials are $\omega'_{(5)}$-decomposable. Indeed, we prove this statement for the monomial $x_1^{3}x_2^{5}x_3^{6}x_4^{8}x_5^{10}.$ The others can be computed by similar techniques. Applying Cartan's formula, we get
$$ \begin{array}{ll}
\medskip
x_1^{3}x_2^{5}x_3^{6}x_4^{8}x_5^{10} &= Sq^{1}(x_1^{3}x_2^{3}x_3^{9}x_4^{4}x_5^{12} + x_1^{3}x_2^{3}x_3^{12}x_4^{4}x_5^{9} + x_1^{6}x_2^{3}x_3^{9}x_4^{4}x_5^{9} + x_1^{3}x_2^{6}x_3^{9}x_4^{8}x_5^{5})\\
\medskip
&\quad + Sq^{2}(x_1^{3}x_2^{5}x_3^{9}x_4^{4}x_5^{9} + x_1^{3}x_2^{3}x_3^{10}x_4^{4}x_5^{10} + x_1^{5}x_2^{3}x_3^{9}x_4^{8}x_5^{5})\\
\medskip
&\quad+ Sq^{4}(x_1^{3}x_2^{3}x_3^{9}x_4^{4}x_5^{9} + x_1^{3}x_2^{3}x_3^{9}x_4^{8}x_5^{5} + x_1^{3}x_2^{9}x_3^{6}x_4^{4}x_5^{6})\\
&\quad + Sq^{8}(x_1^{3}x_2^{5}x_3^{6}x_4^{4}x_5^{6}) \mod (\mathbb P_{32}^{\otimes 5})^{< \omega'_{(5)}}.
\end{array}$$
Next if $\omega(X) = \omega'_{(6)}$ then a direct computation shows that $X$ is of the form $LY_1^{2^{r}},$ where $r$ is a suitable integer and $L$ is one of the following monomials:

\begin{center}
\begin{tabular}{llll}
$L_{1}=x_1^{7}x_2^{2}x_3^{4}x_4^{4}x_5^{7}$, & $L_{2}=x_1^{7}x_2^{2}x_3^{4}x_4^{7}x_5^{4}$, & $L_{3}=x_1^{7}x_2^{2}x_3^{7}x_4^{4}x_5^{4}$, & $L_{4}=x_1^{7}x_2^{7}x_3^{2}x_4^{4}x_5^{4}$, \\
$L_{5}=x_1^{7}x_2^{2}x_3^{4}x_4^{6}x_5^{5}$, & $L_{6}=x_1^{7}x_2^{2}x_3^{6}x_4^{4}x_5^{5}$, & $L_{7}=x_1^{7}x_2^{6}x_3^{2}x_4^{4}x_5^{5}$, & $L_{8}=x_1^{7}x_2^{2}x_3^{4}x_4^{5}x_5^{6}$, \\
$L_{9}=x_1^{7}x_2^{2}x_3^{6}x_4^{5}x_5^{4}$, & $L_{10}=x_1^{7}x_2^{6}x_3^{2}x_4^{5}x_5^{4}$, & $L_{11}=x_1^{7}x_2^{2}x_3^{5}x_4^{4}x_5^{6}$, & $L_{12}=x_1^{7}x_2^{2}x_3^{5}x_4^{6}x_5^{4}$, \\
$L_{13}=x_1^{3}x_2^{4}x_3^{4}x_4^{6}x_5^{7}$, & $L_{14}=x_1^{3}x_2^{4}x_3^{6}x_4^{4}x_5^{7}$, & $L_{15}=x_1^{3}x_2^{6}x_3^{4}x_4^{4}x_5^{7}$, & $L_{16}=x_1^{3}x_2^{4}x_3^{4}x_4^{7}x_5^{6}$, \\
$L_{17}=x_1^{3}x_2^{4}x_3^{6}x_4^{7}x_5^{4}$, & $L_{18}=x_1^{3}x_2^{6}x_3^{4}x_4^{7}x_5^{4}$, & $L_{19}=x_1^{3}x_2^{4}x_3^{7}x_4^{4}x_5^{6}$, & $L_{20}=x_1^{3}x_2^{4}x_3^{7}x_4^{6}x_5^{4}$, \\
$L_{21}=x_1^{3}x_2^{6}x_3^{7}x_4^{4}x_5^{4}$, & $L_{22}=x_1^{3}x_2^{7}x_3^{4}x_4^{4}x_5^{6}$, & $L_{23}=x_1^{3}x_2^{7}x_3^{4}x_4^{6}x_5^{4}$, & $L_{24}=x_1^{3}x_2^{7}x_3^{6}x_4^{4}x_5^{4}$, \\
$L_{25}=x_1^{7}x_2^{6}x_3^{3}x_4^{4}x_5^{4}$, & $L_{26}=x_1^{7}x_2^{3}x_3^{4}x_4^{4}x_5^{6}$, & $L_{27}=x_1^{7}x_2^{3}x_3^{4}x_4^{6}x_5^{4}$, & $L_{28}=x_1^{7}x_2^{3}x_3^{6}x_4^{4}x_5^{4}$, \\
$L_{29}=x_1^{3}x_2^{4}x_3^{6}x_4^{6}x_5^{5}$, & $L_{30}=x_1^{3}x_2^{6}x_3^{4}x_4^{6}x_5^{5}$, & $L_{31}=x_1^{3}x_2^{6}x_3^{6}x_4^{4}x_5^{5}$, & $L_{32}=x_1^{3}x_2^{4}x_3^{6}x_4^{5}x_5^{6}$, \\
$L_{33}=x_1^{3}x_2^{6}x_3^{4}x_4^{5}x_5^{6}$, & $L_{34}=x_1^{3}x_2^{6}x_3^{6}x_4^{5}x_5^{4}$, & $L_{35}=x_1^{3}x_2^{4}x_3^{5}x_4^{6}x_5^{6}$, & $L_{36}=x_1^{3}x_2^{6}x_3^{5}x_4^{4}x_5^{6}$, \\
$L_{37}=x_1^{3}x_2^{6}x_3^{5}x_4^{6}x_5^{4}$, & $L_{38}=x_1^{3}x_2^{5}x_3^{4}x_4^{6}x_5^{6}$, & $L_{39}=x_1^{3}x_2^{5}x_3^{6}x_4^{4}x_5^{6}$, & $L_{40}=x_1^{3}x_2^{5}x_3^{6}x_4^{6}x_5^{4}.$
\end{tabular}%
\end{center}
It is not difficult to check that these monomials are inadmissible and so by Theorem \ref{dlKS}, $X$ is inadmissible.\\[1mm]
 Finally, if $\omega(X) = \omega'_{(7)}$ then $X = x_ix_jx_kx_{\ell}T^{2}$ with $1\leq i<j<k<\ell\leq 5$ and $T\in \mathscr C^{\otimes 5}_{14}.$ Then, by direct computations, we find that either $X = x_1^{3}x_2^{5}x_3^{5}x_4^{5}x_5^{14}$ or $X$ is of the form  $LT_1^{2^{r}},$ where $r$ is a suitable integer and $L$ is one of the following monomials:

\begin{center}
\begin{tabular}{lllr}
$L_{41}=x_1^{7}x_2^{2}x_3^{5}x_4^{5}x_5^{5}$, & $L_{42}=x_1^{3}x_2^{4}x_3^{7}x_4^{5}x_5^{5}$, & $L_{43}=x_1^{3}x_2^{4}x_3^{5}x_4^{7}x_5^{5}$, & \multicolumn{1}{l}{$L_{44}=x_1^{3}x_2^{4}x_3^{5}x_4^{5}x_5^{7}$,} \\
$L_{45}=x_1^{3}x_2^{7}x_3^{4}x_4^{5}x_5^{5}$, & $L_{46}=x_1^{3}x_2^{5}x_3^{4}x_4^{7}x_5^{5}$, & $L_{47}=x_1^{3}x_2^{5}x_3^{4}x_4^{5}x_5^{7}$, & \multicolumn{1}{l}{$L_{48}=x_1^{3}x_2^{7}x_3^{5}x_4^{4}x_5^{5}$,} \\
$L_{49}=x_1^{3}x_2^{5}x_3^{7}x_4^{4}x_5^{5}$, & $L_{50}=x_1^{3}x_2^{5}x_3^{5}x_4^{4}x_5^{7}$, & $L_{51}=x_1^{3}x_2^{7}x_3^{5}x_4^{5}x_5^{4}$, & \multicolumn{1}{l}{$L_{52}=x_1^{3}x_2^{5}x_3^{7}x_4^{5}x_5^{4}$,} \\
$L_{53}=x_1^{3}x_2^{5}x_3^{5}x_4^{7}x_5^{4}$, & $L_{54}=x_1^{7}x_2^{3}x_3^{4}x_4^{5}x_5^{5}$, & $L_{55}=x_1^{7}x_2^{3}x_3^{5}x_4^{4}x_5^{5}$, & \multicolumn{1}{l}{$L_{56}=x_1^{7}x_2^{3}x_3^{5}x_4^{5}x_5^{4}$,} \\
$L_{57}=x_1^{3}x_2^{6}x_3^{5}x_4^{5}x_5^{5}$, & $L_{58}=x_1^{3}x_2^{5}x_3^{6}x_4^{5}x_5^{5}$, & $L_{59}=x_1^{3}x_2^{5}x_3^{5}x_4^{6}x_5^{5}$. &  
\end{tabular}%
\end{center}

Using the Cartan formula, it follows that \\[1mm]
$$ \begin{array}{ll}
\medskip
 x_1^{3}x_2^{5}x_3^{5}x_4^{5}x_5^{14} =& Sq^{1}(x_1^{3}x_2^{3}x_3^{9}x_4^{3}x_5^{14} + x_1^{3}x_2^{3}x_3^{5}x_4^{3}x_5^{17} + x_1^{3}x_2^{3}x_3^{9}x_4^{5}x_5^{11} + x_1^{3}x_2^{3}x_3^{9}x_4^{9}x_5^{7})\\
&+ Sq^{2}(x_1^{5}x_2^{3}x_3^{5}x_4^{3}x_5^{14} + x_1^{5}x_2^{3}x_3^{6}x_4^{3}x_5^{13} + x_1^{5}x_2^{3}x_3^{6}x_4^{5}x_5^{11}\\
\medskip
&\qquad+ x_1^{5}x_2^{3}x_3^{9}x_4^{6}x_5^{7} + x_1^{5}x_2^{3}x_3^{6}x_4^{9}x_5^{7})\\
&+Sq^{4}(x_1^{3}x_2^{3}x_3^{5}x_4^{3}x_5^{14} + x_1^{3}x_2^{3}x_3^{6}x_4^{3}x_5^{13} + x_1^{3}x_2^{3}x_3^{6}x_4^{5}x_5^{11}\\
&\qquad + x_1^{3}x_2^{3}x_3^{9}x_4^{6}x_5^{7} + x_1^{3}x_2^{3}x_3^{6}x_4^{9}x_5^{7}) \mod (\mathbb P_{32}^{\otimes 5})^{< \omega'_{(7)}}
\end{array}$$
and so $X = x_1^{3}x_2^{5}x_3^{5}x_4^{5}x_5^{14}$ is $\omega'_{(7)}$-decomposable. Further, it is easy to verify that $L_i$ are inadmissible for $41\leq i\leq 59.$ Then, according to Theorem \ref{dlKS}, $X$ is inadmissible.\\[1mm]
Now, we have seen that $\mathscr C_{32}^{\otimes 5} = \bigcup_{1\leq j\leq 7}\mathscr C_{32}^{\otimes 5}\cap (\mathbb P_{32}^{\otimes 5})^{\leq \omega'_{(j)}}  =\bigcup_{1\leq j\leq 7} (\mathscr C_{32}^{\otimes 5})^{\omega'_{(j)}}.$ On the other hand, we put $Q^{\otimes 5}_{32}(\omega'_{(j)}) := \langle [X]:\, X\in  (\mathscr C_{32}^{\otimes 5})^{\omega'_{(j)}}\rangle\subset Q^{\otimes 5}_{32}.$ Then, we have an isomorphism $$Q^{\otimes 5}_{32}(\omega'_{(j)})\to (Q^{\otimes 5}_{32})^{\omega'_{(j)}},\ [X]\longmapsto [X]_{\omega'_{(j)}}.$$ Combining this with the above computations gives $$Q^{\otimes 5}_{32} = Q^{\otimes 5}_{32}(\omega'_{(1)}) \bigoplus Q^{\otimes 5}_{32}(\omega'_{(2)}) \bigoplus Q^{\otimes 5}_{32}(\omega'_{(3)})  \cong (Q_{32}^{\otimes 5})^{\omega'_{(1)}}\bigoplus (Q_{32}^{\otimes 5})^{\omega'_{(2)}}\bigoplus (Q_{32}^{\otimes 5})^{\omega'_{(3)}}.$$

\medskip

{\it Step 3.} We now prove that $Q^{\otimes 5}_{32}$ is $1004$-dimensional. It is suffice to determine that
$$ \dim (Q_{32}^{\otimes 5})^{(\omega'_{(j)})^{>0}} = \left\{\begin{array}{ll}
9,& \mbox{if } j = 1,\\
310,& \mbox{if } j = 2,\\
320,& \mbox{if } j = 3.
\end{array}\right.$$
Indeed, for each $1\leq l\leq 5,$ let us consider the following sets:
$$ \begin{array}{ll}
\medskip
B(31) &= \{x_l\mathsf{q}_{(l,\,5)}(X):\ X\in (\mathscr {C}^{\otimes 4}_{31})^{>0},\} ,\\
\medskip
B(29) &= \{x_l^{3}\mathsf{q}_{(l,\,5)}(X):\ X\in (\mathscr {C}^{\otimes 4}_{29})^{>0}\},\\
\medskip
B(25) &= \{x_l^{7}\mathsf{q}_{(l,\,5)}(X):\ X\in (\mathscr {C}^{\otimes 4}_{25})^{>0}\},\\
\medskip
B(17) &= \{x_l^{15}\mathsf{q}_{(l,\,5)}(X):\ X\in (\mathscr {C}^{\otimes 4}_{17})^{>0}\},\\
\medskip
B(1) &= \{x_l^{31}\mathsf{q}_{(l,\,5)}(X):\ X\in (\mathscr {C}^{\otimes 4}_{1})^{>0}\}.
\end{array}$$
For each $1\leq j\leq 3,$ writing $B(\omega'_{(j)}) = (\mathbb P^{\otimes 5})^{(\omega'_{(j)})^{>0}}\cap (B(31)\cup B(29)\cup B(25)\cup B(17)\cup B(1)),$ in which $(\mathbb P^{\otimes 5})^{(\omega'_{(j)})^{>0}} := \langle \{X\in (\mathbb P^{\otimes 5})^{>0}:\ \omega(X)=\omega'_{(j)}\} \rangle$ By direct computations, one can easily obtain
$$ \begin{array}{ll}
\medskip
 B(\omega'_{(1)}) &= \{Z_j|\ 1\leq j\leq 4\},\\
\medskip
 B(\omega'_{(2)}) &= \{Z'_j|\ 1\leq j\leq 196\},\\
\medskip
 B(\omega'_{(3)}) &= \{Z''_j|\ 1\leq j\leq 320\}.
\end{array}$$
Notice that $\langle [B(\omega'_{(3)})]_{\omega'_{(3)}} \rangle$ is a $\mathbb Z_2$-vector subspace of $(Q_{32}^{\otimes 5})^{(\omega'_{(j)})^{>0}}.$ So, by Theorem \ref{dlMU}, we conclude that $[B(\omega'_{(3)})]_{\omega'_{(3)}}$ is an basis of $(Q_{32}^{\otimes 5})^{(\omega'_{(j)})^{>0}}.$ This means that $(Q_{32}^{\otimes 5})^{(\omega'_{(j)})^{>0}}$ is $320$-dimensional. Now, we put $C(\omega'_{(1)}) = \{Z_j|\ 5\leq j\leq 9\},$ and $C(\omega'_{(2)}) = \{Z'_j|\ 197\leq j\leq 310\},$ where the monomials $Z_j,\, Z'_j,\, Z''_j$ are described as in the Subsects.\ref{s32_1}, \ref{s32_2} and \ref{s32_3} of the Appendix. Then, the sets $[D(\omega'_{(1)}):=B(\omega'_{(1)})\cup C(\omega'_{(1)})]_{\omega'_{(1)}}$ and $[D(\omega'_{(2)}):=B(\omega'_{(2)})\cup C(\omega'_{(2)})]_{\omega'_{(2)}}$ are the bases of $(Q_{32}^{\otimes 5})^{(\omega'_{(1)})^{>0}}$ and $(Q_{32}^{\otimes 5})^{(\omega'_{(2)})^{>0}}$ respectively. Indeed, let $\omega:= (1,1,1, 1)$ be the weight vector of degree $15$ and let $X$ be an admissible monomial in $(\mathscr C_{32}^{\otimes 5})^{(\omega'_{(1)})^{>0}}.$ Then by Theorem \ref{dlKS}, $X$ has the form $x_ix_jY^2$ with $1\leq i<j\leq 5,$ and $Y\in (\mathscr C_{15}^{\otimes 5})^{\omega}.$ By direct calculations, we find that if $Y_1\in (\mathscr C^{\otimes 5}_{15})^{\omega}$ and the monomial $T:=x_ix_jY_1^2\neq Z_j,$ for all $j,\ 1\leq j\leq 310,$ then $T$ is of the form $LY_2^{16},$ where $L$ is one of the following inadmissible monomials: $x_1^{3}x_2^{4}x_3x_4^{8},\ x_1^{3}x_2^{4}x_3^{8}x_4,\ x_1^{3}x_2^{4}x_3^{8}x_5.$ By Theorem \ref{dlKS}, $T$ is inadmissible. Since $X = x_ix_jY^2$ is admissible, we have $X = Z_j$ for some $j,\ 1\leq j\leq 9.$ Now, suppose that there is a linear relation $\mathcal S = \sum_{1\leq j\leq 9}\gamma_jZ_j\equiv_{\omega'_{(1)}} 0$ where $\gamma_j\in \mathbb Z_2$ for every $j.$ A simple computation using the relations $\mathsf{p}_{(1, k)}(\mathcal S)\equiv_{\omega'_{(1)}} 0,\ 2\leq k\leq 5,$ one gets $\gamma_j= 0,\, j = 1,2, \ldots, 9,$ and therefore $(Q_{32}^{\otimes 5})^{(\omega'_{(1)})^{>0}}$ is $9$-dimensional.

Finally, to prove that the set $[D(\omega'_{(2)})]_{\omega'_{(2)}}$ is a basis of $(Q_{32}^{\otimes 5})^{(\omega'_{(2)})^{>0}},$  we first need the following.

\begin{rems}\label{nx32}
The following monomials are strictly inadmissible:

\begin{center}
\begin{tabular}{llll}
$x_1^{3}x_2^{12}x_3x_4x_5^{15}$, & $x_1^{3}x_2^{12}x_3x_4^{15}x_5$, & $x_1^{3}x_2^{12}x_3^{15}x_4x_5$, & $x_1^{3}x_2^{15}x_3^{12}x_4x_5$, \\
$x_1^{15}x_2^{3}x_3^{12}x_4x_5$, & $x_1^{7}x_2^{11}x_3^{12}x_4x_5$, & $x_1^{3}x_2^{12}x_3x_4^{3}x_5^{13}$, & $x_1^{3}x_2^{12}x_3^{3}x_4x_5^{13}$, \\
$x_1^{3}x_2^{12}x_3^{3}x_4^{13}x_5$, & $x_1^{3}x_2^{4}x_3x_4^{9}x_5^{15}$, & $x_1^{3}x_2^{4}x_3x_4^{15}x_5^{9}$, & $x_1^{3}x_2^{4}x_3^{9}x_4x_5^{15}$, \\
$x_1^{3}x_2^{4}x_3^{9}x_4^{15}x_5$, & $x_1^{3}x_2^{4}x_3^{15}x_4x_5^{9}$, & $x_1^{3}x_2^{4}x_3^{15}x_4^{9}x_5$, & $x_1^{3}x_2^{15}x_3^{4}x_4x_5^{9}$, \\
$x_1^{3}x_2^{15}x_3^{4}x_4^{9}x_5$, & $x_1^{15}x_2^{3}x_3^{4}x_4x_5^{9}$, & $x_1^{15}x_2^{3}x_3^{4}x_4^{9}x_5$, & $x_1^{3}x_2^{5}x_3^{14}x_4x_5^{9}$, \\
$x_1^{3}x_2^{5}x_3^{14}x_4^{9}x_5$, & $x_1^{7}x_2^{11}x_3^{4}x_4x_5^{9}$, & $x_1^{7}x_2^{11}x_3^{4}x_4^{9}x_5$, & $x_1^{3}x_2^{5}x_3^{8}x_4^{7}x_5^{9}$.
\end{tabular}%
\end{center}

For simplicity, we prove this statement for the monomials $x_1^{3}x_2^{5}x_3^{14}x_4x_5^{9}$ and $x_1^{3}x_2^{5}x_3^{8}x_4^{7}x_5^{9}.$ The others can be proved by similar techniques. Indeed, by a direct computation using Cartan's formula, we get
$$ \begin{array}{ll}
x_1^{3}x_2^{5}x_3^{14}x_4x_5^{9} &=Sq^{1}\big(x_1^{3}x_2^{3}x_3^{7}x_4x_5^{17}+
x_1^{3}x_2^{3}x_3^{11}x_4x_5^{13}+
x_1^{3}x_2^{3}x_3^{13}x_4x_5^{11}+
\medskip
x_1^{3}x_2^{3}x_3^{17}x_4x_5^{7}\big)\\
&\quad + Sq^{2}\big(x_1^{2}x_2^{3}x_3^{11}x_4x_5^{13}+
x_1^{2}x_2^{3}x_3^{13}x_4x_5^{11}+
x_1^{5}x_2^{3}x_3^{7}x_4x_5^{14}+
\medskip
x_1^{5}x_2^{3}x_3^{7}x_4^{2}x_5^{13}\\
&\qquad+
x_1^{5}x_2^{3}x_3^{7}x_4^{4}x_5^{11}+
x_1^{5}x_2^{3}x_3^{7}x_4^{8}x_5^{7}+
x_1^{5}x_2^{3}x_3^{11}x_4^{4}x_5^{7}+
\medskip
x_1^{5}x_2^{3}x_3^{13}x_4^{2}x_5^{7}\\
&\qquad+
\medskip
x_1^{5}x_2^{3}x_3^{14}x_4x_5^{7}\big)\\
&\quad + Sq^{4}\big(x_1^{3}x_2^{3}x_3^{7}x_4x_5^{14}+
x_1^{3}x_2^{3}x_3^{7}x_4^{2}x_5^{13}+
x_1^{3}x_2^{3}x_3^{7}x_4^{4}x_5^{11}+
\medskip
x_1^{3}x_2^{3}x_3^{7}x_4^{8}x_5^{7}\\
&\qquad+
x_1^{3}x_2^{3}x_3^{11}x_4^{4}x_5^{7}+
x_1^{3}x_2^{3}x_3^{13}x_4^{2}x_5^{7}+
\medskip
x_1^{3}x_2^{3}x_3^{14}x_4x_5^{7}\big)\\
&\quad +x_1^{2}x_2^{5}x_3^{11}x_4x_5^{13}+
x_1^{2}x_2^{5}x_3^{13}x_4x_5^{11}+
x_1^{3}x_2^{3}x_3^{12}x_4x_5^{13}+
\medskip
x_1^{3}x_2^{3}x_3^{13}x_4x_5^{12}\\
&\quad+
x_1^{3}x_2^{4}x_3^{11}x_4x_5^{13}+
x_1^{3}x_2^{4}x_3^{13}x_4x_5^{11}+
x_1^{3}x_2^{5}x_3^{7}x_4^{8}x_5^{9}+
\medskip
x_1^{3}x_2^{5}x_3^{9}x_4x_5^{14}\\
&\quad+
x_1^{3}x_2^{5}x_3^{9}x_4^{2}x_5^{13}+
x_1^{3}x_2^{5}x_3^{9}x_4^{4}x_5^{11}+
x_1^{3}x_2^{5}x_3^{9}x_4^{8}x_5^{7}+
\medskip
x_1^{3}x_2^{5}x_3^{11}x_4^{4}x_5^{9}\\
&\quad+
\medskip
x_1^{3}x_2^{5}x_3^{13}x_4^{2}x_5^{9} \mod (\mathbb P_{32}^{\otimes 5})^{< \omega'_{(2)}},\\
x_1^{3}x_2^{5}x_3^{8}x_4^{7}x_5^{9}&=Sq^{1}\big(x_1^{3}x_2^{3}x_3x_4^{7}x_5^{17}+
x_1^{3}x_2^{3}x_3x_4^{11}x_5^{13}+
x_1^{3}x_2^{3}x_3x_4^{13}x_5^{11}+
\medskip
x_1^{3}x_2^{3}x_3x_4^{17}x_5^{7}\big)\\
&\quad + Sq^{2}\big(x_1^{2}x_2^{3}x_3x_4^{11}x_5^{13}+
x_1^{2}x_2^{3}x_3x_4^{13}x_5^{11}+
x_1^{5}x_2^{3}x_3x_4^{7}x_5^{14}+
\medskip
x_1^{5}x_2^{3}x_3x_4^{14}x_5^{7}\\
&\qquad + 
x_1^{5}x_2^{3}x_3^{2}x_4^{7}x_5^{13}+
x_1^{5}x_2^{3}x_3^{2}x_4^{9}x_5^{11}+
x_1^{5}x_2^{3}x_3^{2}x_4^{11}x_5^{9}+
\medskip
x_1^{5}x_2^{3}x_3^{2}x_4^{13}x_5^{7}\\
&\qquad + 
x_1^{5}x_2^{3}x_3^{4}x_4^{7}x_5^{11}+
x_1^{5}x_2^{3}x_3^{4}x_4^{11}x_5^{7}+
\medskip
x_1^{5}x_2^{3}x_3^{8}x_4^{7}x_5^{7}\big)\\
\end{array}$$

\newpage
$$ \begin{array}{ll}
&\quad + Sq^{4}\big(x_1^{3}x_2^{3}x_3x_4^{7}x_5^{14}+
x_1^{3}x_2^{3}x_3x_4^{14}x_5^{7}+
x_1^{3}x_2^{3}x_3^{2}x_4^{7}x_5^{13}+
\medskip
x_1^{3}x_2^{3}x_3^{2}x_4^{9}x_5^{11}\\
&\qquad + 
x_1^{3}x_2^{3}x_3^{2}x_4^{11}x_5^{9}+
x_1^{3}x_2^{3}x_3^{2}x_4^{13}x_5^{7}+
x_1^{3}x_2^{3}x_3^{4}x_4^{7}x_5^{11}+
\medskip
x_1^{3}x_2^{3}x_3^{4}x_4^{11}x_5^{7}\\
&\qquad + 
x_1^{3}x_2^{3}x_3^{8}x_4^{7}x_5^{7}+
\medskip
x_1^{3}x_2^{9}x_3^{4}x_4^{5}x_5^{7}\big) +Sq^{8}\big(x_1^{3}x_2^{5}x_3^{4}x_4^{5}x_5^{7}\big)\\
&\quad + x_1^{2}x_2^{5}x_3x_4^{11}x_5^{13}+
x_1^{2}x_2^{5}x_3x_4^{13}x_5^{11}+
\medskip
x_1^{3}x_2^{3}x_3x_4^{12}x_5^{13}\\
&\quad + 
x_1^{3}x_2^{3}x_3x_4^{13}x_5^{12}+
x_1^{3}x_2^{3}x_3^{4}x_4^{9}x_5^{13}+
\medskip
x_1^{3}x_2^{3}x_3^{4}x_4^{13}x_5^{9}\\
&\quad + 
x_1^{3}x_2^{4}x_3x_4^{11}x_5^{13}+
x_1^{3}x_2^{4}x_3x_4^{13}x_5^{11}+
\medskip
x_1^{3}x_2^{5}x_3x_4^{9}x_5^{14}\\
&\quad + 
x_1^{3}x_2^{5}x_3x_4^{14}x_5^{9}+
x_1^{3}x_2^{5}x_3^{4}x_4^{9}x_5^{11}+
x_1^{3}x_2^{5}x_3^{8}x_4^{5}x_5^{11} \mod (\mathbb P_{32}^{\otimes 5})^{< \omega'_{(2)}}.
\end{array}$$
The above equalities show that $x_1^{3}x_2^{5}x_3^{14}x_4x_5^{9}$ and $x_1^{3}x_2^{5}x_3^{8}x_4^{7}x_5^{9}$ are strictly inadmissible. 
\end{rems}

Next, suppose $X\in (\mathscr C_{32}^{\otimes 5})^{(\omega'_{(2)})^{>0}}.$ Then,  using Theorem \ref{dlKS}, we have $X = x_ax_bx_cx_dL^2$ where $1\leq a<b<c<d\leq 5,$ and $L$ is an admissible monomial in $(\mathbb P^{\otimes 5}_{14})^{\leq \omega}$ with $\omega = (2,2,2).$ A direct computation shows that if $L_1\in (\mathscr C^{\otimes 5}_{14})^{\omega}$ and the monomial $M:=x_ax_bx_cx_dL_1^2\neq Z'_j,$ for all $j,\ 1\leq j\leq 310,$ then either $M$ is of the form $x_1^{3}x_2^{2}x_3x_4x_5L_2^{2^{k}},$ where $k$ is a suitable integer and $x_1^{3}x_2^{2}x_3x_4x_5$ is inadmissible or $M$ is one of the inadmissible monomials as shown in Remark \ref{nx32}. So, according to Theorem \ref{dlKS}, $M$ is inadmissible. Thus, there exists $j$ for $1\leq j\leq 310$ such that $X = Z'_j.$ Now, we will prove that the set $[D(\omega'_{(2)})]_{\omega'_{(2)}}$ is linearly independent in  $(Q_{32}^{\otimes 5})^{(\omega'_{(2)})^{>0}}.$ Indeed, suppose there is a linear relation  $\mathcal S = \sum_{1\leq j\leq 310}\beta_jZ'_j\equiv_{\omega'_{(2)}} 0$  where $\beta_j\in \mathbb Z_2$ for all $j.$ Acting the homomorphisms $\mathsf{p}_{(l, \mathscr L)}: \mathbb P_{32}^{\otimes 5}\to \mathbb P_{32}^{\otimes 5}$ for $\ell(\mathscr L) \leq 2$ on the both sides of $\mathcal S$, we immediately obtain $\gamma_j = 0$ for any $j,\, 1\leq j\leq 310.$ This completes the proof of the theorem.

\begin{rems}\label{nx2}
By a simple computation, we have seen that $\widetilde{\Phi^{0}}((\mathscr C^{\otimes 4}_{32})^{\omega'_{(j)}})\subset  (\mathscr C_{32}^{\otimes 5})^{(\omega'_{(j)})^{0}}$ for $1\leq j\leq 3$ and that $\widetilde{\Phi^{>0}}((\mathscr C^{\otimes 4}_{32})^{\omega'_{(j)}})\subset B(\omega'_{(j)}$ for $j = 1,\, 3$ and $\widetilde{\Phi^{>0}}((\mathscr C^{\otimes 4}_{32})^{\omega'_{(2)}}) \cap B(\omega'_{(2)}\neq \emptyset.$ Therefore, from the proof of Theorem \ref{dlc4}, we have 
$$ \begin{array}{ll}
\medskip
 \widetilde{\Phi^{>0}}((\mathscr C^{\otimes 4}_{32})^{\omega'_{(j)}})&\subset (\mathscr C_{32}^{\otimes 5})^{(\omega'_{(j)})^{>0}}\ \mbox{for $j = 1,\, 3,$}\\
 (\mathscr C_{32}^{\otimes 5})^{(\omega'_{(2)})^{>0}} &= \widetilde{\Phi^{>0}}((\mathscr C^{\otimes 4}_{32})^{\omega'_{(2)}}) \cup B(\omega'_{(2)})\cup \{Z'_j|\, 285\leq j\leq 310 \}.
\end{array}$$ 
Thus, Conjecture \ref{gtSum} holds also in the 5-variable case and the degree $32.$
\end{rems}

\subsection{Proof of Theorem \ref{dlc4}}

Obviously $n_d = 5(2^{d}-1) + 21.2^{d}.$ In \cite{D.P4}, we have shown that $Q^{\otimes 5}_{n_d}$ is $1894$-dimensional for arbitrary $d > 0.$   Let us now consider $\zeta = 1 < d.$ It is easy to check that 
$$ \mu(n_{\zeta}) = 3 = \alpha(n_{\zeta} +\mu(n_{\zeta})).$$ So, invoking Theorem \ref{dlP}, we may conclude that
$$ \dim Q^{\otimes 6}_{5(2^{d+4}-1) + n_1.2^{d+4}} = (2^{6} - 1)\times \dim Q^{\otimes 5}_{n_d} = 63\times 1894 = 119322,$$
for any positive integer $d.$  The theorem is proved. 

\subsection{Proof of Theorem \ref{dlc5}}

Before coming to the proof of the theorem, we give the following definition: Let $\mathbb Z_2^{5}$ be a rank $5$ elementary abelian 2-group, which views as a $5$-dimensional $\mathbb Z_2$-vector space. As well known, $H^*(\mathbb Z_2^{5}, \mathbb Z_2)\cong S((\mathbb Z_2^{5})^{*}),$ the symmetric algebra over the dual space $(\mathbb Z_2^{5})^{*} = H^{1}(\mathbb Z_2^{5}, \mathbb Z_2).$ Pick $x_1, \ldots, x_5$ to be a basis of $H^{1}(\mathbb Z_2^{5}, \mathbb Z_2).$ Then, one has an isomorphism $H^{*}(\mathbb Z_2^{5}, \mathbb Z_2) \cong \mathbb P^{\otimes 5} = \mathbb Z_2[x_1, \ldots, x_5].$ One should note that $\mathbb Z_2^{5}\cong \langle x_1, \ldots, x_5\rangle \subset \mathbb P^{\otimes 5}.$ For $1\leq d\leq 5,$  we define the $\mathbb Z_2$-linear map $\sigma_d: \mathbb Z_2^{5}\to \mathbb Z_2^{5}$ by setting  $ \sigma_d(x_d) = x_{d+1},\;\sigma_d(x_{d+1}) = x_d\;\sigma_d(x_i) = x_i$ for $i\neq d, d +1,\; 1\leq d\leq 4,$ and $ \sigma_5(x_1) = x_1 + x_2,\; \sigma_5(x_i) = x_i$ for $2\leq i \leq 5.$

\begin{rems} Denote by $\Sigma_5$ the symmetric group of degree $5.$ Then, $\Sigma_5$ is generated by the ones associated with $\sigma_1,\, \ldots, \sigma_4.$ For each permutation in $\Sigma_5$, consider corresponding permutation matrix; these form a group of matrices isomorphic to $\Sigma_5.$ Indeed, consider the following map $\Delta: \Sigma_5\to \mathcal P_{5\times 5},$ where the latter is the set of permutation matrices of order $5.$ This map is defined as follows: given $\sigma\in \Sigma_5,$ the $i$-th column of $\Delta(\sigma)$ is the column vector with a $1$ in the $\rho(i)$-th position, and $0$ elsewhere. It is easy to see that $\Delta(\rho)$ is indeed a permutation matrix, since a $1$ occurs in any position if and only if that position is described by $(\rho(i), i),$ for any $1\leq i\leq 5.$ The map $\Delta$ is clearly multiplicative. (It is to be noted that because these are matrices, it is enough to show that each corresponding entry is equal. So let us take the entry $(i, j)$ of each matrix.) Then, $\Delta(\rho\circ\rho')_{ij}  =1$ if and only if $i = \rho\circ \rho'(j).$ Note also that by ordinary matrix multiplication, one has $(\Delta(\rho)\Delta(\rho'))_{ij} = \sum_{1\leq k\leq 5}\Delta(\rho)_{ik}\Delta(\rho')_{kj}.$ Now, we know that $\Delta(\rho)_{ik} = 1$ only when $i = \rho(k).$ Similarly, $\Delta(\rho')_{kj} = 1$ only when $k = \rho'(j).$ Hence, their product is one precisely when both of these happen: $i = \rho(k),$ and $k = \rho'(j).$ If both these do not happen simultaneously, then whenever one of $\Delta(\rho)_{ik},\, \Delta(\rho')_{kj}$ is one of the other will be zero, so the whole sum will be zero. However, this is the same as saying that the sum is one exactly when $i = \rho\circ\rho'(j).$ This description matches with the description given for $\Delta(\rho\circ \rho')_{ij}$ given earlier. Hence, entry by entry these matrices are the same. Therefore the matrices are the same, and hence $\Delta$ is a homomorphism between the two spaces, an isomorphism as it has trivial kernel and the sets are of the same cardinality. Thus, $GL_5\cong GL(\mathbb Z_2^{5}),$ and $GL_5$ is generated by the matrices associated with $\sigma_1, \ldots, \sigma_5.$ Let $X = x_1^{a_1}x_2^{a_2}\ldots x_5^{a_5}$ be a monomial in $\mathbb P^{\otimes 5}.$ Then, the weight vector $\omega(X)$ is invariant under the permutation of the generators $x_j,\ j = 1, 2, \ldots, 5;$ hence $(Q^{\otimes 5})^{\omega(X)}$ also has a $\Sigma_5$-module structure. We see that the linear map $\sigma_d$ induces a homomorphism of $\mathscr A_2$-algebras which is also denoted by $\sigma_d: \mathbb P^{\otimes 5}\to \mathbb P^{\otimes 5}.$ So, a class $[X]_{\omega}\in (Q^{\otimes 5})^{\omega}$ is an $GL_5$-invariant if and only if $\sigma_d(X) \equiv_{\omega} X$ for $1\leq d\leq 5.$ If  $\sigma_d(X) \equiv_{\omega} X$ for $1\leq d\leq 4,$ then $[X]_{\omega}$ is an $\Sigma_5$-invariant. 
\end{rems}

We are now ready to prove Theorem \ref{dlc5}. In what follows: Let $\omega$ be a weight vector of degree $n$ and let $F_1, F_2, \ldots, F_s$ be the monomials in $(\mathbb P^{\otimes 5}_n)^{\omega}.$ For a subgroup $G$ of $GL_5(\mathbb Z_2),$ denote by $G(F_1; F_2; \ldots, F_s)$ the $G$-submodule of $(Q^{\otimes 5}_n)^{\omega}$ generated by the set $\{[F_i]_{\omega}:\, 1\leq i\leq s\}.$ We also note that if $\omega$ is a weight vector of a minimal spike, then $[F_i]_{\omega} = [F_i]$ for all $i.$

\begin{proof}[{\it Proof of Case n = 14}]
Firstly, we have the following remarks.

\begin{rems}\label{nx14}
According to Ly-Tin \cite{L.T}, if $X$ is an admissible monomial in $\mathbb P^{\otimes 5}_{14},$ then $\omega(X)$ is one of the following sequences: $\widetilde{\omega}_{(1)}: =(2,2,2),$ $\widetilde{\omega}_{(2)}: =(2,4,1)$ and $\widetilde{\omega}_{(3)}: =(4,3,1).$ A direct computation using a result in Sum \cite{N.S1}, we have $$(\mathscr C_{14}^{\otimes 5})^{0} = \bigcup_{1\leq l\leq 5}\mathsf{q}_{(l,\,t)}(\mathscr C_{14}^{\otimes 4}) = (\mathscr C_{14}^{\otimes 5})^{(\widetilde{\omega}_{(1)})^{0}}\cup (\mathscr C_{14}^{\otimes 5})^{(\widetilde{\omega}_{(3)})^{0}},$$
where $|(\mathscr C_{14}^{\otimes 5})^{(\widetilde{\omega}_{(1)})^{0}}| = 115$ and $|(\mathscr C_{14}^{\otimes 5})^{(\widetilde{\omega}_{(3)})^{0}}| = 75.$  For each $1\leq l\leq 5,$ letting
$$ \begin{array}{ll}
\medskip
E(13) &= \{x_l\mathsf{q}_{(l,\,5)}(X):\ X\in (\mathscr {C}^{\otimes 4}_{13})^{>0},\} ,\\
\medskip
E(11) &= \{x_l^{3}\mathsf{q}_{(l,\,5)}(X):\ X\in (\mathscr {C}^{\otimes 4}_{11})^{>0}\},\\
\medskip
E(7) &= \{x_l^{7}\mathsf{q}_{(l,\,5)}(X):\ X\in (\mathscr {C}^{\otimes 4}_{7})^{>0}\},\\
\end{array}$$
and denote $E(\widetilde{\omega}_{(j)}) = (\mathbb P^{\otimes 5})^{(\widetilde{\omega}_{(j)})^{>0}} \cap (E(13)\cup E(11)\cup E(7)),$ with $$(\mathbb P^{\otimes 5})^{(\widetilde{\omega}_{(j)})^{>0}} := \langle \{X\in (\mathbb P_{14}^{\otimes 5})^{>0}:\ \omega(X)=\widetilde{\omega}_{(j)}\} \rangle$$
for $j = 2,\, 3.$ Then, by a simple computation, it follows that
$$ \begin{array}{ll}  
\medskip
(\mathscr C^{\otimes 5}_{14})^{(\widetilde{\omega}_{(1)})^{>0}} &=   \widetilde{\Phi^{>0}}((\mathscr C^{\otimes 4}_{14})^{\widetilde{\omega}_{(1)}})\cup \{x_1x_2^{2}x_3^{4}x_4^{3}x_5^{4},\, x_1^{3}x_2^{4}x_3x_4^{2}x_5^{4}\}\\
\medskip
(\mathscr C^{\otimes 5}_{14})^{(\widetilde{\omega}_{(2)})^{>0}}&= E(\widetilde{\omega}_{(2)}) \cup D,\\
(\mathscr C^{\otimes 5}_{14})^{(\widetilde{\omega}_{(3)})^{>0}}&= E(\widetilde{\omega}_{(3)}),
\end{array}$$
where $D$ is the set of the following admissible monomials:

\begin{center}
\begin{tabular}{llrr}
$x_1x_2^{2}x_3^{2}x_4^{3}x_5^{6}$, & $x_1x_2^{2}x_3^{3}x_4^{2}x_5^{6}$, & \multicolumn{1}{l}{$x_1x_2^{2}x_3^{3}x_4^{6}x_5^{2}$,} & \multicolumn{1}{l}{$x_1x_2^{3}x_3^{2}x_4^{2}x_5^{6}$,} \\
$x_1x_2^{3}x_3^{2}x_4^{6}x_5^{2}$, & $x_1x_2^{3}x_3^{6}x_4^{2}x_5^{2}$, & \multicolumn{1}{l}{$x_1^{3}x_2x_3^{2}x_4^{2}x_5^{6}$,} & \multicolumn{1}{l}{$x_1^{3}x_2x_3^{2}x_4^{6}x_5^{2}$,} \\
$x_1^{3}x_2x_3^{6}x_4^{2}x_5^{2}$, & $x_1^{3}x_2^{5}x_3^{2}x_4^{2}x_5^{2}$. &       &  
\end{tabular}%
\end{center}
So, we have
$$ \begin{array}{ll}
\medskip
 (\mathscr C^{\otimes 5}_{14})^{\widetilde{\omega}_{(1)}} &= \{W_j|\, 1\leq j\leq 130\},\\
\medskip
 (\mathscr C^{\otimes 5}_{14})^{\widetilde{\omega}_{(2)}} &= (\mathscr C^{\otimes 5}_{14})^{(\widetilde{\omega}_{(2)})^{>0}}=  \{W_j|\, 131\leq j\leq 145\},\\
 (\mathscr C^{\otimes 5}_{14})^{\widetilde{\omega}_{(3)}} &= \{W_j|\, 146\leq j\leq 320\},\\
\end{array}$$
where the monomials $W_j$ for $1\leq j\leq 320,$ are determined as in Subsects.\ref{s14_1}, \ref{s14_2} and \ref{s14_3} of Appendix. Therefore, we have an isomorphism $$Q^{\otimes 5}_{14}\cong (Q^{\otimes 5}_{14})^{\widetilde{\omega}_{(1)}}\bigoplus (Q^{\otimes 5}_{14})^{\widetilde{\omega}_{(2)}}\bigoplus (Q^{\otimes 5}_{14})^{\widetilde{\omega}_{(3)}},$$
where $$ \dim  (Q^{\otimes 5}_{14})^{\widetilde{\omega}_{(j)}} = \left\{\begin{array}{ll}
130&\mbox{if $j = 1$},\\
15&\mbox{if $j = 2$},\\
175&\mbox{if $j = 3$},
\end{array}\right.$$
and so $Q^{\otimes 5}_{14}$ has dimension $320.$ This result has also been computed by Ly-Tin \cite{L.T}. However, our approach above is much less computational. Further, we observe that $\widetilde{\Phi}_*((\mathscr C^{\otimes 4}_{14})^{\widetilde{\omega}_{(j)}})\subset  (\mathscr C^{\otimes 5}_{14})^{(\widetilde{\omega}_{(j)})}$ for $j = 1,\, 3$ and $\widetilde{\Phi}_*((\mathscr C^{\otimes 4}_{14})^{\widetilde{\omega}_{(2)}}) = \emptyset$ since $ (\mathscr C^{\otimes 4}_{14})^{\widetilde{\omega}_{(2)}} = \emptyset$ and so Conjecture \ref{gtSum} satisfies for $m = 5$ and the degree $14.$ 
\end{rems}

Now, to prove that $[P_{\mathscr A_2}((\mathbb P_{14}^{\otimes 5})^{*})]_{GL_5}$ is 1-dimensional, we need determine the spaces of $GL_5$-invariants $[(Q^{\otimes 5}_{14})^{\widetilde{\omega}_{(j)}}]^{GL_5}$ for $1\leq j\leq 3.$ We obtain the following.

\begin{propo}\label{md14}
We have
$$ [(Q^{\otimes 5}_{14})^{\widetilde{\omega}_{(j)}}]^{GL_5} = \left\{\begin{array}{ll}
\langle [\zeta] \rangle&\mbox{if $j = 1$},\\
0&\mbox{if $j = 2,\, 3$},
\end{array}\right.$$
where $$ \begin{array}{ll}
\medskip
\zeta &= W_{51} + W_{53} + W_{55} + W_{56} + W_{57}+W_{111} + W_{112} + W_{113} + W_{114} + W_{115}\\
\medskip
&\quad+ W_{116} + W_{117} + W_{118} + W_{119} + W_{115} + W_{121} + W_{122} + W_{124} + W_{125}\\
&\quad + W_{127} + W_{128} + W_{129} + W_{130}.
\end{array}$$
Consequently, $\dim [P_{\mathscr A_2}((\mathbb P_{14}^{\otimes 5})^{*})]_{GL_5}\leq 1.$
\end{propo}

\begin{proof}
We prove the proposition for the cases $j = 1, 2.$ The result when $j = 3$ can be obtained by using similar idea. We first consider the case $j = 1.$ Note that since $\widetilde{\omega}_{(1)}$ is the weight vector of the minimal spike $x_1^{7}x_2^{7}\in \mathbb P^{\otimes 5}_{14},$ $[W_j]_{\widetilde{\omega}_{(1)}} = [W_j]$ for all $W_j\in (Q^{\otimes 5}_{14})^{\widetilde{\omega}_{(1)}}.$ From the admissible monomial basis of $(Q^{\otimes 5}_{14})^{\widetilde{\omega}_{(1)}},$ we have a direct summand decomposition of the $\Sigma_5$-submodules:
$$ \begin{array}{ll}
 (Q^{\otimes 5}_{14})^{\widetilde{\omega}_{(1)}} &= \Sigma_5(W_1) \bigoplus \Sigma_5(W_{11})  \bigoplus \Sigma_5(W_{41}) \bigoplus \Sigma_5(W_{51})\\
&\quad \bigoplus \Sigma_5(W_{71}, W_{81}, W_{106}, W_{111}) \bigoplus \mathbb V,
\end{array}$$    
where $$ \begin{array}{ll}
\medskip
&\Sigma_5(W_1) = \langle \{[W_j]:\, 1\leq j\leq 10 \} \rangle,\ \ \Sigma_5(W_{11}) = \langle \{[W_j]:\, 11\leq j\leq 40 \} \rangle,\\
\medskip
&\Sigma_5(W_{41}) = \langle \{[W_j]:\, 41\leq j\leq 50 \} \rangle, \ \ \Sigma_5(W_{51}) = \langle \{[W_j]:\, 51\leq j\leq 70 \} \rangle,\\
\medskip
 &\Sigma_5(W_{71}, W_{81}, W_{106}, W_{111}) = \langle \{[W_j]:\, 71\leq j\leq 115 \} \rangle,\\
&\mathbb V = \langle \{[W_j]:\, 116\leq j\leq 130 \} \rangle.
\end{array}$$  
Then, a direct computation shows that 

\begin{itemize}
\item[(i)]  $[\Sigma_5(W_1)]^{\Sigma_5} = \langle [\widehat{p _1}] \rangle$ with $\widehat{p _1}:= \sum_{1\leq j\leq 10}W_j,$
\item[(ii)] $[\Sigma_5(W_{11})]^{\Sigma_5} = \langle [\widehat{p _2}] \rangle$ with $\widehat{p _2}:= \sum_{11\leq j\leq 40}W_j,$
\item[(iii)] $[\Sigma_5(W_{41})]^{\Sigma_5} = \langle [\widehat{p _3}] \rangle$ with $\widehat{p _3}:= \sum_{41\leq j\leq 50}W_j,$
\item[(iv)] $[\Sigma_5(W_{51})]^{\Sigma_5} = \langle [\widehat{p _4}] \rangle$ with $\widehat{p _4}:= \sum_{51\leq j\leq 70}W_j,$
\item[(v)] $[\Sigma_5(W_{71}, W_{81}, W_{106}, W_{111})]^{\Sigma_5} = \langle [\widehat{p _5}] \rangle$\\[1mm] with $\widehat{p _5}:= W_{71} + W_{73} + W_{75} + W_{76} + W_{77} +\sum_{111\leq j\leq 115}W_j,$
\item[(vi)] $[\mathbb V]^{\Sigma_5} =  \langle [\widehat{p _6}] \rangle$ with $\widehat{p _6}:= \sum_{116\leq j\leq 119}W_j + W_{121} + W_{122}+ W_{124} + W_{125} + \sum_{127\leq j\leq 130}W_j.$
\end{itemize}

Indeed, for simplicity, we prove the result (vi) in detail. The others can be proved by the similar computations. An admissible monomial basis for $\mathbb V$ is the set $\{[W_j]:\, 116\leq j\leq 130\}.$ Suppose that $[F]\in [\mathbb V]^{\Sigma_5},$ then $$ F\equiv \sum_{116\leq j\leq 130}\gamma_jW_j$$
with $\gamma_j\in \mathbb Z_2.$ Acting the homomorphisms $\sigma_d: \mathbb P^{\otimes 5}_{14}\to \mathbb P^{\otimes 5}_{14}$ on both sides of this equality and using the relations $\sigma_d(F) + F\equiv 0,$ for $1\leq d\leq 4,$ we get
$$ \begin{array}{ll}
\medskip
\sigma_1(F) + F&\equiv (\gamma_{119} + \gamma_{121})W_{116} + \gamma_{120}W_{117} + \gamma_{123}W_{118} + (\gamma_{122} + \gamma_{130})(W_{122} + W_{130})\\
\medskip
 &\quad + (\gamma_{125} + \gamma_{127})(W_{125} + W_{127}) + (\gamma_{126} + \gamma_{128} + \gamma_{130})W_{126}\\
\medskip
&\quad + (\gamma_{122} + \gamma_{126} + \gamma_{128})W_{128}\equiv 0,\\
\medskip
\sigma_2(F) + F&\equiv (\gamma_{116} + \gamma_{119})(W_{116} + W_{119}) + (\gamma_{117} + \gamma_{120} + \gamma_{129})(W_{117}+W_{120}) \\
\medskip
&\quad + (\gamma_{118} + \gamma_{122})(W_{118} + W_{122})  +  (\gamma_{123} + \gamma_{126})(W_{123} + W_{126})\\
\medskip
&\quad +  (\gamma_{124} + \gamma_{125} + \gamma_{126})W_{124}+(\gamma_{123} + \gamma_{124} + \gamma_{125})W_{125}\\
\medskip
&\quad + (\gamma_{128} + \gamma_{130})(W_{128} + W_{130})\equiv 0,\\
\medskip
\sigma_3(F) + F&\equiv (\gamma_{117} + \gamma_{118} + \gamma_{123})W_{117} + (\gamma_{117} + \gamma_{118} + \gamma_{120})W_{118} \\
\medskip
&\quad + (\gamma_{119} + \gamma_{121})(W_{119} + W_{121})+  (\gamma_{120} + \gamma_{123})(W_{120} + W_{123})\\
\medskip
&\quad   +(\gamma_{124} + \gamma_{129})(W_{124} + W_{129}) +  (\gamma_{122} + \gamma_{125} + \gamma_{126})(W_{125}+W_{126}) \\
\medskip
&\quad + (\gamma_{127} + \gamma_{128})(W_{127} + W_{128})\equiv 0,\\
\medskip
\sigma_4(F) + F&\equiv   (\gamma_{116} + \gamma_{117})(W_{116} + W_{117}) + (\gamma_{119} + \gamma_{120} + \gamma_{129})W_{119}\\
\medskip
&\quad + (\gamma_{119} + \gamma_{120} + \gamma_{121})W_{120} +  (\gamma_{121} + \gamma_{129})(W_{121} + W_{129}) \\
&\quad + \gamma_{123}W_{124} + \gamma_{126}W_{125} + (\gamma_{128} + \gamma_{130})W_{127} \equiv 0.
\end{array}$$
These equalities show that $\gamma_{120} = \gamma_{123} = \gamma_{126}=0$ and $\gamma_j = \gamma_{121}$ for $j\neq 120,\, 123,\, 126.$ This confirms the above statement. Now, let $G\in (\mathbb P^{\otimes 5}_{14})^{\widetilde{\omega}_{(1)}}$ such that $[G]\in [(Q^{\otimes 5}_{14})^{\widetilde{\omega}_{(1)}}]^{GL_5}.$ Because $\Sigma_5\subset GL_5,$ from the above calculations, we have
$$ G\equiv \beta_1\widehat{p _1} + \beta_2\widehat{p _2}+\beta_3\widehat{p _3}+\beta_4\widehat{p _4}+\beta_5\widehat{p _5}+\beta_6\widehat{p _6}$$
in which $\beta_j\in \mathbb Z_2.$ Using the relation $\sigma_5(G) + G\equiv 0,$ we get
$$  (\beta_1 + \beta_2)W_4 +  (\beta_2 + \beta_5)W_{14} + \beta_2W_{20} +   (\beta_2 + \beta_3)W_{32} + (\beta_4 + \beta_6)W_{111}+ \ \mbox{other terms} \equiv 0.$$
This equality implies that $\beta_4 = \beta_6$ and $\beta_j = 0$ for $j\neq 4,\, 6.$ This means that $$[(Q^{\otimes 5}_{14})^{\widetilde{\omega}_{(1)}}]^{GL_5} = \langle [\zeta:= \widehat{p _4} + \widehat{p _6}] \rangle.$$  

Now we prove the proposition for the case $j = 2.$ According to Remark \ref{nx14}, an admissible basis of $(Q^{\otimes 5}_{14})^{\widetilde{\omega}_{(2)}}$ is the set $\{[W_j]_{\widetilde{\omega}_{(2)}}:\, 131\leq j\leq 145\}.$ From this, a direct computation shows that
$$ (Q^{\otimes 5}_{14})^{\widetilde{\omega}_{(2)}} = (Q^{\otimes 5}_{14})^{(\widetilde{\omega}_{(2)})^{>0}}  \cong \Sigma_5(W_{131})\bigoplus \Sigma_5(W_{136}),$$
where $\Sigma_5(W_{131}) = \langle \{[W_j]_{\widetilde{\omega}_{(2)}}:\, 131\leq j\leq 135 \}\rangle$ and $\Sigma_5(W_{136}) = \langle \{[W_j]_{\widetilde{\omega}_{(2)}}:\, 136\leq j\leq 145 \}\rangle$ and that
$$ \begin{array}{ll}
\medskip
 [\Sigma_5(W_{131})]^{\Sigma_5} &= \langle [\sum_{131\leq j\leq 135}W_j]_{\widetilde{\omega}_{(2)}} \rangle\\
 \mbox{[}\Sigma_5(W_{136})\mbox{]}^{\Sigma_5} &= \langle [\sum_{136\leq j\leq 145}W_j]_{\widetilde{\omega}_{(2)}} \rangle.
\end{array}$$
Then if $H\in (\mathbb P^{\otimes 5}_{14})^{\widetilde{\omega}_{(2)}}$ such that $[H]\in [(Q^{\otimes 5}_{14})^{\widetilde{\omega}_{(2)}}]^{GL_5},$ then $$H\equiv_{\widetilde{\omega}_{(2)}} \xi_1\sum_{131\leq j\leq 135}W_j +\xi_2\sum_{136\leq j\leq 145}W_j$$ with $\xi_1$ and $\xi_2$ belong to $\mathbb Z_2.$ Straightforward calculations indicate that
$$ \sigma_5(H) + H\equiv_{\widetilde{\omega}_{(2)}}  (\beta_1 + \beta_2)W_{134} +  \beta_2 (W_{139}  +W_{140} + W_{141}) \equiv_{\widetilde{\omega}_{(2)}} 0$$
which implies $\xi_1 = \xi_2 = 0.$ So $[(Q^{\otimes 5}_{14})^{\widetilde{\omega}_{(2)}}]^{GL_5}$ is trivial. 

Now because $Q^{\otimes 5}_{14}\cong \bigoplus_{1\leq j\leq 3} (Q^{\otimes 5}_{14})^{\widetilde{\omega}_{(j)}}$ (see Remark \ref{nx14}) and $[P_{\mathscr A_2}((\mathbb P_{14}^{\otimes 5})^{*})]_{GL_5}\cong [Q^{\otimes 5}_{14}]^{GL_5},$ by the above calculations, we conclude that
$$ \dim [P_{\mathscr A_2}((\mathbb P_{14}^{\otimes 5})^{*})]_{GL_5}\leq \sum_{1\leq j\leq 3}\dim [(Q^{\otimes 5}_{14})^{\widetilde{\omega}_{(j)}}]^{GL_5} = 1.$$
The proof of the proposition is completed.
\end{proof}

Now following Lin \cite{Lin} and Chen \cite{Chen}, it implies that ${\rm Ext}_{\mathscr {A}_2}^{5,5+14}(\mathbb Z_2,\mathbb Z_2) = \langle h_0d_0 \rangle$ with $h_0d_0\neq 0.$ Further, $h_0\in {\rm Im}(Tr_1^{\mathscr A_2})$ (see Singer \cite{W.S1}),  $d_0\in {\rm Im}(Tr_4^{\mathscr A_2})$ (see H\`a \cite{Ha}) and the "total" transfer $Tr^{\mathscr A_2}$ is an algebraic homomorphism. So, $h_0d_0\in {\rm Im}(Tr_5^{\mathscr A_2}),$ that is, $\dim [P_{\mathscr A_2}((\mathbb P_{14}^{\otimes 5})^{*})]_{GL_5}\geq 1.$ This and Proposition \ref{md14} show that $[P_{\mathscr A_2}((\mathbb P_{14}^{\otimes 5})^{*})]_{GL_5}$ is one-dimensional, which confirms the result for the case $n = 14.$
\end{proof}

\begin{proof}[{\it Proof of Case n = 31}]

Let us recall that Kameko's squaring operation $\widetilde {Sq^0_*}: Q^{\otimes 5}_{31}  \longrightarrow Q^{\otimes 5}_{13}$ is an epimorphism of $\mathbb Z_2GL_5$-modules. So, $Q^{\otimes 5}_{31}\cong {\rm Ker}\widetilde {Sq^0_*}\bigoplus Q^{\otimes 5}_{13}.$  By our previous result in \cite{D.P4} that the invariant space $[Q^{\otimes 5}_{13}]^{GL_5}$ is trivial, and we therefore deduce an estimate 
$$ \dim [P_{\mathscr A_2}((\mathbb P_{31}^{\otimes 5})^{*})]_{GL_5} = \dim [Q^{\otimes 5}_{31}]^{GL_5}\leq \dim [{\rm Ker}\widetilde {Sq^0_*}]^{GL_5}.$$
On the other hand, the works by Lin \cite{Lin} and Chen \cite{Chen} show that ${\rm Ext}_{\mathscr {A}_2}^{5,5+31}(\mathbb Z_2,\mathbb Z_2) = \langle n_0, h_0^{4}h_5 \rangle.$ Moreover, the non-zero elements $h_0^{4}h_5$ and $n_0$ are detected by the rank 5 transfer (see Singer \cite{W.S1}, Ch\ohorn n-H\`a \cite{C.H}). Hence, $\dim [P_{\mathscr A_2}((\mathbb P_{31}^{\otimes 5})^{*})]_{GL_5}\geq 2.$ Thus, to prove that $[P_{\mathscr A_2}((\mathbb P_{31}^{\otimes 5})^{*})]_{GL_5}$ has dimension $2,$ we will show that $\dim [{\rm Ker}\widetilde {Sq^0_*}]^{GL_5}\leq 2.$ But this is immediate from Theorem \ref{dlc1} and the following proposition.

\begin{propo}\label{mdbb31}
The following holds: 
$$ \dim [(Q^{\otimes 5}_{31})^{\omega_{(j)}}]^{GL_5} = \left\{\begin{array}{ll}
1 &\mbox{if $j = 1,\, 2$},\\
0&\mbox{if $j = 3$}.
\end{array}\right.$$
\end{propo}

\begin{proof}
We prove the proposition for $j = 1,\, 3$ and leave the rest to the reader.  First of all, we consider the space $(Q^{\otimes 5}_{31})^{\omega_{(1)}} = (Q^{\otimes 5}_{31})^{\omega_{(1)}^{0}}\bigoplus (Q^{\otimes 5}_{31})^{\omega_{(1)}^{>0}},$ where as shown in the proofs of Theorem \ref{dlc1} and Proposition \ref{md1}, the readers can notice that $Q^{\otimes 5}_{31})^{\omega_{(1)}^{0}} $ is a $\mathbb Z_2$-vector space of dimension $30$ with a basis consisting all the classes represented by the following admissible monomials:

\begin{center}
\begin{tabular}{lllll}
$L_{1}=x_5^{31}$, & $L_{2}=x_4^{31}$, & $L_{3}=x_3^{31}$, & $L_{4}=x_2^{31}$, & $L_{5}=x_1^{31}$, \\
$L_{6}=x_4x_5^{30}$, & $L_{7}=x_3x_5^{30}$, & $L_{8}=x_3x_4^{30}$, & $L_{9}=x_2x_5^{30}$, & $L_{10}=x_2x_4^{30}$, \\
$L_{11}=x_2x_3^{30}$, & $L_{12}=x_1x_5^{30}$, & $L_{13}=x_1x_4^{30}$, & $L_{14}=x_1x_3^{30}$, & $L_{15}=x_1x_2^{30}$, \\
$L_{16}=x_3x_4^{2}x_5^{28}$, & $L_{17}=x_2x_4^{2}x_5^{28}$, & $L_{18}=x_2x_3^{2}x_5^{28}$, & $L_{19}=x_2x_3^{2}x_4^{28}$, & $L_{20}=x_1x_4^{2}x_5^{28}$, \\
$L_{21}=x_1x_3^{2}x_5^{28}$, & $L_{22}=x_1x_3^{2}x_4^{28}$, & $L_{23}=x_1x_2^{2}x_5^{28}$, & $L_{24}=x_1x_2^{2}x_4^{28}$, & $L_{25}=x_1x_2^{2}x_3^{28}$, \\
$L_{26}=x_2x_3^{2}x_4^{4}x_5^{24}$, & $L_{27}=x_1x_3^{2}x_4^{4}x_5^{24}$, & $L_{28}=x_1x_2^{2}x_4^{4}x_5^{24}$, & $L_{29}=x_1x_2^{2}x_3^{4}x_5^{24}$, & $L_{30}=x_1x_2^{2}x_3^{4}x_4^{24}$
\end{tabular}%
\end{center}
and that $Q^{\otimes 5}_{31})^{\omega_{(1)}^{>0}}$ is 1-dimensional with the basis $\{[L_{31}= x_1x_2^{2}x_3^{4}x_4^{8}x_5^{16}]\}.$ Thus, an admissible monomial basis of $(Q^{\otimes 5}_{31})^{\omega_{(1)}}$ is the set $\{[L_j]_{\omega_{(1)}} = [L_j]:\, 1\leq j\leq 31\}.$ Then it is not difficult to check that 
$$ (Q^{\otimes 5}_{31})^{\omega_{(1)}} \cong \Sigma_5(L_1) \bigoplus \Sigma_5(L_6) \bigoplus \Sigma_5(L_{16})\bigoplus \Sigma_5(L_{26})\bigoplus \Sigma_5(L_{31}),$$
where $$ \begin{array}{ll}
\medskip
\Sigma_5(L_1) &= \langle \{[L_j]:\, 1\leq j\leq 5\} \rangle, \ \ \Sigma_5(L_6) = \langle \{[L_j]:\, 6\leq j\leq 15\} \rangle \\
\medskip
\Sigma_5(L_{16}) &= \langle \{[L_j]:\, 16\leq j\leq 25\} \rangle, \ \ \Sigma_5(L_{26}) = \langle \{[L_j]:\, 26\leq j\leq 30\} \rangle, \\
\Sigma_5(L_{31}) &= \langle [L_{31}] \rangle.
\end{array}$$
Then, by using the homomorphisms $\sigma_d: \mathbb P^{\otimes 5}_{31}\to \mathbb P^{\otimes 5}_{31}$ for $1\leq d\leq 4$ and similar computations as in the proof of Proposition \ref{md14}, we obtain 
 $$ \begin{array}{ll}
\medskip
[\Sigma_5(L_1)]^{\Sigma_5} &= \langle [\sum_{1\leq j\leq 5}L_j] \rangle,\ \ \mbox{[}\Sigma_5(L_6)\mbox{]}^{\Sigma_5} = \langle [\sum_{6\leq j\leq 15}L_j] \rangle,\\
\medskip
\mbox{[}\Sigma_5(L_{16})\mbox{]}^{\Sigma_5} &= \langle [\sum_{16\leq j\leq 25}L_j] \rangle,\ \ \mbox{[}\Sigma_5(L_{26})\mbox{]}^{\Sigma_5} = \langle [\sum_{26\leq j\leq 30}L_j] \rangle,\\
\mbox{[}\Sigma_5(L_{31})\mbox{]}^{\Sigma_5} &= \langle [L_{31}] \rangle.
\end{array}$$
These show that if $[T]\in  [(Q^{\otimes 5}_{31})^{\omega_{(1)}}]^{GL_5}$ then $$T\equiv \gamma_1\sum_{1\leq j\leq 5}L_j + \gamma_2\sum_{6\leq j\leq 15}L_j +\gamma_3\sum_{16\leq j\leq 25}L_j +\gamma_4\sum_{26\leq j\leq 30}L_j +\gamma_5L_{31}$$
where $\gamma_j\in \mathbb Z_2,\, 1\leq j\leq 5.$ Acting the homomorphism $\sigma_5: \mathbb P^{\otimes 5}_{31}\to \mathbb P^{\otimes 5}_{31}$ on both sides of this equality and using the relation $\sigma_5(T) + T\equiv 0,$ we get
$$ (\gamma_1 + \gamma_2)L_4 + (\gamma_2 + \gamma_3)(L_9 + L_{10}+L_{11}) + (\gamma_3 + \gamma_4)(L_{17} + L_{18}+L_{19}) + (\gamma_4 + \gamma_5)L_{26} \equiv 0$$
and therefore $\gamma_1 = \gamma_2 = \gamma_3 = \gamma_4 = \gamma_5,$ that is $[(Q^{\otimes 5}_{31})^{\omega_{(1)}}]^{GL_5} = \langle [\sum_{1\leq j\leq 31}L_j] \rangle.$ Therefore  $[(Q^{\otimes 5}_{31})^{\omega_{(1)}}]^{GL_5}$ has dimension 1.

Now, we prove the proposition for $j = 3.$ Recall that the set $\{[Y_j]_{\omega_{(3)}}:\, 1\leq j\leq 70\}$ is an admissible monomial basis of $(Q^{\otimes 5}_{31})^{\omega_{(3)}} = (Q^{\otimes 5}_{31})^{\omega_{(1)}^{>0}}.$ By a simple computation, we find that 
$(Q^{\otimes 5}_{31})^{\omega_{(3)}} \cong \Sigma_5(Y_1)\bigoplus \Sigma_5(Y_{21})\bigoplus \mathbb U,$ where $\Sigma_5(Y_1) =\langle \{[Y_j]_{\omega_{(3)}}:\, 1\leq j\leq 20\} \rangle,$ $\Sigma_5(Y_{21}) =\langle \{[Y_j]_{\omega_{(3)}}:\, 21\leq j\leq 50\} \rangle$ and $\mathbb U = \langle \{[Y_j]_{\omega_{(3)}}:\, 51\leq j\leq 70\} \rangle,$ and that $[\Sigma_5(Y_1)]^{\Sigma_5} =\langle [\sum_{1\leq j\leq 20}Y_j]_{\omega_{(3)}} \rangle,$ $[\Sigma_5(Y_{21})]^{\Sigma_5} =\langle [\sum_{21\leq j\leq 50}Y_j]_{\omega_{(3)}} \rangle$ and $[\mathbb U]^{\Sigma_5} = 0.$ From these calculations, we find that for any $W\in \mathbb P^{\otimes 5}_{31}$ such that $[W]_{\omega_{(3)}}\in [(Q^{\otimes 5}_{31})^{\omega_{(3)}}]^{GL_5},$ then $W\equiv_{\omega_{(3)}} \beta_1\sum_{1\leq j\leq 20}Y_j + \beta_2\sum_{21\leq j\leq 50}Y_j$ where $\beta_1$ and $\beta_2$ belong to $\mathbb Z_2.$ Now, it is not too difficult to check that 
$$ \sigma_5(W) + W \equiv_{\omega_{(3)}} \beta_1(\sum_{1\leq j\leq 4}Y_j) + \beta_2(Y_8 + Y_{17} + Y_{18}) + \mbox{other terms}\ \equiv_{\omega_{(3)}} 0$$
and so $\beta_1 = \beta_2 = 0.$ This means that  $[(Q^{\otimes 5}_{31})^{\omega_{(3)}}]^{GL_5}$ is trivial. The proposition follows. 
\end{proof}
Thus, as argued above, we claim that the coinvariant $[P_{\mathscr A_2}((\mathbb P_{31}^{\otimes 5})^{*})]_{GL_5}$ is 2-dimensional. This completes the proof of the case $n = 31.$
\end{proof}

\begin{proof}[{\it Proof of Case n = 32}]

It suffices to show that the spaces of $GL_5$-invariants $[(Q^{\otimes 5}_{32})^{\omega'_{(j)}}]^{GL_5}$ are trivial, for any $j,$ $1\leq j\leq 3.$  Indeed, we prove this statement for $j = 1.$ The others can be obtained by similar computations. Let us recall that from the proof of Theorem \ref{dlc4}, we have $(Q^{\otimes 5}_{32})^{\omega'_{(1)}}\cong (Q^{\otimes 5}_{32})^{(\omega'_{(1)})^{0}}\bigoplus (Q^{\otimes 5}_{32})^{(\omega'_{(1)})^{>0}},$ where $\dim (Q^{\otimes 5}_{32})^{(\omega'_{(1)})^{0}} = 115$ and $\dim (Q^{\otimes 5}_{32})^{(\omega'_{(1)})^{>0}} = 9.$ A simple computation shows that a basis of $(Q^{\otimes 5}_{32})^{(\omega'_{(1)})^{0}}$ is the set of all classes represented by the admissible monomials $\widetilde{Z_j}$ for $1\leq j\leq 115$ (see Subsect.\ref{s32_1}). We put $A(\omega'_{(1)}) = \{\widetilde{Z_j}:\, 1\leq j\leq 115\},$ then from the calculations in Theorem \ref{dlc4}, we see that the set $[A(\omega'_{(1)}) \cup B(\omega'_{(1)}) \cup C(\omega'_{(1)})]_{\omega'_{(1)}} = [A(\omega'_{(1)}) \cup B(\omega'_{(1)}) \cup C(\omega'_{(1)})]$ is an admissible basis of $(Q^{\otimes 5}_{32})^{\omega'_{(1)}}$ and therefore  we have a direct summand decomposition of the $\Sigma_5$-modules:
$$ (Q^{\otimes 5}_{32})^{\omega'_{(1)}}  = \Sigma_5(\widetilde{Z_1})\bigoplus  \Sigma_5(\widetilde{Z_{21}}) \bigoplus  \Sigma_5(\widetilde{Z_{31}}, \widetilde{Z_{51}}) \bigoplus \Sigma_5(\widetilde{Z_{81}}, \widetilde{Z_{101}})\bigoplus \mathbb M,$$
where $$ \begin{array}{ll}
\medskip
&\Sigma_5(\widetilde{Z_1}) = \langle \{[\widetilde{Z_j}]:\, 1\leq j\leq 20\} \rangle, \ \ \Sigma_5(\widetilde{Z_{21}}) = \langle \{[\widetilde{Z_j}]:\, 21\leq j\leq 30\} \rangle, \\
\medskip
&\Sigma_5(\widetilde{Z_{31}}, \widetilde{Z_{51}}) = \langle \{[\widetilde{Z_j}]:\, 31\leq j\leq 80\} \rangle, \ \ \Sigma_5(\widetilde{Z_{81}}, \widetilde{Z_{101}}) = \langle \{[\widetilde{Z_j}]:\, 81\leq j\leq 115\} \rangle,\\
&\mathbb M = \langle \{[Z_j]:\, 1\leq j\leq 9\} \rangle.
\end{array}$$
Then, based on the homomorphisms $\sigma_1, \ldots, \sigma_4: \mathbb P^{\otimes 5}_{32}\to \mathbb P^{\otimes 5}_{32}$ and direct  calculations, it may be concluded that
$$  \begin{array}{ll}
\medskip
&[\Sigma_5(\widetilde{Z_1})]^{\Sigma_5} = \langle [\widehat{q_1}] \rangle,\ \mbox{with $\widehat{q_1}:= \sum_{1\leq j\leq 20}\widetilde{Z_j}$},\\
\medskip
&\mbox{[}\Sigma_5(\widetilde{Z_{21}})\mbox{]}^{\Sigma_5} = \langle [\widehat{q_2}] \rangle,\ \mbox{with $\widehat{q_2}:= \sum_{21\leq j\leq 30}\widetilde{Z_j}$},\\
\medskip
&\mbox{[}\Sigma_5(\widetilde{Z_{31}}, \widetilde{Z_{51}})\mbox{]}^{\Sigma_5} = \langle [\widehat{q_3}] \rangle,\ \mbox{with $\widehat{q_3}:= \widetilde{Z_{32}} + \widetilde{Z_{34}} +\widetilde{Z_{37}} +\widetilde{Z_{38}} + \widetilde{Z_{40}} +\widetilde{Z_{43}} +\widetilde{Z_{44}} +  \sum_{48\leq j\leq 80}\widetilde{Z_j}$},\\
&\mbox{[}\Sigma_5(\widetilde{Z_{81}}, \widetilde{Z_{101}})\mbox{]}^{\Sigma_5} = \langle [\widehat{q_4} + \widehat{q_5}], [\widehat{q_5} + \widehat{q_6}] \rangle,\ \mbox{where}\\
 &\qquad\widehat{q_4}:=\widetilde{Z_{82}} + \widetilde{Z_{83}} +\widetilde{Z_{85}} +\widetilde{Z_{86}} + \sum_{90\leq j\leq 95}\widetilde{Z_{j}},\\ 
 &\qquad\widehat{q_5}:=\widetilde{Z_{81}} + \widetilde{Z_{84}} +\widetilde{Z_{87}} +\widetilde{Z_{88}} + \widetilde{Z_{89}},\\ 
\medskip
 &\qquad\widehat{q_6}:= \sum_{96\leq j\leq 115}\widetilde{Z_{j}},\\ 
&\mbox{[}\mathbb M\mbox{]}^{\Sigma_5}= \langle [\widehat{q_7}] \rangle,\ \mbox{with $\widehat{q_7}:= \sum_{2\leq j\leq 9}Z_j$}.
\end{array}$$
Now let us consider $K\in \mathbb P^{\otimes 5}_{32}$ such that $[K]\in [(Q^{\otimes 5}_{32})^{\omega'_{(1)}}]^{GL_5}.$ Then, since $\Sigma_5\subset GL_5,$ by the above computations, we have $K\equiv \sum_{1\leq i\leq 3}\xi_i\widehat{q_i} + \xi_4(\widehat{q_4} + \widehat{q_5}) + \xi_5(\widehat{q_5} + \widehat{q_6}) + \xi_6\widehat{q_7}$ where $\xi_i\in \mathbb Z_2.$ By direct computations using the relation $\sigma_5(K)\equiv K,$ we obtain
$$ \xi_1\widetilde{Z_{7}} + (\xi_1 + \xi_3)\widetilde{Z_{10}} + (\xi_1 + \xi_2)\widetilde{Z_{16}} + \xi_4\widetilde{Z_{33}} + (\xi_3 + \xi_5)\widetilde{Z_{52}} + \xi_6Z_8 + \ \mbox{other terms}\ \equiv 0,$$
which leads to $\xi_i = 0$ for all $i.$ This claims the result when $n = 32,$ and therefore, the proof of Theorem \ref{dlc5} is completed.  
\end{proof}

\section{Appendix}

In this section, we describe the following sets: $(\mathscr C^{\otimes 5}_{14})^{\widetilde{\omega}_{(j)}}$ for $1\leq j\leq 3,$  $(\mathscr C^{\otimes 5}_{31})^{\omega_{(j)}^{>0}}$ for $j = 2, 3$ and $(\mathscr C^{\otimes 5}_{32})^{\omega'_{(j)}}$ for $1\leq j\leq 3.$ At the same time, we provide an algorithm in MAGMA for verifying the dimensions of $Q^{\otimes 5}_{31}$ and $Q^{\otimes 5}_{32}.$

\subsection{Admissible monomials in $(\mathscr C^{\otimes 5}_{14})^{\widetilde{\omega}_{(1)}}$}\label{s14_1}

\begin{center}
\begin{tabular}{llrr}
$W_{1}=x_4^{7}x_5^{7}$, & $W_{2}=x_3^{7}x_5^{7}$, & \multicolumn{1}{l}{$W_{3}=x_3^{7}x_4^{7}$,} & \multicolumn{1}{l}{$W_{4}=x_2^{7}x_5^{7}$,} \\
$W_{5}=x_2^{7}x_4^{7}$, & $W_{6}=x_2^{7}x_3^{7}$, & \multicolumn{1}{l}{$W_{7}=x_1^{7}x_5^{7}$,} & \multicolumn{1}{l}{$W_{8}=x_1^{7}x_4^{7}$,} \\
$W_{9}=x_1^{7}x_3^{7}$, & $W_{10}=x_1^{7}x_2^{7}$, & \multicolumn{1}{l}{$W_{11}=x_3x_4^{6}x_5^{7}$,} & \multicolumn{1}{l}{$W_{12}=x_3x_4^{7}x_5^{6}$,} \\
$W_{13}=x_3^{7}x_4x_5^{6}$, & $W_{14}=x_2x_4^{6}x_5^{7}$, & \multicolumn{1}{l}{$W_{15}=x_2x_4^{7}x_5^{6}$,} & \multicolumn{1}{l}{$W_{16}=x_2x_3^{6}x_5^{7}$,} \\
$W_{17}=x_2x_3^{6}x_4^{7}$, & $W_{18}=x_2x_3^{7}x_5^{6}$, & \multicolumn{1}{l}{$W_{19}=x_2x_3^{7}x_4^{6}$,} & \multicolumn{1}{l}{$W_{20}=x_2^{7}x_4x_5^{6}$,} \\
$W_{21}=x_2^{7}x_3x_5^{6}$, & $W_{22}=x_2^{7}x_3x_4^{6}$, & \multicolumn{1}{l}{$W_{23}=x_1x_4^{6}x_5^{7}$,} & \multicolumn{1}{l}{$W_{24}=x_1x_4^{7}x_5^{6}$,} \\
$W_{25}=x_1x_3^{6}x_5^{7}$, & $W_{26}=x_1x_3^{6}x_4^{7}$, & \multicolumn{1}{l}{$W_{27}=x_1x_3^{7}x_5^{6}$,} & \multicolumn{1}{l}{$W_{28}=x_1x_3^{7}x_4^{6}$,} \\
$W_{29}=x_1x_2^{6}x_5^{7}$, & $W_{30}=x_1x_2^{6}x_4^{7}$, & \multicolumn{1}{l}{$W_{31}=x_1x_2^{6}x_3^{7}$,} & \multicolumn{1}{l}{$W_{32}=x_1x_2^{7}x_5^{6}$,} \\
\end{tabular}%
\end{center}

\subsection{Admissible monomials in $(\mathscr C^{\otimes 5}_{14})^{\widetilde{\omega}_{(2)}}$}\label{s14_2}

\begin{center}
\begin{tabular}{lllr}
$W_{131}=x_1x_2^{2}x_3^{2}x_4^{7}x_5^{2}$, & $W_{132}=x_1x_2^{2}x_3^{2}x_4^{2}x_5^{7}$, & $W_{133}=x_1x_2^{2}x_3^{7}x_4^{2}x_5^{2}$, & \multicolumn{1}{l}{$W_{134}=x_1x_2^{7}x_3^{2}x_4^{2}x_5^{2}$,} \\
$W_{135}=x_1^{7}x_2x_3^{2}x_4^{2}x_5^{2}$, & $W_{136}=x_1x_2^{2}x_3^{2}x_4^{3}x_5^{6}$, & $W_{137}=x_1x_2^{2}x_3^{3}x_4^{2}x_5^{6}$, & \multicolumn{1}{l}{$W_{138}=x_1x_2^{2}x_3^{3}x_4^{6}x_5^{2}$,} \\
$W_{139}=x_1x_2^{3}x_3^{2}x_4^{2}x_5^{6}$, & $W_{140}=x_1x_2^{3}x_3^{2}x_4^{6}x_5^{2}$, & $W_{141}=x_1x_2^{3}x_3^{6}x_4^{2}x_5^{2}$, & \multicolumn{1}{l}{$W_{142}=x_1^{3}x_2x_3^{2}x_4^{2}x_5^{6}$,} \\
$W_{143}=x_1^{3}x_2x_3^{2}x_4^{6}x_5^{2}$, & $W_{144}=x_1^{3}x_2x_3^{6}x_4^{2}x_5^{2}$, & $W_{145}=x_1^{3}x_2^{5}x_3^{2}x_4^{2}x_5^{2}$. &  
\end{tabular}%
\end{center}

\subsection{Admissible monomials in $(\mathscr C^{\otimes 5}_{14})^{\widetilde{\omega}_{(3)}}$}\label{s14_3}

\begin{center}
\begin{tabular}{lllr}
$W_{146}=x_2x_3^{3}x_4^{3}x_5^{7}$, & $W_{147}=x_2x_3^{3}x_4^{7}x_5^{3}$, & $W_{148}=x_2x_3^{7}x_4^{3}x_5^{3}$, & \multicolumn{1}{l}{$W_{149}=x_2^{3}x_3x_4^{3}x_5^{7}$,} \\
$W_{150}=x_2^{3}x_3x_4^{7}x_5^{3}$, & $W_{151}=x_2^{3}x_3^{3}x_4x_5^{7}$, & $W_{152}=x_2^{3}x_3^{3}x_4^{3}x_5^{5}$, & \multicolumn{1}{l}{$W_{153}=x_2^{3}x_3^{3}x_4^{5}x_5^{3}$,} \\
$W_{154}=x_2^{3}x_3^{3}x_4^{7}x_5$, & $W_{155}=x_2^{3}x_3^{5}x_4^{3}x_5^{3}$, & $W_{156}=x_2^{3}x_3^{7}x_4x_5^{3}$, & \multicolumn{1}{l}{$W_{157}=x_2^{3}x_3^{7}x_4^{3}x_5$,} \\
$W_{158}=x_2^{7}x_3x_4^{3}x_5^{3}$, & $W_{159}=x_2^{7}x_3^{3}x_4x_5^{3}$, & $W_{160}=x_2^{7}x_3^{3}x_4^{3}x_5$, & \multicolumn{1}{l}{$W_{161}=x_1x_3^{3}x_4^{3}x_5^{7}$,} \\
$W_{162}=x_1x_3^{3}x_4^{7}x_5^{3}$, & $W_{163}=x_1x_3^{7}x_4^{3}x_5^{3}$, & $W_{164}=x_1x_2^{3}x_4^{3}x_5^{7}$, & \multicolumn{1}{l}{$W_{165}=x_1x_2^{3}x_4^{7}x_5^{3}$,} \\
$W_{166}=x_1x_2^{3}x_3^{3}x_5^{7}$, & $W_{167}=x_1x_2^{3}x_3^{3}x_4^{7}$, & $W_{168}=x_1x_2^{3}x_3^{7}x_5^{3}$, & \multicolumn{1}{l}{$W_{169}=x_1x_2^{3}x_3^{7}x_4^{3}$,} \\
$W_{170}=x_1x_2^{7}x_4^{3}x_5^{3}$, & $W_{171}=x_1x_2^{7}x_3^{3}x_5^{3}$, & $W_{172}=x_1x_2^{7}x_3^{3}x_4^{3}$, & \multicolumn{1}{l}{$W_{173}=x_1^{3}x_3x_4^{3}x_5^{7}$,} \\
$W_{174}=x_1^{3}x_3x_4^{7}x_5^{3}$, & $W_{175}=x_1^{3}x_3^{3}x_4x_5^{7}$, & $W_{176}=x_1^{3}x_3^{3}x_4^{3}x_5^{5}$, & \multicolumn{1}{l}{$W_{177}=x_1^{3}x_3^{3}x_4^{5}x_5^{3}$,} \\
\end{tabular}%
\end{center}

\newpage
\begin{center}
\begin{tabular}{lllr}
$W_{306}=x_1^{3}x_2^{5}x_3^{3}x_4x_5^{2}$, & $W_{307}=x_1^{3}x_2^{5}x_3^{3}x_4^{2}x_5$, & $W_{308}=x_1^{3}x_2^{7}x_3x_4x_5^{2}$, & \multicolumn{1}{l}{$W_{309}=x_1^{3}x_2^{7}x_3x_4^{2}x_5$,} \\
$W_{310}=x_1^{7}x_2x_3x_4^{2}x_5^{3}$, & $W_{311}=x_1^{7}x_2x_3x_4^{3}x_5^{2}$, & $W_{312}=x_1^{7}x_2x_3^{2}x_4x_5^{3}$, & \multicolumn{1}{l}{$W_{313}=x_1^{7}x_2x_3^{2}x_4^{3}x_5$,} \\
$W_{314}=x_1^{7}x_2x_3^{3}x_4x_5^{2}$, & $W_{315}=x_1^{7}x_2x_3^{3}x_4^{2}x_5$, & $W_{316}=x_1^{7}x_2^{3}x_3x_4x_5^{2}$, & \multicolumn{1}{l}{$W_{317}=x_1^{7}x_2^{3}x_3x_4^{2}x_5$,} \\
$W_{318}=x_1^{3}x_2^{4}x_3^{3}x_4x_5^{3}$, & $W_{319}=x_1^{3}x_2^{4}x_3^{3}x_4^{3}x_5$, & $W_{320}=x_1^{3}x_2^{4}x_3x_4^{3}x_5^{3}$. &  
\end{tabular}%
\end{center}

\subsection{Admissible monomials in $(\mathscr C^{\otimes 5}_{31})^{\omega_{(2)}^{>0}}$}\label{s31}

\begin{center}
\begin{tabular}{lllr}
$X_{1}=x_1x_2^{2}x_3x_4^{12}x_5^{15}$, & $X_{2}=x_1x_2^{2}x_3x_4^{15}x_5^{12}$, & $X_{3}=x_1x_2^{2}x_3^{15}x_4x_5^{12}$, & \multicolumn{1}{l}{$X_{4}=x_1x_2^{2}x_3x_4^{13}x_5^{14}$,} \\
$X_{5}=x_1x_2^{2}x_3^{13}x_4x_5^{14}$, & $X_{6}=x_1x_2^{2}x_3^{13}x_4^{14}x_5$, & $X_{7}=x_1x_2^{14}x_3x_4x_5^{14}$, & \multicolumn{1}{l}{$X_{8}=x_1x_2^{14}x_3x_4^{14}x_5$,} \\
$X_{9}=x_1x_2x_3^{14}x_4x_5^{14}$, & $X_{10}=x_1x_2x_3^{14}x_4^{14}x_5$, & $X_{11}=x_1x_2x_3^{2}x_4^{12}x_5^{15}$, & \multicolumn{1}{l}{$X_{12}=x_1x_2x_3^{2}x_4^{15}x_5^{12}$,} \\
$X_{13}=x_1x_2x_3^{2}x_4^{13}x_5^{14}$, & $X_{14}=x_1x_2^{3}x_3^{14}x_4x_5^{12}$, & $X_{15}=x_1x_2^{3}x_3^{12}x_4x_5^{14}$, & \multicolumn{1}{l}{$X_{16}=x_1x_2^{3}x_3^{12}x_4^{14}x_5$,} \\
$X_{17}=x_1x_2^{15}x_3^{2}x_4x_5^{12}$, & $X_{18}=x_1x_2x_3x_4^{14}x_5^{14}$, & $X_{19}=x_1x_2x_3^{3}x_4^{14}x_5^{12}$, & \multicolumn{1}{l}{$X_{20}=x_1x_2x_3^{15}x_4^{2}x_5^{12}$,} \\
$X_{21}=x_1x_2x_3^{3}x_4^{12}x_5^{14}$, & $X_{22}=x_1x_2^{3}x_3x_4^{14}x_5^{12}$, & $X_{23}=x_1x_2^{3}x_3x_4^{12}x_5^{14}$, & \multicolumn{1}{l}{$X_{24}=x_1x_2^{3}x_3^{13}x_4^{2}x_5^{12}$,} \\
$X_{25}=x_1x_2^{3}x_3^{5}x_4^{10}x_5^{12}$, & $X_{26}=x_1x_2^{15}x_3x_4^{2}x_5^{12}$, & $X_{27}=x_1^{3}x_2x_3^{14}x_4x_5^{12}$, & \multicolumn{1}{l}{$X_{28}=x_1^{3}x_2x_3^{12}x_4x_5^{14}$,} \\
$X_{29}=x_1^{3}x_2x_3^{12}x_4^{14}x_5$, & $X_{30}=x_1^{3}x_2x_3x_4^{14}x_5^{12}$, & $X_{31}=x_1^{3}x_2x_3x_4^{12}x_5^{14}$, & \multicolumn{1}{l}{$X_{32}=x_1^{3}x_2x_3^{13}x_4^{2}x_5^{12}$,} \\
\end{tabular}
\end{center}

\newpage
\begin{center}
\begin{tabular}{lllr}
$X_{165}=x_1x_2^{7}x_3^{11}x_4^{4}x_5^{8}$, & $X_{166}=x_1x_2^{3}x_3^{4}x_4^{9}x_5^{14}$, & $X_{167}=x_1x_2^{3}x_3^{4}x_4^{14}x_5^{9}$, & \multicolumn{1}{l}{$X_{168}=x_1x_2^{3}x_3^{14}x_4^{4}x_5^{9}$,} \\
$X_{169}=x_1x_2^{3}x_3^{7}x_4^{12}x_5^{8}$, & $X_{170}=x_1x_2^{7}x_3^{3}x_4^{12}x_5^{8}$, & $X_{171}=x_1x_2^{3}x_3^{5}x_4^{8}x_5^{14}$, & \multicolumn{1}{l}{$X_{172}=x_1x_2^{3}x_3^{5}x_4^{14}x_5^{8}$,} \\
$X_{173}=x_1x_2^{3}x_3^{14}x_4^{5}x_5^{8}$, & $X_{174}=x_1^{3}x_2^{5}x_3^{10}x_4^{4}x_5^{9}$, & $X_{175}=x_1^{3}x_2^{5}x_3^{10}x_4^{5}x_5^{8}$, & \multicolumn{1}{l}{$X_{176}=x_1^{3}x_2^{5}x_3^{2}x_4^{12}x_5^{9}$,} \\
$X_{177}=x_1^{3}x_2^{13}x_3^{2}x_4^{4}x_5^{9}$, & $X_{178}=x_1^{3}x_2^{5}x_3^{11}x_4^{4}x_5^{8}$, & $X_{179}=x_1^{3}x_2x_3^{6}x_4^{13}x_5^{8}$, & \multicolumn{1}{l}{$X_{180}=x_1^{3}x_2x_3^{6}x_4^{12}x_5^{9}$,} \\
$X_{181}=x_1^{3}x_2^{5}x_3^{2}x_4^{13}x_5^{8}$, & $X_{182}=x_1^{3}x_2^{13}x_3^{2}x_4^{5}x_5^{8}$, & $X_{183}=x_1^{3}x_2^{3}x_3^{12}x_4^{4}x_5^{9}$, & \multicolumn{1}{l}{$X_{184}=x_1^{3}x_2^{3}x_3^{4}x_4^{12}x_5^{9}$,} \\
$X_{185}=x_1^{3}x_2^{5}x_3^{3}x_4^{12}x_5^{8}$, & $X_{186}=x_1^{3}x_2^{3}x_3^{5}x_4^{12}x_5^{8}$, & $X_{187}=x_1^{3}x_2^{3}x_3^{12}x_4^{5}x_5^{8}$, & \multicolumn{1}{l}{$X_{188}=x_1^{3}x_2^{13}x_3^{3}x_4^{4}x_5^{8}$,} \\
$X_{189}=x_1^{3}x_2^{3}x_3^{13}x_4^{4}x_5^{8}$, & $X_{190}=x_1^{3}x_2^{3}x_3^{4}x_4^{13}x_5^{8}$, & $X_{191}=x_1^{3}x_2^{7}x_3^{9}x_4^{4}x_5^{8}$, & \multicolumn{1}{l}{$X_{192}=x_1^{3}x_2^{7}x_3^{4}x_4^{9}x_5^{8}$,} \\
$X_{193}=x_1^{3}x_2^{7}x_3^{4}x_4^{8}x_5^{9}$, & $X_{194}=x_1^{3}x_2x_3^{4}x_4^{9}x_5^{14}$, & $X_{195}=x_1^{3}x_2x_3^{4}x_4^{14}x_5^{9}$, & \multicolumn{1}{l}{$X_{196}=x_1^{3}x_2x_3^{14}x_4^{4}x_5^{9}$,} \\
\end{tabular}%
\end{center}

\subsection{Admissible monomials in $(\mathscr C^{\otimes 5}_{31})^{\omega_{(3)}^{>0}}$}\label{s32}

\begin{center}
\begin{tabular}{llrr}
$Y_{1}=x_1x_2^{3}x_3^{6}x_4^{6}x_5^{15}$, & $Y_{2}=x_1x_2^{3}x_3^{6}x_4^{15}x_5^{6}$, & \multicolumn{1}{l}{$Y_{3}=x_1x_2^{3}x_3^{15}x_4^{6}x_5^{6}$,} & \multicolumn{1}{l}{$Y_{4}=x_1x_2^{15}x_3^{3}x_4^{6}x_5^{6}$,} \\
$Y_{5}=x_1^{3}x_2x_3^{6}x_4^{6}x_5^{15}$, & $Y_{6}=x_1^{3}x_2x_3^{6}x_4^{15}x_5^{6}$, & \multicolumn{1}{l}{$Y_{7}=x_1^{3}x_2x_3^{15}x_4^{6}x_5^{6}$,} & \multicolumn{1}{l}{$Y_{8}=x_1^{3}x_2^{15}x_3x_4^{6}x_5^{6}$,} \\
$Y_{9}=x_1^{15}x_2x_3^{3}x_4^{6}x_5^{6}$, & $Y_{10}=x_1^{15}x_2^{3}x_3x_4^{6}x_5^{6}$, & \multicolumn{1}{l}{$Y_{11}=x_1^{3}x_2^{5}x_3^{2}x_4^{6}x_5^{15}$,} & \multicolumn{1}{l}{$Y_{12}=x_1^{3}x_2^{5}x_3^{2}x_4^{15}x_5^{6}$,} \\
$Y_{13}=x_1^{3}x_2^{5}x_3^{6}x_4^{2}x_5^{15}$, & $Y_{14}=x_1^{3}x_2^{5}x_3^{6}x_4^{15}x_5^{2}$, & \multicolumn{1}{l}{$Y_{15}=x_1^{3}x_2^{5}x_3^{15}x_4^{2}x_5^{6}$,} & \multicolumn{1}{l}{$Y_{16}=x_1^{3}x_2^{5}x_3^{15}x_4^{6}x_5^{2}$,} \\
$Y_{17}=x_1^{3}x_2^{15}x_3^{5}x_4^{2}x_5^{6}$, & $Y_{18}=x_1^{3}x_2^{15}x_3^{5}x_4^{6}x_5^{2}$, & \multicolumn{1}{l}{$Y_{19}=x_1^{15}x_2^{3}x_3^{5}x_4^{2}x_5^{6}$,} & \multicolumn{1}{l}{$Y_{20}=x_1^{15}x_2^{3}x_3^{5}x_4^{6}x_5^{2}$,} \\
$Y_{21}=x_1x_2^{3}x_3^{6}x_4^{7}x_5^{14}$, & $Y_{22}=x_1x_2^{3}x_3^{7}x_4^{6}x_5^{14}$, & \multicolumn{1}{l}{$Y_{23}=x_1x_2^{3}x_3^{7}x_4^{14}x_5^{6}$,} & \multicolumn{1}{l}{$Y_{24}=x_1x_2^{7}x_3^{3}x_4^{6}x_5^{14}$,} \\
$Y_{25}=x_1x_2^{7}x_3^{3}x_4^{14}x_5^{6}$, & $Y_{26}=x_1^{3}x_2x_3^{6}x_4^{7}x_5^{14}$, & \multicolumn{1}{l}{$Y_{27}=x_1^{3}x_2x_3^{7}x_4^{6}x_5^{14}$,} & \multicolumn{1}{l}{$Y_{28}=x_1^{3}x_2x_3^{7}x_4^{14}x_5^{6}$,} \\
$Y_{29}=x_1^{3}x_2^{7}x_3x_4^{6}x_5^{14}$, & $Y_{30}=x_1^{3}x_2^{7}x_3x_4^{14}x_5^{6}$, & \multicolumn{1}{l}{$Y_{31}=x_1^{7}x_2x_3^{3}x_4^{6}x_5^{14}$,} & \multicolumn{1}{l}{$Y_{32}=x_1^{7}x_2x_3^{3}x_4^{14}x_5^{6}$,} \\
\end{tabular}%
\end{center}

\subsection{Admissible monomials in $(\mathscr C^{\otimes 5}_{32})^{\omega'_{(1)}}$}\label{s32_1}

We have $(\mathscr C^{\otimes 5}_{32})^{\omega'_{(1)}} = (\mathscr C^{\otimes 5}_{32})^{(\omega'_{(1)})^{0}}\cup (\mathscr C^{\otimes 5}_{32})^{(\omega'_{(1)})^{>0}},$ where 
$$ \begin{array}{ll}
\medskip
&(\mathscr C^{\otimes 5}_{32})^{(\omega'_{(1)})^{0}} = A(\omega'_{(1)}) = \{\widetilde{Z_j}:\, 1\leq j\leq 115\},\\
 &(\mathscr C^{\otimes 5}_{32})^{(\omega'_{(1)})^{>0}} = B(\omega'_{(1)}) \cup C(\omega'_{(1)}) = \{Z_j:\, 1\leq j\leq 9\}.
\end{array}$$ 

The admissible monomials $\widetilde{Z_j}$ and $Z_j$ are determined as follows:

\begin{center}
\begin{tabular}{llll}
$\widetilde {Z_{1}}=x_4x_5^{31}$, & $\widetilde {Z_{2}}=x_4^{31}x_5$, & $\widetilde {Z_{3}}=x_3x_5^{31}$, & $\widetilde {Z_{4}}=x_3x_4^{31}$, \\
$\widetilde {Z_{5}}=x_3^{31}x_5$, & $\widetilde {Z_{6}}=x_3^{31}x_4$, & $\widetilde {Z_{7}}=x_2x_5^{31}$, & $\widetilde {Z_{8}}=x_2x_4^{31}$, \\
$\widetilde {Z_{9}}=x_2x_3^{31}$, & $\widetilde {Z_{10}}=x_2^{31}x_5$, & $\widetilde {Z_{11}}=x_2^{31}x_4$, & $\widetilde {Z_{12}}=x_2^{31}x_3$, \\
$\widetilde {Z_{13}}=x_1x_5^{31}$, & $\widetilde {Z_{14}}=x_1x_4^{31}$, & $\widetilde {Z_{15}}=x_1x_3^{31}$, & $\widetilde {Z_{16}}=x_1x_2^{31}$, \\
\end{tabular}%
\end{center}
\newpage

\begin{center}
\begin{tabular}{llll}
$\widetilde {Z_{17}}=x_1^{31}x_5$, & $\widetilde {Z_{18}}=x_1^{31}x_4$, & $\widetilde {Z_{19}}=x_1^{31}x_3$, & $\widetilde {Z_{20}}=x_1^{31}x_2$, \\
$\widetilde {Z_{21}}=x_4^{3}x_5^{29}$, & $\widetilde {Z_{22}}=x_3^{3}x_5^{29}$, & $\widetilde {Z_{23}}=x_3^{3}x_4^{29}$, & $\widetilde {Z_{24}}=x_2^{3}x_5^{29}$, \\
$\widetilde {Z_{25}}=x_2^{3}x_4^{29}$, & $\widetilde {Z_{26}}=x_2^{3}x_3^{29}$, & $\widetilde {Z_{27}}=x_1^{3}x_5^{29}$, & $\widetilde {Z_{28}}=x_1^{3}x_4^{29}$, \\
$\widetilde {Z_{29}}=x_1^{3}x_3^{29}$, & $\widetilde {Z_{30}}=x_1^{3}x_2^{29}$, & $\widetilde {Z_{31}}=x_3x_4x_5^{30}$, & $\widetilde {Z_{32}}=x_3x_4^{30}x_5$, \\
$\widetilde {Z_{33}}=x_2x_4x_5^{30}$, & $\widetilde {Z_{34}}=x_2x_4^{30}x_5$, & $\widetilde {Z_{35}}=x_2x_3x_5^{30}$, & $\widetilde {Z_{36}}=x_2x_3x_4^{30}$, \\
$\widetilde {Z_{37}}=x_2x_3^{30}x_5$, & $\widetilde {Z_{38}}=x_2x_3^{30}x_4$, & $\widetilde {Z_{39}}=x_1x_4x_5^{30}$, & $\widetilde {Z_{40}}=x_1x_4^{30}x_5$, \\
$\widetilde {Z_{41}}=x_1x_3x_5^{30}$, & $\widetilde {Z_{42}}=x_1x_3x_4^{30}$, & $\widetilde {Z_{43}}=x_1x_3^{30}x_5$, & $\widetilde {Z_{44}}=x_1x_3^{30}x_4$, \\
$\widetilde {Z_{45}}=x_1x_2x_5^{30}$, & $\widetilde {Z_{46}}=x_1x_2x_4^{30}$, & $\widetilde {Z_{47}}=x_1x_2x_3^{30}$, & $\widetilde {Z_{48}}=x_1x_2^{30}x_5$, \\
\end{tabular}%
\end{center}

\subsection{Admissible monomials in $(\mathscr C^{\otimes 5}_{32})^{\omega'_{(2)}}$}\label{s32_2}

We have $(\mathscr C^{\otimes 5}_{32})^{\omega'_{(2)}} = (\mathscr C^{\otimes 5}_{32})^{(\omega'_{(2)})^{0}}\cup (\mathscr C^{\otimes 5}_{32})^{(\omega'_{(2)})^{>0}},$ where 
$$ \begin{array}{ll}
\medskip
&(\mathscr C^{\otimes 5}_{32})^{(\omega'_{(2)})^{0}} = A(\omega'_{(2)}) = \{\widetilde{Z'_j}:\, 1\leq j\leq 175\},\\
 &(\mathscr C^{\otimes 5}_{32})^{(\omega'_{(2)})^{>0}} = B(\omega'_{(2)}) \cup C(\omega'_{(2)}) = \{Z'_j:\, 1\leq j\leq 310\}.
\end{array}$$ 
The admissible monomials $\widetilde{Z'_j}$ and $Z'_j$ are determined as follows:

\begin{center}
\begin{tabular}{lrrr}
$\widetilde {Z'_{1}}=x_2x_3x_4^{15}x_5^{15}$, & \multicolumn{1}{l}{$\widetilde {Z'_{2}}=x_2x_3^{3}x_4^{13}x_5^{15}$,} & \multicolumn{1}{l}{$\widetilde {Z'_{3}}=x_2x_3^{3}x_4^{15}x_5^{13}$,} & \multicolumn{1}{l}{$\widetilde {Z'_{4}}=x_2x_3^{7}x_4^{11}x_5^{13}$,} \\
$\widetilde {Z'_{5}}=x_2x_3^{15}x_4x_5^{15}$, & \multicolumn{1}{l}{$\widetilde {Z'_{6}}=x_2x_3^{15}x_4^{3}x_5^{13}$,} & \multicolumn{1}{l}{$\widetilde {Z'_{7}}=x_2x_3^{15}x_4^{15}x_5$,} & \multicolumn{1}{l}{$\widetilde {Z'_{8}}=x_2^{3}x_3x_4^{13}x_5^{15}$,} \\
$\widetilde {Z'_{9}}=x_2^{3}x_3x_4^{15}x_5^{13}$, & \multicolumn{1}{l}{$\widetilde {Z'_{10}}=x_2^{3}x_3^{3}x_4^{13}x_5^{13}$,} & \multicolumn{1}{l}{$\widetilde {Z'_{11}}=x_2^{3}x_3^{5}x_4^{9}x_5^{15}$,} & \multicolumn{1}{l}{$\widetilde {Z'_{12}}=x_2^{3}x_3^{5}x_4^{11}x_5^{13}$,} \\
$\widetilde {Z'_{13}}=x_2^{3}x_3^{5}x_4^{15}x_5^{9}$, & \multicolumn{1}{l}{$\widetilde {Z'_{14}}=x_2^{3}x_3^{7}x_4^{9}x_5^{13}$,} & \multicolumn{1}{l}{$\widetilde {Z'_{15}}=x_2^{3}x_3^{7}x_4^{13}x_5^{9}$,} & \multicolumn{1}{l}{$\widetilde {Z'_{16}}=x_2^{3}x_3^{13}x_4x_5^{15}$,} \\
$\widetilde {Z'_{17}}=x_2^{3}x_3^{13}x_4^{3}x_5^{13}$, & \multicolumn{1}{l}{$\widetilde {Z'_{18}}=x_2^{3}x_3^{13}x_4^{15}x_5$,} & \multicolumn{1}{l}{$\widetilde {Z'_{19}}=x_2^{3}x_3^{15}x_4x_5^{13}$,} & \multicolumn{1}{l}{$\widetilde {Z'_{20}}=x_2^{3}x_3^{15}x_4^{5}x_5^{9}$,} \\
$\widetilde {Z'_{21}}=x_2^{3}x_3^{15}x_4^{13}x_5$, & \multicolumn{1}{l}{$\widetilde {Z'_{22}}=x_2^{7}x_3x_4^{11}x_5^{13}$,} & \multicolumn{1}{l}{$\widetilde {Z'_{23}}=x_2^{7}x_3^{3}x_4^{9}x_5^{13}$,} & \multicolumn{1}{l}{$\widetilde {Z'_{24}}=x_2^{7}x_3^{3}x_4^{13}x_5^{9}$,} \\
$\widetilde {Z'_{25}}=x_2^{7}x_3^{7}x_4^{9}x_5^{9}$, & \multicolumn{1}{l}{$\widetilde {Z'_{26}}=x_2^{7}x_3^{11}x_4x_5^{13}$,} & \multicolumn{1}{l}{$\widetilde {Z'_{27}}=x_2^{7}x_3^{11}x_4^{5}x_5^{9}$,} & \multicolumn{1}{l}{$\widetilde {Z'_{28}}=x_2^{7}x_3^{11}x_4^{13}x_5$,} \\
\end{tabular}
\end{center}

\newpage
\begin{center}
\begin{tabular}{lrrr}
$\widetilde {Z'_{33}}=x_2^{15}x_3^{3}x_4^{5}x_5^{9}$, & \multicolumn{1}{l}{$\widetilde {Z'_{34}}=x_2^{15}x_3^{3}x_4^{13}x_5$,} & \multicolumn{1}{l}{$\widetilde {Z'_{35}}=x_2^{15}x_3^{15}x_4x_5$,} & \multicolumn{1}{l}{$\widetilde {Z'_{36}}=x_1x_3x_4^{15}x_5^{15}$,} \\
$\widetilde {Z'_{37}}=x_1x_3^{3}x_4^{13}x_5^{15}$, & \multicolumn{1}{l}{$\widetilde {Z'_{38}}=x_1x_3^{3}x_4^{15}x_5^{13}$,} & \multicolumn{1}{l}{$\widetilde {Z'_{39}}=x_1x_3^{7}x_4^{11}x_5^{13}$,} & \multicolumn{1}{l}{$\widetilde {Z'_{40}}=x_1x_3^{15}x_4x_5^{15}$,} \\
$\widetilde {Z'_{41}}=x_1x_3^{15}x_4^{3}x_5^{13}$, & \multicolumn{1}{l}{$\widetilde {Z'_{42}}=x_1x_3^{15}x_4^{15}x_5$,} & \multicolumn{1}{l}{$\widetilde {Z'_{43}}=x_1x_2x_4^{15}x_5^{15}$,} & \multicolumn{1}{l}{$\widetilde {Z'_{44}}=x_1x_2x_3^{15}x_5^{15}$,} \\
$\widetilde {Z'_{45}}=x_1x_2x_3^{15}x_4^{15}$, & \multicolumn{1}{l}{$\widetilde {Z'_{46}}=x_1x_2^{3}x_4^{13}x_5^{15}$,} & \multicolumn{1}{l}{$\widetilde {Z'_{47}}=x_1x_2^{3}x_4^{15}x_5^{13}$,} & \multicolumn{1}{l}{$\widetilde {Z'_{48}}=x_1x_2^{3}x_3^{13}x_5^{15}$,} \\
$\widetilde {Z'_{49}}=x_1x_2^{3}x_3^{13}x_4^{15}$, & \multicolumn{1}{l}{$\widetilde {Z'_{50}}=x_1x_2^{3}x_3^{15}x_5^{13}$,} & \multicolumn{1}{l}{$\widetilde {Z'_{51}}=x_1x_2^{3}x_3^{15}x_4^{13}$,} & \multicolumn{1}{l}{$\widetilde {Z'_{52}}=x_1x_2^{7}x_4^{11}x_5^{13}$,} \\
$\widetilde {Z'_{53}}=x_1x_2^{7}x_3^{11}x_5^{13}$, & \multicolumn{1}{l}{$\widetilde {Z'_{54}}=x_1x_2^{7}x_3^{11}x_4^{13}$,} & \multicolumn{1}{l}{$\widetilde {Z'_{55}}=x_1x_2^{15}x_4x_5^{15}$,} & \multicolumn{1}{l}{$\widetilde {Z'_{56}}=x_1x_2^{15}x_4^{3}x_5^{13}$,} \\
$\widetilde {Z'_{57}}=x_1x_2^{15}x_4^{15}x_5$, & \multicolumn{1}{l}{$\widetilde {Z'_{58}}=x_1x_2^{15}x_3x_5^{15}$,} & \multicolumn{1}{l}{$\widetilde {Z'_{59}}=x_1x_2^{15}x_3x_4^{15}$,} & \multicolumn{1}{l}{$\widetilde {Z'_{60}}=x_1x_2^{15}x_3^{3}x_5^{13}$,} \\
\end{tabular}
\end{center}

\newpage
\begin{center}
\begin{tabular}{lrrr}
$Z'_{26}=x_1x_2^{2}x_3^{7}x_4^{13}x_5^{9}$, & \multicolumn{1}{l}{$Z'_{27}=x_1x_2^{2}x_3^{13}x_4x_5^{15}$,} & \multicolumn{1}{l}{$Z'_{28}=x_1x_2^{2}x_3^{13}x_4^{3}x_5^{13}$,} & \multicolumn{1}{l}{$Z'_{29}=x_1x_2^{2}x_3^{13}x_4^{15}x_5$,} \\
$Z'_{30}=x_1x_2^{2}x_3^{15}x_4x_5^{13}$, & \multicolumn{1}{l}{$Z'_{31}=x_1x_2^{2}x_3^{15}x_4^{5}x_5^{9}$,} & \multicolumn{1}{l}{$Z'_{32}=x_1x_2^{2}x_3^{15}x_4^{13}x_5$,} & \multicolumn{1}{l}{$Z'_{33}=x_1x_2^{3}x_3x_4^{12}x_5^{15}$,} \\
$Z'_{34}=x_1x_2^{3}x_3x_4^{13}x_5^{14}$, & \multicolumn{1}{l}{$Z'_{35}=x_1x_2^{3}x_3x_4^{14}x_5^{13}$,} & \multicolumn{1}{l}{$Z'_{36}=x_1x_2^{3}x_3x_4^{15}x_5^{12}$,} & \multicolumn{1}{l}{$Z'_{37}=x_1x_2^{3}x_3^{2}x_4^{13}x_5^{13}$,} \\
$Z'_{38}=x_1x_2^{3}x_3^{3}x_4^{12}x_5^{13}$, & \multicolumn{1}{l}{$Z'_{39}=x_1x_2^{3}x_3^{3}x_4^{13}x_5^{12}$,} & \multicolumn{1}{l}{$Z'_{40}=x_1x_2^{3}x_3^{4}x_4^{9}x_5^{15}$,} & \multicolumn{1}{l}{$Z'_{41}=x_1x_2^{3}x_3^{4}x_4^{11}x_5^{13}$,} \\
$Z'_{42}=x_1x_2^{3}x_3^{4}x_4^{15}x_5^{9}$, & \multicolumn{1}{l}{$Z'_{43}=x_1x_2^{3}x_3^{5}x_4^{8}x_5^{15}$,} & \multicolumn{1}{l}{$Z'_{44}=x_1x_2^{3}x_3^{5}x_4^{9}x_5^{14}$,} & \multicolumn{1}{l}{$Z'_{45}=x_1x_2^{3}x_3^{5}x_4^{10}x_5^{13}$,} \\
$Z'_{46}=x_1x_2^{3}x_3^{5}x_4^{11}x_5^{12}$, & \multicolumn{1}{l}{$Z'_{47}=x_1x_2^{3}x_3^{5}x_4^{14}x_5^{9}$,} & \multicolumn{1}{l}{$Z'_{48}=x_1x_2^{3}x_3^{5}x_4^{15}x_5^{8}$,} & \multicolumn{1}{l}{$Z'_{49}=x_1x_2^{3}x_3^{6}x_4^{9}x_5^{13}$,} \\
$Z'_{50}=x_1x_2^{3}x_3^{6}x_4^{13}x_5^{9}$, & \multicolumn{1}{l}{$Z'_{51}=x_1x_2^{3}x_3^{7}x_4^{8}x_5^{13}$,} & \multicolumn{1}{l}{$Z'_{52}=x_1x_2^{3}x_3^{7}x_4^{9}x_5^{12}$,} & \multicolumn{1}{l}{$Z'_{53}=x_1x_2^{3}x_3^{7}x_4^{12}x_5^{9}$,} \\
$Z'_{54}=x_1x_2^{3}x_3^{7}x_4^{13}x_5^{8}$, & \multicolumn{1}{l}{$Z'_{55}=x_1x_2^{3}x_3^{12}x_4x_5^{15}$,} & \multicolumn{1}{l}{$Z'_{56}=x_1x_2^{3}x_3^{12}x_4^{3}x_5^{13}$,} & \multicolumn{1}{l}{$Z'_{57}=x_1x_2^{3}x_3^{12}x_4^{15}x_5$,} \\
\end{tabular}
\end{center}

\newpage
\begin{center}
\begin{tabular}{lrrr}
$Z'_{218}=x_1^{7}x_2x_3^{2}x_4^{9}x_5^{13}$, & \multicolumn{1}{l}{$Z'_{219}=x_1^{7}x_2x_3^{2}x_4^{13}x_5^{9}$,} & \multicolumn{1}{l}{$Z'_{220}=x_1^{7}x_2x_3^{3}x_4^{8}x_5^{13}$,} & \multicolumn{1}{l}{$Z'_{221}=x_1^{7}x_2x_3^{3}x_4^{9}x_5^{12}$,} \\
$Z'_{222}=x_1^{7}x_2x_3^{3}x_4^{12}x_5^{9}$, & \multicolumn{1}{l}{$Z'_{223}=x_1^{7}x_2x_3^{3}x_4^{13}x_5^{8}$,} & \multicolumn{1}{l}{$Z'_{224}=x_1^{7}x_2x_3^{6}x_4^{9}x_5^{9}$,} & \multicolumn{1}{l}{$Z'_{225}=x_1^{7}x_2x_3^{7}x_4^{8}x_5^{9}$,} \\
$Z'_{226}=x_1^{7}x_2x_3^{7}x_4^{9}x_5^{8}$, & \multicolumn{1}{l}{$Z'_{227}=x_1^{7}x_2x_3^{10}x_4x_5^{13}$,} & \multicolumn{1}{l}{$Z'_{228}=x_1^{7}x_2x_3^{10}x_4^{5}x_5^{9}$,} & \multicolumn{1}{l}{$Z'_{229}=x_1^{7}x_2x_3^{10}x_4^{13}x_5$,} \\
$Z'_{230}=x_1^{7}x_2x_3^{11}x_4x_5^{12}$, & \multicolumn{1}{l}{$Z'_{231}=x_1^{7}x_2x_3^{11}x_4^{4}x_5^{9}$,} & \multicolumn{1}{l}{$Z'_{232}=x_1^{7}x_2x_3^{11}x_4^{5}x_5^{8}$,} & \multicolumn{1}{l}{$Z'_{233}=x_1^{7}x_2x_3^{11}x_4^{12}x_5$,} \\
$Z'_{234}=x_1^{7}x_2^{3}x_3x_4^{8}x_5^{13}$, & \multicolumn{1}{l}{$Z'_{235}=x_1^{7}x_2^{3}x_3x_4^{9}x_5^{12}$,} & \multicolumn{1}{l}{$Z'_{236}=x_1^{7}x_2^{3}x_3x_4^{12}x_5^{9}$,} & \multicolumn{1}{l}{$Z'_{237}=x_1^{7}x_2^{3}x_3x_4^{13}x_5^{8}$,} \\
$Z'_{238}=x_1^{7}x_2^{3}x_3^{5}x_4^{8}x_5^{9}$, & \multicolumn{1}{l}{$Z'_{239}=x_1^{7}x_2^{3}x_3^{5}x_4^{9}x_5^{8}$,} & \multicolumn{1}{l}{$Z'_{240}=x_1^{7}x_2^{3}x_3^{9}x_4x_5^{12}$,} & \multicolumn{1}{l}{$Z'_{241}=x_1^{7}x_2^{3}x_3^{9}x_4^{4}x_5^{9}$,} \\
$Z'_{242}=x_1^{7}x_2^{3}x_3^{9}x_4^{5}x_5^{8}$, & \multicolumn{1}{l}{$Z'_{243}=x_1^{7}x_2^{3}x_3^{9}x_4^{12}x_5$,} & \multicolumn{1}{l}{$Z'_{244}=x_1^{7}x_2^{3}x_3^{13}x_4x_5^{8}$,} & \multicolumn{1}{l}{$Z'_{245}=x_1^{7}x_2^{7}x_3x_4^{8}x_5^{9}$,} \\
$Z'_{246}=x_1^{7}x_2^{7}x_3x_4^{9}x_5^{8}$, & \multicolumn{1}{l}{$Z'_{247}=x_1^{7}x_2^{7}x_3^{9}x_4x_5^{8}$,} & \multicolumn{1}{l}{$Z'_{248}=x_1^{7}x_2^{11}x_3x_4x_5^{12}$,} & \multicolumn{1}{l}{$Z'_{249}=x_1^{7}x_2^{11}x_3x_4^{4}x_5^{9}$,} \\
\end{tabular}%
\end{center}

\subsection{Admissible monomials in $(\mathscr C^{\otimes 5}_{32})^{\omega'_{(3)}}$}\label{s32_3}

We have $(\mathscr C^{\otimes 5}_{32})^{\omega'_{(3)}} = (\mathscr C^{\otimes 5}_{32})^{(\omega'_{(3)})^{0}}\cup (\mathscr C^{\otimes 5}_{32})^{(\omega'_{(3)})^{>0}},$ where 
$$ \begin{array}{ll}
\medskip
&(\mathscr C^{\otimes 5}_{32})^{(\omega'_{(3)})^{0}} = A(\omega'_{(3)}) = \{\widetilde{Z''_j}:\, 1\leq j\leq 75\},\\
 &(\mathscr C^{\otimes 5}_{32})^{(\omega'_{(3)})^{>0}} = B(\omega'_{(3)}) \cup C(\omega'_{(3)}) = \{Z''_j:\, 1\leq j\leq 320\}.
\end{array}$$ 
The admissible monomials $\widetilde{Z''_j}$ and $Z''_j$ are determined as follows:

\begin{center}
\begin{tabular}{lllr}
$\widetilde {Z''_{1}}=x_2^{3}x_3^{7}x_4^{7}x_5^{15}$, & $\widetilde {Z''_{2}}=x_2^{3}x_3^{7}x_4^{15}x_5^{7}$, & $\widetilde {Z''_{3}}=x_2^{3}x_3^{15}x_4^{7}x_5^{7}$, & \multicolumn{1}{l}{$\widetilde {Z''_{4}}=x_2^{7}x_3^{3}x_4^{7}x_5^{15}$,} \\
$\widetilde {Z''_{5}}=x_2^{7}x_3^{3}x_4^{15}x_5^{7}$, & $\widetilde {Z''_{6}}=x_2^{7}x_3^{7}x_4^{3}x_5^{15}$, & $\widetilde {Z''_{7}}=x_2^{7}x_3^{7}x_4^{7}x_5^{11}$, & \multicolumn{1}{l}{$\widetilde {Z''_{8}}=x_2^{7}x_3^{7}x_4^{11}x_5^{7}$,} \\
$\widetilde {Z''_{9}}=x_2^{7}x_3^{7}x_4^{15}x_5^{3}$, & $\widetilde {Z''_{10}}=x_2^{7}x_3^{11}x_4^{7}x_5^{7}$, & $\widetilde {Z''_{11}}=x_2^{7}x_3^{15}x_4^{3}x_5^{7}$, & \multicolumn{1}{l}{$\widetilde {Z''_{12}}=x_2^{7}x_3^{15}x_4^{7}x_5^{3}$,} \\
$\widetilde {Z''_{13}}=x_2^{15}x_3^{3}x_4^{7}x_5^{7}$, & $\widetilde {Z''_{14}}=x_2^{15}x_3^{7}x_4^{3}x_5^{7}$, & $\widetilde {Z''_{15}}=x_2^{15}x_3^{7}x_4^{7}x_5^{3}$, & \multicolumn{1}{l}{$\widetilde {Z''_{16}}=x_1^{3}x_3^{7}x_4^{7}x_5^{15}$,} \\
$\widetilde {Z''_{17}}=x_1^{3}x_3^{7}x_4^{15}x_5^{7}$, & $\widetilde {Z''_{18}}=x_1^{3}x_3^{15}x_4^{7}x_5^{7}$, & $\widetilde {Z''_{19}}=x_1^{3}x_2^{7}x_4^{7}x_5^{15}$, & \multicolumn{1}{l}{$\widetilde {Z''_{20}}=x_1^{3}x_2^{7}x_4^{15}x_5^{7}$,} \\
$\widetilde {Z''_{21}}=x_1^{3}x_2^{7}x_3^{7}x_5^{15}$, & $\widetilde {Z''_{22}}=x_1^{3}x_2^{7}x_3^{7}x_4^{15}$, & $\widetilde {Z''_{23}}=x_1^{3}x_2^{7}x_3^{15}x_5^{7}$, & \multicolumn{1}{l}{$\widetilde {Z''_{24}}=x_1^{3}x_2^{7}x_3^{15}x_4^{7}$,} \\
$\widetilde {Z''_{25}}=x_1^{3}x_2^{15}x_4^{7}x_5^{7}$, & $\widetilde {Z''_{26}}=x_1^{3}x_2^{15}x_3^{7}x_5^{7}$, & $\widetilde {Z''_{27}}=x_1^{3}x_2^{15}x_3^{7}x_4^{7}$, & \multicolumn{1}{l}{$\widetilde {Z''_{28}}=x_1^{7}x_3^{3}x_4^{7}x_5^{15}$,} \\
$\widetilde {Z''_{29}}=x_1^{7}x_3^{3}x_4^{15}x_5^{7}$, & $\widetilde {Z''_{30}}=x_1^{7}x_3^{7}x_4^{3}x_5^{15}$, & $\widetilde {Z''_{31}}=x_1^{7}x_3^{7}x_4^{7}x_5^{11}$, & \multicolumn{1}{l}{$\widetilde {Z''_{32}}=x_1^{7}x_3^{7}x_4^{11}x_5^{7}$,} \\
\end{tabular}
\end{center}

\newpage
\begin{center}
\begin{tabular}{lllr}
$\widetilde {Z''_{49}}=x_1^{7}x_2^{7}x_3^{7}x_5^{11}$, & $\widetilde {Z''_{50}}=x_1^{7}x_2^{7}x_3^{7}x_4^{11}$, & $\widetilde {Z''_{51}}=x_1^{7}x_2^{7}x_3^{11}x_5^{7}$, & \multicolumn{1}{l}{$\widetilde {Z''_{52}}=x_1^{7}x_2^{7}x_3^{11}x_4^{7}$,} \\
$\widetilde {Z''_{53}}=x_1^{7}x_2^{7}x_3^{15}x_5^{3}$, & $\widetilde {Z''_{54}}=x_1^{7}x_2^{7}x_3^{15}x_4^{3}$, & $\widetilde {Z''_{55}}=x_1^{7}x_2^{11}x_4^{7}x_5^{7}$, & \multicolumn{1}{l}{$\widetilde {Z''_{56}}=x_1^{7}x_2^{11}x_3^{7}x_5^{7}$,} \\
$\widetilde {Z''_{57}}=x_1^{7}x_2^{11}x_3^{7}x_4^{7}$, & $\widetilde {Z''_{58}}=x_1^{7}x_2^{15}x_4^{3}x_5^{7}$, & $\widetilde {Z''_{59}}=x_1^{7}x_2^{15}x_4^{7}x_5^{3}$, & \multicolumn{1}{l}{$\widetilde {Z''_{60}}=x_1^{7}x_2^{15}x_3^{3}x_5^{7}$,} \\
$\widetilde {Z''_{61}}=x_1^{7}x_2^{15}x_3^{3}x_4^{7}$, & $\widetilde {Z''_{62}}=x_1^{7}x_2^{15}x_3^{7}x_5^{3}$, & $\widetilde {Z''_{63}}=x_1^{7}x_2^{15}x_3^{7}x_4^{3}$, & \multicolumn{1}{l}{$\widetilde {Z''_{64}}=x_1^{15}x_3^{3}x_4^{7}x_5^{7}$,} \\
$\widetilde {Z''_{65}}=x_1^{15}x_3^{7}x_4^{3}x_5^{7}$, & $\widetilde {Z''_{66}}=x_1^{15}x_3^{7}x_4^{7}x_5^{3}$, & $\widetilde {Z''_{67}}=x_1^{15}x_2^{3}x_4^{7}x_5^{7}$, & \multicolumn{1}{l}{$\widetilde {Z''_{68}}=x_1^{15}x_2^{3}x_3^{7}x_5^{7}$,} \\
$\widetilde {Z''_{69}}=x_1^{15}x_2^{3}x_3^{7}x_4^{7}$, & $\widetilde {Z''_{70}}=x_1^{15}x_2^{7}x_4^{3}x_5^{7}$, & $\widetilde {Z''_{71}}=x_1^{15}x_2^{7}x_4^{7}x_5^{3}$, & \multicolumn{1}{l}{$\widetilde {Z''_{72}}=x_1^{15}x_2^{7}x_3^{3}x_5^{7}$,} \\
$\widetilde {Z''_{73}}=x_1^{15}x_2^{7}x_3^{3}x_4^{7}$, & $\widetilde {Z''_{74}}=x_1^{15}x_2^{7}x_3^{7}x_5^{3}$, & $\widetilde {Z''_{75}}=x_1^{15}x_2^{7}x_3^{7}x_4^{3}$, & \multicolumn{1}{l}{$Z''_{1}=x_1x_2^{2}x_3^{7}x_4^{7}x_5^{15}$,} \\
$Z''_{2}=x_1x_2^{2}x_3^{7}x_4^{15}x_5^{7}$, & $Z''_{3}=x_1x_2^{2}x_3^{15}x_4^{7}x_5^{7}$, & $Z''_{4}=x_1x_2^{3}x_3^{6}x_4^{7}x_5^{15}$, & \multicolumn{1}{l}{$Z''_{5}=x_1x_2^{3}x_3^{6}x_4^{15}x_5^{7}$,} \\
\end{tabular}
\end{center}

\newpage
\begin{center}
\begin{tabular}{lllr}
$Z''_{158}=x_1^{3}x_2^{15}x_3x_4^{6}x_5^{7}$, & $Z''_{159}=x_1^{3}x_2^{15}x_3x_4^{7}x_5^{6}$, & $Z''_{160}=x_1^{3}x_2^{15}x_3^{3}x_4^{4}x_5^{7}$, & \multicolumn{1}{l}{$Z''_{161}=x_1^{3}x_2^{15}x_3^{3}x_4^{5}x_5^{6}$,} \\
$Z''_{162}=x_1^{3}x_2^{15}x_3^{3}x_4^{7}x_5^{4}$, & $Z''_{163}=x_1^{3}x_2^{15}x_3^{5}x_4^{2}x_5^{7}$, & $Z''_{164}=x_1^{3}x_2^{15}x_3^{5}x_4^{3}x_5^{6}$, & \multicolumn{1}{l}{$Z''_{165}=x_1^{3}x_2^{15}x_3^{5}x_4^{6}x_5^{3}$,} \\
$Z''_{166}=x_1^{3}x_2^{15}x_3^{5}x_4^{7}x_5^{2}$, & $Z''_{167}=x_1^{3}x_2^{15}x_3^{7}x_4x_5^{6}$, & $Z''_{168}=x_1^{3}x_2^{15}x_3^{7}x_4^{3}x_5^{4}$, & \multicolumn{1}{l}{$Z''_{169}=x_1^{3}x_2^{15}x_3^{7}x_4^{5}x_5^{2}$,} \\
$Z''_{170}=x_1^{7}x_2x_3^{2}x_4^{7}x_5^{15}$, & $Z''_{171}=x_1^{7}x_2x_3^{2}x_4^{15}x_5^{7}$, & $Z''_{172}=x_1^{7}x_2x_3^{3}x_4^{6}x_5^{15}$, & \multicolumn{1}{l}{$Z''_{173}=x_1^{7}x_2x_3^{3}x_4^{7}x_5^{14}$,} \\
$Z''_{174}=x_1^{7}x_2x_3^{3}x_4^{14}x_5^{7}$, & $Z''_{175}=x_1^{7}x_2x_3^{3}x_4^{15}x_5^{6}$, & $Z''_{176}=x_1^{7}x_2x_3^{6}x_4^{3}x_5^{15}$, & \multicolumn{1}{l}{$Z''_{177}=x_1^{7}x_2x_3^{6}x_4^{7}x_5^{11}$,} \\
$Z''_{178}=x_1^{7}x_2x_3^{6}x_4^{11}x_5^{7}$, & $Z''_{179}=x_1^{7}x_2x_3^{6}x_4^{15}x_5^{3}$, & $Z''_{180}=x_1^{7}x_2x_3^{7}x_4^{2}x_5^{15}$, & \multicolumn{1}{l}{$Z''_{181}=x_1^{7}x_2x_3^{7}x_4^{3}x_5^{14}$,} \\
$Z''_{182}=x_1^{7}x_2x_3^{7}x_4^{6}x_5^{11}$, & $Z''_{183}=x_1^{7}x_2x_3^{7}x_4^{7}x_5^{10}$, & $Z''_{184}=x_1^{7}x_2x_3^{7}x_4^{10}x_5^{7}$, & \multicolumn{1}{l}{$Z''_{185}=x_1^{7}x_2x_3^{7}x_4^{11}x_5^{6}$,} \\
$Z''_{186}=x_1^{7}x_2x_3^{7}x_4^{14}x_5^{3}$, & $Z''_{187}=x_1^{7}x_2x_3^{7}x_4^{15}x_5^{2}$, & $Z''_{188}=x_1^{7}x_2x_3^{10}x_4^{7}x_5^{7}$, & \multicolumn{1}{l}{$Z''_{189}=x_1^{7}x_2x_3^{11}x_4^{6}x_5^{7}$,} \\
\end{tabular}%

\end{center}

\subsection{An algorithm in MAGMA}\label{s33}

In this subsection, to verify the dimension of the space $Q^{\otimes 5}$ in degree $N = 31$ and degree $N = 32,$ we provide an algorithm in MAGMA, which is established as follows:

\begin{lstlisting}[basicstyle={\fontsize{11pt}{9.5pt}\ttfamily},]

N := 31;

R<x1,x2,x3,x4,x5> := PolynomialRing(GF(2),5);
P<t> := PolynomialRing(R);
Sqhom := hom<R->P | [x + t*x^2 : x in [x1,x2,x3,x4,x5]]>;
function Sq(j,x) 
  c := Coefficients(Sqhom(x));
  if j+1 gt #c then
     return R!0;
  else
     return c[j+1];
  end if;
end function;
M := MonomialsOfDegree(R,N);
V := VectorSpace(GF(2),#M);
function MtoV(m)              
  if IsZero(m) then
    return V!0;
  else
    cv := [Index(M,mm) : mm in Terms(m)];
    return CharacteristicVector(V,cv);
  end if;
end function;
M1 := MonomialsOfDegree(R,N-1);
M2 := MonomialsOfDegree(R,N-2);
M4 := MonomialsOfDegree(R,N-4);
M8 := MonomialsOfDegree(R,N-8);
M16 := MonomialsOfDegree(R,N-16);
print [#m : m in [M1,M2,M4,M8, M16]];
\end{lstlisting}
>\ {\bf output: [46376,\ \ 40920,\ \ 31465,\ \ 17550,\ \ 3876]}

\begin{lstlisting}[basicstyle={\fontsize{11pt}{9.5pt}\ttfamily},]
S1 := [Sq(1,x) : x in M1];
print "S1";

H1 := sub<V | [MtoV(x) : x in S1]>;
print Dimension(H1);
\end{lstlisting}
>\ {\bf output: 24615} 

\begin{lstlisting}[basicstyle={\fontsize{11pt}{9.5pt}\ttfamily},]
S2 := [Sq(2,x) : x in M2];
print "S2";

H2 := sub<V | [MtoV(x) : x in S2]>;
print Dimension(H2);
\end{lstlisting}
>\ {\bf output: 28665} 

\begin{lstlisting}[basicstyle={\fontsize{11pt}{9.5pt}\ttfamily},]
H := H1 + H2;
print Dimension(H);
\end{lstlisting}
>\ {\bf output: 43334} 

\begin{lstlisting}[basicstyle={\fontsize{11pt}{9.5pt}\ttfamily},]
S4 := [Sq(4,x) : x in M4];
print "S4";

H4 := sub<V | [MtoV(x): x in S4]>;
print Dimension(H4);
\end{lstlisting}
>\ {\bf output: 26520} 

\begin{lstlisting}[basicstyle={\fontsize{11pt}{9.5pt}\ttfamily},]
H := H + H4;
print Dimension(H);
\end{lstlisting}
>\ {\bf output: 49530} 

\begin{lstlisting}[basicstyle={\fontsize{11pt}{9.5pt}\ttfamily},]
S8 := [Sq(8,x) : x in M8];
print "S8";

H8 := sub<V | [MtoV(x) : x in S8]>;
print Dimension(H8);
\end{lstlisting}
>\ {\bf output: 15900} 

\begin{lstlisting}[basicstyle={\fontsize{11pt}{9.5pt}\ttfamily},]
H := H + H8;
print Dimension(H);
\end{lstlisting}
>\ {\bf output: 51494} 

\begin{lstlisting}[basicstyle={\fontsize{11pt}{9.5pt}\ttfamily},]
S16 := [Sq(16,x) : x in M16];
print "S16";

H16 := sub<V | [MtoV(x) : x in S16]>;
print Dimension(H16);
\end{lstlisting}
>\ {\bf output: 0} 

\begin{lstlisting}[basicstyle={\fontsize{11pt}{9.5pt}\ttfamily},]
H := H + H16;
print Dimension(H);
\end{lstlisting}
>\ {\bf output: 51494} 

\begin{lstlisting}[basicstyle={\fontsize{11pt}{9.5pt}\ttfamily},]
print "\n",Dimension(V);
\end{lstlisting}
>\ {\bf output: 52360} 

\begin{lstlisting}[basicstyle={\fontsize{11pt}{9.5pt}\ttfamily},]
print Dimension(H);
\end{lstlisting}
>\ {\bf output: 51494} 

\begin{lstlisting}[basicstyle={\fontsize{11pt}{9.5pt}\ttfamily},]
print Dimension(V) - Dimension(H);
\end{lstlisting}

\medskip

The calculations indicated that $$ \dim Q^{\otimes 5}_{31} = \dim V - \dim H = 52360 - 51494 = 866.$$  Replacing $N = 31$ by $N = 32$ in the above algorithm, we obtain $$\dim Q^{\otimes 5}_{32} = \dim V - \dim H =  58905-57901 = 1004.$$

{\bf Remark.} To compute the hit monomials and the inadmissible monomials, we only need to consider the effects of the Steenrod squaring operations  $Sq^{j}: \mathbb P^{\otimes 5}_{N-j}\to \mathbb P^{\otimes 5}_{N}$ for $j = 2^{t},\, t\geq 0, \, 2j \leq N.$ This can be explained as follows:  Because the set of all these squares are a minimal generating set for $\mathscr A_2$ as an algebra over $\mathbb Z_2$. Further, as well known, for any $X\in \mathbb P_{N-j}^{\otimes 5}$, $Sq^{2^{t}}(X) = 0$ if $\deg(X) = N-j< 2^{t}.$ And therefore we need only to compute the action of the squares $Sq^{2^{t}}$ on $\mathbb P_{N}^{\otimes 5}$ such that $2^{t}\leq \deg(X)$. More importantly, by using Kameko's homomorphism and changing the number of variables in the polynomial ring "$R<x_1, x_2, \ldots, x_5>$" and degree "$N$" in the above algorithm, we can verify all the results of Peterson \cite{F.P1}, Kameko \cite{M.K} and Sum \cite{N.S1} on the dimension of the space \eqref{kgvt2} for $t\leq 4.$ 


\begin{thebibliography}{99}

\bibitem{Adams}
J.F. Adams, \textit{On the non-existence of elements of Hopf invariant one}, Ann. of Math. (2) \textbf{72} (1960),  no. 1, 20-104. \url{https://doi.org/10.2307/1970147}

\bibitem{Adem}
J. Adem, \textit{The iteration of the Steenrod squares in Algebraic Topology}, Proc. Natl. Acad. Sci. USA \textbf{38} (1952), no. 8, 20-726. \url{https://doi.org/10.1073/pnas.38.8.720}

\bibitem{J.B}
J.M. Boardman, 
\textit{Modular representations on the homology of power of real projective space}, in Algebraic Topology: Oaxtepec 1991, ed. M. C. Tangor;  in Contemp. Math. \textbf{146} (1993), 49-70. \url{http://dx.doi.org/10.1090/conm/146}

\bibitem{Browder}
W. Browder, \textit{The Kervaire invariant of framed manifolds and its generalization}, Ann. of Math. (2) \textbf{90} (1969), no. 1, 157-186. \url{https://doi.org/10.2307/1970686}

\bibitem{Bruner}
R.R. Bruner, \textit{The cohomology of the mod 2 Steenrod algebra: A computer calculation}, WSU Research Report 37 (1997).

\bibitem{Cartan}
H. Cartan, \textit{Sur l'iteration des operations de Steenrod}, Comment. Math. Helv. \textbf{29} (1955), 40-58. \url{https://doi.org/10.1007/BF02564270} 

\bibitem{Chen}
T.W. Chen, 
\textit{Determination of  $\mbox{Ext}^{5,*}_{\mathscr{A}}(\mathbb{Z}/2, \mathbb{Z}/2)$}, Topol. Appl. \textbf{158} (2011), no. 5, 660-689. \url{https://doi.org/10.1016/j.topol.2011.01.002}


\bibitem{C.H}
P.H. Ch\ohorn n and L.M. H\`a, \textit{On the May spectral sequence and the algebraic transfer II}, Topol. Appl. \textbf{178} (2014), no. 1-2, 372-383. \url{https://doi.org/10.1016/j.topol.2014.10.013}

\bibitem{Ha}
L.M. H\`a, \textit{Sub-Hopf algebras of the Steenrod algebra and the Singer transfer}, Geom. Topol. Publ. \textbf{11} (2007), 101-124.  \url{https://doi.org/10.2140/gtm.2007.11.81}

\bibitem{Hill}
M. A. Hill, M. J. Hopkins and D. C. Ravenel, \textit{On the non-existence of elements of kervaire invariant one}, Ann. of Math. (2) \textbf{184} (2016), no. 1, 1-262. \url{https://doi.org/10.4007/annals.2016.184.1.1}

\bibitem{H.P}
N.H.V. H\uhorn ng  and F.P. Peterson, 
\textit{$\mathcal A$-generators for the Dickson algebra}, Trans. Am. Math. Soc. \textbf{347} (1995), no. 12, 4687-4728. \url{https://doi.org/10.1090/S0002-9947-1995-1316852-X}

\bibitem{Hung}
N.H.V. H\uhorn ng, \textit{The cohomology of the Steenrod algebra and representations of the general linear groups}, Trans. Amer. Math. Soc. \textbf{357} (2005), no. 10, 4065-4089. \url{https://doi.org/10.1090/S0002-9947-05-03889-4}


\bibitem{M.K}
 M. Kameko,
\textit{Products of projective spaces as Steenrod modules}, PhD. thesis, The Johns Hopkins University, ProQuest LLC, Ann Arbor, MI, 1990, 29 pages.

\bibitem{L.T}
H.N. Ly and N.K. Tin, \textit{The admissible monomial basis for the polynomial algebra of five variables in degree fourteen}, South East Asian J. of Mathematics and Mathematical Sciences, \textbf{16} (2020), no. 3, 27-38.

\bibitem{Lin}
W.H. Lin, 
\textit{${\rm Ext}_{\mathcal A}^{4, *}(\mathbb Z/2, \mathbb Z/2)\mbox{ {\it and} } {\rm Ext}_{\mathcal A}^{5, *}(\mathbb Z/2, \mathbb Z/2)$}, Topol. Appl. \textbf{155} (2008), no. 5, 459-496.  \url{https://doi.org/10.1016/j.topol.2007.11.003} 

\bibitem{Minami}
N. Minami, \textit{The Adams spectral sequence and the triple transfer}, Amer. J. Math. \textbf{117} (1995), no. 4, 965-985. \url{https://doi.org/10.2307/2374955}

\bibitem{M.M}
M.F. Mothebe and L. Uys, \textit{Some relations between admissible monomials for the polynomial algebra}, Int. J. Math. Math. Sci., Article ID 235806, 2015, 7 pages.  \url{https://doi.org/10.1155/2015/235806}

\bibitem{T.N} 
T.N. Nam, 
\textit{$\mathcal {A}$-g\'en\'erateurs g\'en\'eriques pour l'alg\`ebre polynomiale}, Adv. Math. \textbf{186} (2004) no. 2,  334-362.  \url{https://doi.org/10.1016/j.aim.2003.08.004}

\bibitem{P.W}
D.J. Pengelley and F. Williams, 
\textit{A new action of the Kudo-Araki-May algebra on the dual of the symmetric algebras, with applications to the hit problem}, Algebr. Geom. Topol. \textbf{11} (2011), no. 3, 1767-1780.  \url{https://doi.org/10.2140/agt.2011.11.1767}

\bibitem{F.P1}
 F.P. Peterson,
\textit{Generators of  $H^*(\mathbb{R}P^{\infty}\times \mathbb{R}P^{\infty})$ as a module over the Steenrod algebra}, Abstracts Amer. Math. Soc., Providence, RI, April 1987. 

\bibitem{F.P2}
 F.P. Peterson,
\textit{$\mathcal A$-generators for certain polynomial algebras},  Math. Proc. Cambridge Philos. Soc. \textbf{105} (1989), no. 2,  311-312.  \url{https://doi.org/10.1017/S0305004100067803}

\bibitem{P.S1}
\DJ.V. Ph\'uc and N. Sum, 
\textit{On the generators of the polynomial algebra as a module over the Steenrod algebra}, C.R.Math. Acad. Sci. Paris \textbf{353} (2015),  no. 11, 1035-1040. \url{https://doi.org/10.1016/j.crma.2015.09.002}


\bibitem{D.P1}
\DJ.V. Ph\'uc, 
\textit{The hit problem for the polynomial algebra of five variables in degree seventeen and its application},  East-West J. Math. \textbf{18} (2016), no. 1, 27-46.

\bibitem{D.P2}
\DJ.V. Ph\'uc, 
\textit{The "hit" problem of five variables in the generic degree and its application}, Topol. Appl. \textbf{282} (2020), 107321, in press. \url{https://doi.org/10.1016/j.topol.2020.107321}

\bibitem{D.P3}
\DJ.V. Ph\'uc, 
\textit{$\mathcal A$-generators for the polynomial algebra of five variables in degree $5(2^t-1) + 6.2^t$},  Commun. Korean Math. Soc. \textbf{35} (2020), no. 2, 371-399. \url{https://doi.org/10.4134/CKMS.c190076}

\bibitem{D.P4}
\DJ.V. Ph\'uc, 
\textit{On Peterson's open problem and representations of the general linear groups},  J. Korean Math. Soc. \textbf{58} (2021), no. 3, 643-702. \url{https://doi.org/10.4134/JKMS.j200219}


\bibitem{D.P6}
\DJ.V. Ph\'uc, 
\textit{On the dimension of the "cohits" space $\mathbb Z_2\otimes_{\mathcal A_2} H^{*}((\mathbb RP(\infty))^{\times t}, \mathbb Z_2)$ and some applications},  Preprint (2021) (submitted for publication), arXiv:submit/3964919. \url{https://www.researchgate.net/publication/350430054}




\bibitem{D.P10}
\DJ.V. Ph\'uc, 
\textit{The answer to Singer's conjecture on the cohomological transfer of rank 4}, Preprint (2021) (submitted for publication), arXiv:submit/4033637. \url{https://www.researchgate.net/publication/352284459}

\bibitem{J.S}
J-P. Serre, \textit{Cohomologie modulo 2 des complexes d'Eilenberg-MacLane}, Comment. Math. Helv. \textbf{27} (1953), 198-232. \url{https://doi.org/10.1007/BF02564562}

\bibitem{W.S1}
W.M. Singer, 
\textit{The transfer in homological algebra}, Math. Z. \textbf{202} (1989), no. 4, 493-523.  \url{https://doi.org/10.1007/BF01221587}

\bibitem{W.S2}
W.M. Singer, \textit{Rings of symmetric functions as modules over the Steenrod algebra}, Algebr. Geom. Topol. \textbf{8} (2008), no. 1, 541-562. \url{https://doi.org/10.2140/agt.2008.8.541}

\bibitem{N.E}
 N.E. Steenrod,
\textit{Products of cocycles and extensions of mappings}, Ann. of Math. (2) \textbf{48} (1947), no. 2, 290-320. \url{https://doi.org/10.2307/1969172}

\bibitem{N.S}
N. Sum, 
\textit{On the Peterson hit problem of five variables and its applications to the fifth Singer transfer}, East-West J. Math. \textbf{16} (2014), no. 1, 47-62.

\bibitem{N.S1}
N. Sum, 
\textit{On the Peterson hit problem}, Adv. Math. \textbf{274} (2015), 432-489.  \url{https://doi.org/10.1016/j.aim.2015.01.010}

\bibitem{N.S2}
N. Sum, 
\textit{On a construction for the generators of the polynomial algebra as a module over the Steenrod algebra}, in Singh M., Song Y., Wu J. (eds), Algebraic Topology and Related Topics. Trends in Mathematics. Birkh\"auser/Springer, Singapore (2019), 265-286. \url{https://doi.org/10.1007/978-981-13-5742-8_14}.

\bibitem{N.S3}
N. Sum,
\textit{The squaring operation and the Singer algebraic transfer}, Vietnam J. Math. \textbf{49} (2021), no. 4, 1079-1096. \url{https://doi.org/10.1007/s10013-020-00423-1}

\bibitem{W.W}
G. Walker and R.M.W. Wood, 
\textit{Polynomials and the mod 2 Steenrod Algebra: Volume 1, The Peterson hit problem}, in London Math. Soc.  Lecture Note Ser., Cambridge Univ. Press, January 11, 2018.

\bibitem{R.W}
 R.M.W. Wood, 
\textit{Steenrod squares of polynomials and the Peterson conjecture}, Math. Proc. Cambriges Phil. Soc. \textbf{105} (1989), no. 2, 307-309. \url{https://doi.org/10.1017/S0305004100067797}

\end{thebibliography}
\end{document}